\documentclass[a4paper,leqno]{article} 
\pdfoutput=1
\usepackage[utf8]{inputenc}
\usepackage[T1]{fontenc}
\usepackage{enumitem,amsfonts,amsmath,amssymb,amsthm,stmaryrd}
\usepackage{lmodern} 
\usepackage{dsfont} 
\usepackage{fullpage} 
\usepackage{microtype} 
\usepackage{comment} 
\usepackage{graphicx}
\usepackage{mathrsfs}
\usepackage{tikz} 
\usetikzlibrary {patterns} 
\usepackage{template}
\usepackage{float} 
\usepackage{authblk}
\usepackage{mathrsfs}
\usetikzlibrary{patterns}

\DeclareMathOperator\sP{P}
\DeclareMathOperator\E{E}

\title{Asymptotic direction of a ballistic random walk in a two-dimensional random environment with nonuniform mixing}
\author{Julien \textsc{Allasia}\thanks{\texttt{allasia@math.univ-lyon1.fr}. \ Universite Claude Bernard Lyon 1, CNRS, Ecole Centrale de Lyon, INSA Lyon, Université Jean Monnet, ICJ UMR5208,
69622 Villeurbanne, France.}}
\date{}

\begin{document}
\maketitle

\begin{abstract}
In this paper, we study random walks evolving with a directional bias in a two-dimensional random environment with correlations that vanish polynomially. Using renormalization methods first employed for one-dimensional dynamic environments along with additional ideas specific to this new framework, we show that there exists an asymptotic direction for such a random walk. We also provide examples of classical models for which our results apply.
\end{abstract}

\section{Introduction}

Research on random walks in random environments (RWREs) has been active since the 1970s and has found its motivation in various applied fields. Typically, in a static framework, we allocate to each point in $\ZZ^d$ ($d\geqslant 1$) a probability measure on the set of its neighbors, given by a random process, to determine the law of the jump of a random walk located at this site. Contrary to classical random walks, results as simple as laws of large numbers (LLNs) are often hard to obtain, and strong assumptions describing the dependencies in the environment and the ballisticity of the random walk usually have to be made.

The one-dimensional i.i.d. case is well understood, and a LLN was shown in \cite{Solomon} using ergodicity arguments. In the larger dimension i.i.d. setup, a LLN can be proven under ballisticity conditions. For instance, \cite{Zerner1} used a drift assumption that implies large deviation results. In \cite{SZ} and \cite{Zerner2}, the authors introduced a seminal regeneration argument that gives a LLN under Kalikow's condition; see \cite{Kalikow}. In \cite{Sznitman}, a weaker ballisticity condition known as Sznitman's condition $(T)$ was introduced, which gives a LLN in the uniformly elliptic setting. In any cases, not even the i.i.d. framework is well understood without ballisticity assumptions, and our paper is no exception.


One can wonder if conditions on the dependencies of the environment weaker than the i.i.d. assumption would be sufficient to derive a LLN. In \cite{CZ}, the authors managed to adapt the regeneration argument from \cite{SZ} when the environment is assumed to satisfy some uniform mixing conditions. Recent progress has also been made for one-dimensional dynamic random environments (the environment is allowed to evolve in time). Seen as random walks in dimension $d+1$, where the $(d+1)^\text{th}$ direction is time, they are automatically ballistic in that direction. With this analogy in mind, similar mixing conditions to that of \cite{CZ} were used to derive a LLN; see for instance \cite{AHR}. Asymptotic results were also shown for some particular dynamic environments using their specific properties, like the contact process in \cite{HS} and \cite{MV}, or the environment given by independent simple random walks in \cite{HHSST}.

In \cite{BHT}, a LLN was shown for general one-dimensional dynamic environments satisfying a non-uniform polynomial mixing condition, using multi-scale renormalization methods. The proof fundamentally relies on a monotonicity property of the model (see (2.9) in \cite{BHT}), which is ensured by the dynamic one-dimensional framework and a nearest-neighbor assumption. Generalizing the methods of \cite{BHT} when this essential property is missing was already explored in \cite{Allasia_a} by lifting the nearest-neighbor assumption. In the present paper, we keep this assumption but we move from the dynamic one-dimensional framework to a static two-dimensional one.

More precisely, we assume that we are given a polynomially mixing random environment $\omega$ on $\ZZ^2$, where for each site $x$ in $\ZZ^2$, knowing $\omega(x)$ gives the transition probabilities for the jump of a random walk located at $x$ to one of the four nearest-neighbors of $x$. We consider the random walk starting at the origin in environment $\omega$. We denote its location at time $n$ by $Z_n=(X_n,Y_n)\in\ZZ^2$. In order to use the ideas of \cite{BHT}, we assume to have a strong bias upwards, in the sense that we have exponential quenched estimates for the speed of the random walk and the height of its first cut line (see Assumptions \ref{S:a:ballisticity} and \ref{S:a:cut_line}). This allows us to think of the vertical coordinate as roughly equivalent to time. Thus we will be able to show the existence of an asymptotic direction for our random walk, which is an almost sure limit of $Z_n/\vert Z_n\vert$ when $n$ goes to infinity, where $\vert \cdot \vert$ is the Euclidean norm on $\RR^2$; see Theorem \ref{S:t:direction}. This could be the first step towards showing a LLN in this framework, i.e. the almost sure convergence of $Z_n/n$.

The question of the existence of an asymptotic direction for RWREs has already been discussed in the i.i.d. setup in \cite{Simenhaus} and \cite{DR}. One important result is that if the random walk is transient in the neighborhood of a given direction, then an asymptotic direction can be found using renewal structures. But again these methods fail when we have weaker mixing assumptions for the environment.

Extending the ideas of \cite{BHT} does not merely consist in rewriting its arguments in a different framework. On top of the additional technical considerations about the second coordinate of the random walk (which is deterministic in the dynamic framework, for it is time), it requires finding a way to generalize the lost monotonicity property. More precisely, we need to guarantee that if a random walk starts on the left of another one, it will remain on its left forever, unless one goes around the other from below. This is made possible by choosing the right coupling for our random walks and proving a "barrier" property: see Proposition \ref{S:p:barrier_property} and Figure \ref{S:f:barrier_property} for an illustration of what can happen. Furthermore, the fact that our random walks can revisit the sites of the past calls for an argument to somehow split sample paths into different sections that do not meet, which is the reason behind Assumption \ref{S:a:cut_line} on cut lines.

\paragraph{Outline of the paper.}
In Section \ref{S:s:framework}, we define precisely the framework of this paper by defining static environments and random walks on them, before stating our main result, Theorem \ref{S:t:direction}. Its proof is divided into two parts, which correspond respectively to Sections \ref{S:ss:part1} and \ref{S:ss:part2}. In the first part, we show the existence of two limiting directions that bound the spatial behavior of our random walks in some sense. In the second part, we show that these two directions coincide, which will give the asymptotic direction that we are after. In Section \ref{S:ss:key_properties}, we introduce tools that will be instrumental in both parts of the proof. In Section \ref{S:s:complete_LLN}, we give some ideas and problems that we are facing to show a complete LLN. In Section \ref{S:s:applications}, we present some models for which our results apply.

\paragraph{Conventions.}
\begin{itemize}[leftmargin=*]
    \item $\N$, $\ZZ$ and $\R$ respectively denote the set of natural integers (starting from $0$), relative integers and real numbers. $\N^*$ denotes $\N\setminus \{0\}$, $\RR_+$ denotes $\{x\in\RR,\,x\geqslant 0\}$ and $\RR_+^*$ is $\RR_+\setminus\{0\}.$ If $n\leqslant m$ are two integers, $\llbracket n,m\rrbracket$ is the set of integers $[n,m]\cap\ZZ.$ For a couple $x=(a,b)\in\RR^2$, we write $a=\pi_1(x)$ and $b=\pi_2(x)$, and we refer to them as the horizontal and vertical coordinates of $x$. Mind that $x,y,z$ can all denote points in $\ZZ^2$, contrary to the notation for the random walk $Z_n=(X_n,Y_n)$ where $X_n\in\ZZ$ and $Y_n\in\ZZ$. We denote by $o$ the origin $(0,0)\in\RR^2$, and we set $\mathbf{0}$ to be the everywhere zero function of $\NN^{\ZZ^2}.$ If $S$ is a finite set, $|S|$ and $\# S$ denote the cardinality of $S$.
    \item $c$ denotes a positive constant that can change throughout the paper and even throughout one single computation. Constants that are used again later in the paper will be denoted with an index when they appear for the first time.
    \item Drawings across the paper are not to scale and they do not necessarily represent the random walks in an accurate way: they are only meant to make the reading easier. For instance, sample paths are depicted as smooth curves, although our random walks evolve on $\ZZ^2.$ 
\end{itemize}

\section{Framework}\label{S:s:framework}

\subsection{Environment}\label{S:ss:environment}

Let $S$ be a topological space. We let $\Omega=S^{\ZZ^{2}}$ and we endow it with the product topology and the subsequent $\sigma$-algebra, noted $\mathcal{T}$. Let $\sP$ be a probability measure on $(\Omega,\mathcal{T})$. On $(\Omega,\mathcal{T},\sP)$, the random variable $\mathrm{id}_{\Omega}$ is called a static two-dimensional random environment with law $\sP$. We denote it using the same letter $\omega$ by abuse of notation. Let $\mathbb{S}=\{(1,0),(-1,0),(0,1),(0,-1)\}$ be the Euclidean unit sphere in $\ZZ^2$. For each $s\in S$, we define a probability kernel $p(s,\cdot)$ on $\mathbb{S}$. A random walk on environment $\omega$ is a Markov process $Z=(Z_n)_{n\in\NN}$ whose law satisfies, for every $x\in\ZZ^2$ and $y\in\mathbb{S}$,
\begin{align}\label{d:RW}
\PP^\omega(Z_{n+1}=x+y\vert Z_n=x)=p\left(\omega(x),y\right).
\end{align}
We let $Z_n=(X_n,Y_n)$ for every $n\in\NN.$

\begin{remark}
Mind that static environments on $\ZZ^{2}$ are usually defined as random functions from $\ZZ^{2}$ to $\{(p_y)_{y\in\mathbb{S}}\in\RR_+^{4},\,\sum_y p_y=1\}$. For $x\in\ZZ^2$, we then write $\omega(x)=(\omega_{x,x+y},y\in\mathbb{S})$. Defined so, we can naturally define a random walk on $\omega$ so that for every $x\in\ZZ^2$, $\omega_{x,x+y}$ denotes the probability for the random walk located at $x$ to jump to $x+y$. The probability kernel $p$ here is simply given, for $x\in\ZZ^2$ and $y\in \mathbb{S}$, by $p(\omega(x),y)=\omega_{x,x+y}.$ Our definition therefore encompasses this one, and its asset is that the environment can be any process on $\ZZ^2$ (see Section \ref{S:s:applications} for examples), which we translate into transition probabilities using a transition kernel.
\end{remark}

If $\omega\in\Omega$ and $x_0\in\ZZ^2$, we define the translated environment $\theta^{x_0}\omega:x\in\ZZ^2\mapsto \omega(x_0+x).$ For the rest of the paper, we make the following assumption.

\begin{assumption}[Translation invariance]\label{S:a:translation_inv}
For every $x_0\in\ZZ^2$, $\theta^{x_0} \omega$ has the same law as $\omega$ under $\sP$.
\end{assumption}

Let us now introduce some definitions that will lead us to our essential mixing assumption. Let $B\subseteq\RR^{2}$. We say that $B$ is a box if it is of the form $[a_1,b_1)\times[a_2,b_2)$ with $a_1<b_1$ and $a_2<b_2$. If $B$ is a box, we define its spatial diameter $\mathrm{diam}(B)=b_1-a_1$ and height $h(B)=b_2-a_2$. If $B=[a_1,b_1)\times[a_2,b_2)$ and $B'=[a_1',b_1')\times[a_2',b_2')$ are two boxes, we call $\mathrm{sep}(B,B')$ their vertical separation defined by
\begin{align*}
 \mathrm{sep}(B,B')= \left\{\begin{array}{ll}
              a_{2}-b'_{2}& \text{ if } b'_2< a_2,\\
              a'_2-b_2&\text{ if }b_2<a'_2,\\
              0&\text{ else.}
             \end{array}\right.
\end{align*}

\begin{definition}\label{d:mixing}\ncS{S:c:mixing} Let $a,b,\alpha>0.$ We say that $\omega$ mixes polynomially in the vertical direction with size parameters $(a,b)$ and exponent $\alpha$ if there exists $\ucS{S:c:mixing}=\ucS{S:c:mixing}(a,b,\alpha)>0$ such that for every $h>0,$ for every pair of boxes $B,B'\subseteq \RR^2$ satisfying
$$\max(\mathrm{diam}(B),\mathrm{diam}(B'))\leqslant ah,\quad \max(h(B), h(B'))\leqslant bh,\quad \mathrm{sep}(B,B')\geqslant h,$$
for every pair of $\{0,1\}$-valued functions $f_1$ and $f_2$ on $\Omega$ such that $f_1(\omega)$ is $\sigma(\omega\vert_ {B\cap\ZZ^2})$-measurable and $f_2(\omega)$ is $\sigma(\omega\vert_ {B'\cap\ZZ^2})$-measurable,
    \begin{align*}
    \E[f_1(\omega)f_2(\omega)]\leqslant \E[f_1(\omega)]\E[f_2(\omega)]+ \ucS{S:c:mixing}\, h^{-\alpha}.
    \end{align*}
\end{definition}

We differ the statement of the assumption on the mixing of $\omega$ to later, for we need an assumption on the random walk to choose the size parameters properly.

\subsection{Random walk}\label{S:ss:random_walk}
We will work with random walks jumping at discrete times, but our results also hold in continuous time (in the Poissonian framework); see Remark \ref{S:r:continuous_time}. This is to be expected since our main result, namely the existence of an asymptotic direction, does not give information on the time behavior of the random walk. Actually in this model, time does not play a role in the coupling of random walks as important as in \cite{BHT}; see Section \ref{S:ss:coupling} for more details.

We now define the random walk we are interested in and state our main results. For the sake of clarity, we define it in a simplified intuitive way before introducing a complete construction and a coupling in Section \ref{S:ss:coupling}. For $\omega\in\Omega$ fixed, we say that $(Z_n)_{n\in\NN^*}$ is a random walk on $\omega$ if it is a Markov chain that jumps at time $n$ to $Z_n+z$ ($z\in\mathbb{S}$) with probability $p(\omega(Z_n),z)$.

The assumptions we make on our random walk are quenched information (that is, valid for almost every environment) on its upwards behavior. The first one is a ballisticity assumption that is pretty strong in the sense that it involves a quenched and exponential estimate. Recall that $Z_n=(X_n,Y_n)$ for every $n\in\NN.$

\ncS{S:c:ballisticity}
\begin{assumption}\label{S:a:ballisticity}
There exists $\ucS{S:c:ballisticity},\beta>0$ such that for $\sP$-almost every $\omega$ and for every $n\in\NN$,
$$\PP^\omega\left( Y_n\leqslant \beta^{-1} n\right)\leqslant \ucS{S:c:ballisticity}^{-1} e^{-\ucS{S:c:ballisticity} n}.$$
\end{assumption}

Note that this assumption is not optimal. We could replace the exponential decay with any superpolynomial decay (because our mixing assumption is polynomial), and the same holds for Assumption \ref{S:a:cut_line}. However, it is handy to work with an explicit quantity for the decay.

Now that some control on the location of the random walk has been introduced through parameter $\beta$, we will state our mixing assumption, whose size parameters are tailored to work with Assumption \ref{S:a:ballisticity}. 

\begin{assumption}\label{S:a:mixing_env} 
    We assume that $\omega$ mixes polynomially in the vertical direction with size parameters $(2(2\beta+1),3)$ and exponent $\alpha>12$.
\end{assumption}

Note that this assumption does not require a lower bound for the covariance. This is because the main use of this assumption is Fact \ref{S:p:mixing}, which we will state later on.

The second assumption on the random walk itself calls for some definitions. When trying to adapt the ideas of \cite{BHT}, the history that our random walk accumulates will raise issues. Therefore, it will be very useful to find a time after which our random walk does not revisit the sites visited in the past. In this sense, everything will be as if, considering the random walk after this time, its initial history is everywhere zero.

\begin{definition}\label{S:d:cut_line}
Let $a\in\NN$. Let $M=(M_n)_{n\in\NN}$ be a random walk in $\ZZ^2$. Denote by $T_a$ the hitting times by the process $\pi_2(M)-\pi_2(M_0)$ of $a$. We say that $\RR\times\{\pi_2(M_0)+a\}$ is a cut line for $M$ if $T_a<\infty$ almost surely and if for every $n>T_a,$ $\pi_2(M_n)\geqslant \pi_2(M_0)+a$. In other words, the sample path of $M$ can be split into two parts contained in the two half-planes delimited by the cut line. We set
\begin{align*}
&\Xi(M)=\inf\{a\in\NN,\,\RR\times\{\pi_2(M_0)+a\}\text{ is a cut line for }M\};\\
&T(M)=T_{\Xi(M)}.
\end{align*}
\end{definition}


\ncS{S:c:cut_line}
\begin{assumption}\label{S:a:cut_line}
We assume that there exists $\ucS{S:c:cut_line}>0$ such that for $\sP$-almost every $\omega$, for every $H\in\NN$,
$$\PP^\omega\left(\Xi(Z)\geqslant H\right)\leqslant \ucS{S:c:cut_line}^{-1}e^{-\ucS{S:c:cut_line}H}.$$
\end{assumption}

\begin{remark}
    It might be possible to derive this assumption from Assumptions \ref{S:a:ballisticity} and \ref{S:a:mixing_env} provided that we have some additional assumption of uniform ellipticity upwards, using the renormalization method presented in Section \ref{S:sss:renormalization_scheme}. This remains an open question that would be interesting to solve in order to replace Assumption \ref{S:a:cut_line}, which is complex, by a more classical uniform ellipticity one. Note also that we can show that both Assumptions \ref{S:a:ballisticity} and \ref{S:a:cut_line} hold if we assume that there exists $\zeta>0$ such that $\sP$-almost surely, for every $x\in\ZZ^2$, we have $p(\omega(x),(0,1))\geqslant 1/2+\zeta$. Indeed, in that case, we can couple $Z$ with a standard random walk $\hat{Z}$ on $\ZZ$ which jumps to the right with probability $1/2+\zeta$ and the left with probability $1/2-\zeta$, so that $\hat{Z}$ is a lower-bound for $Y=\pi_2(Z)$. Ballisticity for $\hat{Z}$ is a classical result and straightforwardly implies that of $Z$. As for Assumption \ref{S:a:cut_line}, we can construct a cut line for $Z$ using a cut point of $\hat{Z}$. In order to show that with good probability we do not wait too long to get a cut point for $\hat{Z}$, we show that every point is a cut point with uniform positive probability, and we use independent attempts to get such a point.
\end{remark}

\begin{notation}\label{S:n:A}
    We denote by $\mathcal{A}$ the almost sure set of environments $\omega$ that make both Assumptions \ref{S:a:ballisticity} and \ref{S:a:cut_line} hold.
\end{notation}

The goal of this paper is to show the existence of an asymptotic direction for $Z=(Z_n)_{n\in\NN}$. This is stated in the following theorem, where $\vert \cdot \vert$ denotes the Euclidean norm on $\RR^2$ and $\mathbb{S}^1$ the unit Euclidean sphere.

\ncS{S:c:concentration}
\begin{theorem}[Asymptotic direction]\label{S:t:direction}
There exists $\chi\in\mathbb{S}^1$ with $\pi_2(\chi)>0$ such that $$\PP\text{-almost surely,}\;\;\; \frac{Z_n}{\vert Z_n \vert} \xrightarrow[n\to\infty]{}\chi,$$ where $\frac{Z_n}{\vert Z_n\vert}$ is almost surely well-defined for $n$ large. Moreover we have a polynomial rate of convergence:
\begin{align*}
\forall \varepsilon>0,\;\exists\, \ucS{S:c:concentration}=\ucS{S:c:concentration}(\varepsilon)>0,\;\forall n\in\NN^*,\;\;\;\PP\left(\bigl\vert Z_n- \vert Z_n \vert \chi \bigr\vert\geqslant\varepsilon \vert Z_n \vert \right)\leqslant \ucS{S:c:concentration}\,n^{-\alpha/4}.\end{align*}
\end{theorem}

It is straightforward to check that this result is a consequence of the following result, which is less appealing, but whose formulation is closer to the methods used in \cite{BHT}.

\ncS{S:c:LLN}
\begin{theorem}\label{S:t:LLN_spatial}
There exists $\nu\in\RR$ such that 
$$\PP\text{-almost surely},\;\;\;\frac{ X_n}{ Y_n}\xrightarrow[n\to\infty]{} \nu,$$
where $\frac{ X_n}{ Y_n}$ is almost surely well-defined for $n$ large. Moreover we have a polynomial rate of convergence:
\begin{align*}
\forall \varepsilon>0,\;\exists\, \ucS{S:c:LLN}=\ucS{S:c:LLN}(\varepsilon)>0,\;\forall n\in\NN^*,\;\;\;\PP\left(\left\vert  X_n-\nu\, Y_n\right\vert\geqslant\varepsilon\, | Y_n| \right)\leqslant \ucS{S:c:LLN}\,n^{-\alpha/4}.\end{align*}
\end{theorem}

We will only focus on this theorem from now on. This is why, for the rest of the paper, what we (abusively) call the direction of a point $x\in\ZZ^2$ satisfying $\pi_2(x)\neq 0$ is simply $\pi_1(x)/\pi_2(x)$, while whenever $Y_n\neq 0$, $X_n/ Y_n$ will be called the direction of $Z$ at time $n$, and we will refer to $\nu$ as the limiting direction of $Z$. The link between $\nu$ and $\chi$ from Theorems \ref{S:t:direction} and \ref{S:t:LLN_spatial} is given by $$\chi=\frac{(\nu,1)}{\sqrt{\nu^2+1}}\;\;\;\text{and}\;\;\; \nu=\frac{\pi_1(\chi)}{\pi_2(\chi)}.$$

\begin{remark}\label{S:r:continuous_time}
    Theorem \ref{S:t:direction} also holds for the random walk $(Y_t)_{t\geqslant 0}$ in $\ZZ^2$, started at $o$, in the following continuous time framework. Instead of jumping at times in $\NN$, we set a Poisson process $(T_n)_{n\in\NN^*}$ of parameter $1$ in $\RR_+^*$ (independent of $\omega$) and we set $(Y_t)_{t\geqslant 0}$ to be the right-continuous process that jumps at each time given by this Poisson process; everything else is the same as in the discrete time framework. Then, $(Z_n=Y_{T_n})_{n\in\NN}$ (where $T_0=0$) satisfies Theorem \ref{S:t:direction}. From there we can check that $Y_t/|Y_t|$ converges to the same asymptotic direction as $Z_n/|Z_n|$ when $t$ goes to infinity.
\end{remark}

\subsection{Complete construction and coupling}\label{S:ss:coupling}

Inspired by \cite{BHT}, we want to define random walks starting from all possible starting points in $\ZZ^2$ and couple them in the following way: no matter its starting point, a random walk visiting a fixed site for the first time should jump to the same neighboring site. To define this properly, for each $s\in S$, we partition $[0,1]$ into $4$ intervals $(I_s^z)_{z\in\mathbb{S}}$ (recall the notations from Section \ref{S:ss:environment}) so that the length of $I_s^z$ is $p(s,z)$. For $u\in [0,1]$, we then set $$g(s,u)=\sum_{z\in\mathbb{S}} z\mathbf{1}_{I_s^z}(u).$$
This is a measurable function from $S\times [0,1]$ to $\mathbb{S}$ aimed at deciding the jump of a random walk if the state of the environment at its location is $s$, depending on a parameter $u$. Now, let $(U(x,i))_{x\in\ZZ^2,\,i\in\NN^*}$ be a family of independent uniform random variables in $[0,1]$, defined on a probability space $(\Omega',\mathcal{T}',\sP')$. The idea is that $U(x,i)$ will be the source of randomness used for the jump of a random walk visiting $x$ for the $i^{\text{th}}$ time. Let
\begin{align}\label{S:e:product_space}
\Omega_0=\Omega\times\Omega',\;\; \mathcal{T}_0=\mathcal{T}\otimes\mathcal{T}',\;\; \PP=\sP\otimes \sP'.
\end{align}

We call $\PP$ the annealed law. When $\omega\in\Omega$ is a fixed environment, $\PP^\omega=\delta_{\{\omega\}}\otimes \sP'$ is called the quenched law. Note that we have the relation $\PP(\cdot)=\int_\Omega \PP^{\omega}(\cdot)\,\dd \sP(\omega).$

In order to couple the random walks, we have to count the number of times that each random walk has visited each site. Therefore, for every starting point $y\in\ZZ^2$, we define simultaneously a random walk $Z^{y}$ and a counting process $N^{y}$, both as random variables on $(\Omega_0,\mathcal{T}_0)$, by the following:
\begin{equation}\label{S:e:jump_rule}
\begin{array}{l}
Z_0^{y}=y;\\
\forall x\in\ZZ^2,\;N_0^{y}(x)=0;\\
\forall n\in\NN,\;\left\{\begin{array}{l} \forall x\in\ZZ^2,\; N_{n+1}^{y}(x)=N_n^{y}(x)+\delta_{x,Z_{n}^{y}};\\
Z_{n+1}^{y}=Z_n^{y}+g\left(\omega(Z_n^{y}),\,U(Z_n^{y},N_{n+1}^{y}(Z_n^{y}))\right)\end{array}\right.
\end{array}
\end{equation}
where $\delta$ is the Kronecker symbol (note that by construction $N_{n+1}^y(Z_n^y)\geqslant 1$, so that the third equality makes sense). Let us rephrase what these formulas mean. If a random walk started at $y$ reaches $x$ for the first time at time $n$, then its jump at time $n$ (namely $Z^y_{n+1}-Z^y_n$) is determined by the state of the environment at $x$ and $U(x,1)$ (which does not depend on $n$ or $y$!). If it comes back to $x$ later in time, it will use $U(x,2)$ to choose where to jump, and so on.

Note that the process given by $Z=Z^o$ is indeed a RWRE in the sense of Section \ref{S:ss:random_walk}. Indeed, for every $s\in S$, $x\in\ZZ^2$ and $i\in\NN^*$, we have $$\sP'(g(s,U(x,i))=y)=\int_0^1 \mathbf{1}_{g(s,u)=y}\,\mathrm{d}u=\int_0^1 \mathbf{1}_{I_s^y}(u)\,\mathrm{d}u=p(s,y),$$
and the fact that $Z$ is a Markov chain under $\PP^\omega$ comes from the fact that our coupling ensures that the sequence of uniform variables used for the jumps of $Z$ is i.i.d. (for a detailed proof, see Proposition \ref{S:p:independence_uniforms}). From now on, when working with $y=o$, the superscript $y$ will be omitted.

We will use a more practical notation for the uniform variables that are read by the random walk.

\begin{notation}\label{S:n:uniform}
For $n\in\NN$, we set $U_n^{y}=U(Z_{n}^{y}, N_{n+1}^{y}(Z_{n}^{y})).$
\end{notation}
With this notation, the induction formula that defines our random walks in \ref{S:e:jump_rule} can be written in a more straightforward manner:
\begin{align}\label{S:e:jump_rule_straightforward}
\forall n\in\NN,\;\; Z_{n+1}^{y}=Z_n^{y}+g\left(\omega(Z_n^{y}),\,U_n^{y}\right).
\end{align}

In practice, we will use polynomial mixing to control the correlation of events that involve our random walks: we say that we are decoupling those events. This is actually not stronger than Assumption \ref{S:a:mixing_env}, because the uniform variables used for the jumps of our random walks are i.i.d., so two sets of uniform variables supported by disjoint boxes are independent. In practice, we will always use mixing to upper bound the probability of the intersection of two events of $\mathcal{T}_0$. We will say that an event $A\in\mathcal{T}_0$ is measurable with respect to a box $B\subseteq \RR^2$ if it is a measurable function of $\omega\vert_{B\cap\ZZ^2}$ and $\{U(x,i),\,i\in\NN^*,\,x\in B\cap\ZZ^2\}$.

\begin{fact}[Decoupling]\label{S:p:mixing}
Assume Assumption \ref{S:a:mixing_env} is satisfied. Let $h>0.$ Let $B,B'\subseteq\RR^2$ be two boxes satisfying
$$\max(\mathrm{diam}(B),\mathrm{diam}(B'))\leqslant 2(2\beta+1)h,\quad \max(h(B), h(B'))\leqslant 3h,\quad \mathrm{sep}(B,B')\geqslant h.$$
Let $A$ resp. $A'$ be events of $\mathcal{T}_0$ that are measurable with respect to $B$ resp. $B'$. We have \begin{align*}
\PP(A\cap A')\leqslant \PP(A)\PP(A')+ \ucS{S:c:mixing} h^{-\alpha}.
\end{align*}
\end{fact}

The proof of this fact is straightforward. See Figure \ref{S:f:mixing} for an illustration of this fact: events describing respectively the two sample paths drawn here can be decoupled using Fact \ref{S:p:mixing}.

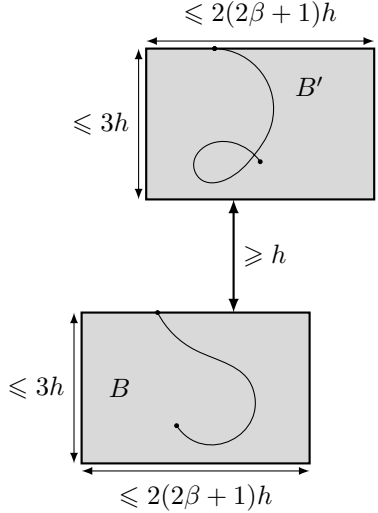
\begin{figure}[!h]
    \centering
    \begin{tikzpicture}[use Hobby shortcut,scale=0.5]
        \draw[fill, color=black!15!white] (-2.3,-1) rectangle (3.7,3);
        \draw[thick, right] (-2.3,-1) rectangle (3.7,3);
        \draw (2,2) node {$B'$};
        \draw[<->,thin,>=latex] (-2.5,-1) -- (-2.5,3);
        \draw[<->,thin,>=latex] (-2.3,3.2) -- (3.7,3.2);
        \draw[above] (0.7,3.3) node {$\leqslant 2(2\beta+1)h$};
        \draw[left] (-2.5,1) node {$\leqslant 3h$};
        \draw[fill, below left] (0.7,0) circle (.05);
        \draw (0.7,0) .. (0,0.5) .. (-1,-0.3) .. (0.4,0) .. (1,1) .. (-0.5,3);
        \draw[fill, above] (-0.5,3) circle (.05);
        \draw[fill, above] (-0.5,3) circle (.05);
        
        \draw[fill, color=black!15!white] (-4,-4) rectangle (2,-8);
        \draw[thick, below left] (-4,-4) rectangle (2,-8);
        \draw (-3,-6) node {$B$};
        \draw[fill, below left] (-1.5,-7) circle (.05);
        \draw (-1.5,-7).. (0.5,-6) .. (-1,-5) .. (-2,-4);
        \draw[fill] (-2,-4) circle (.05);
        \draw[<->,thin,>=latex] (-4,-8.2) -- (2,-8.2);
        \draw[<->,thin,>=latex] (-4.2,-8) -- (-4.2,-4);
        \draw[left] (-4.2,-6) node {$\leqslant 3h$};
        \draw[below] (-1,-8.3) node {$\leqslant 2(2\beta+1)h$};

        \draw[<->, thick,>=latex] (0,-4) -- (0,-1);
        \draw[right] (0,-2.5) node {$\geqslant h$};
    \end{tikzpicture}
    \caption{Illustration of the mixing property and Fact \ref{S:p:mixing}.}
    \label{S:f:mixing}
\end{figure}

\subsection{History}\label{S:ss:history}

Let $n_0\in\N^*$. Because of our coupling, the random walks given by $Z^{y}_{n_0+\cdot}$ and $Z^{Z_{n_0}^{y}}$ do not necessarily have the same sample paths. Indeed, the first one has a non-empty history, in the sense that between times $0$ and $n_0$, it has visited some positive number of sites and it has looked at $n_0$ random variables among the $\{U(x,i),\,x\in\ZZ^2,\,i\in\NN^*\}$, which it will not look at again in the future.

In order to address this issue, it will be convenient to define our random walks by adding an initial parameter alongside the starting point, which we will call the initial history of the random walk.

\begin{definition}\label{S:d:history}\hfill
\begin{itemize}[leftmargin=*]
    \item For each function $\Gamma:\ZZ^2\rightarrow \NN$, we define its support as $\mathrm{Supp}\,\Gamma=\{x\in\ZZ^2,\,\Gamma(x)>0\}.$ We let $$\mathcal{H}=\{\Gamma: \ZZ^2\rightarrow \NN \text{ such that $\mathrm{Supp}\,\Gamma$ is finite}\}.$$
    \item Let $y\in\ZZ^2$ and $n\in\NN$. The random variable $N_n^{y}$ defined in \eqref{S:e:jump_rule}, taking values in $\mathcal{H}$, is called the history of random walk $Z^{y}$ at time $n$.
    \item Let $y\in\ZZ^2$ and $\Gamma\in \mathcal{H}$. The random walk $Z^{y,\Gamma}$ starting at $y$ with history $\Gamma$ is defined in the same way as $Z^y$, except that in \eqref{S:e:jump_rule}, we replace $U(x,i)$ by $U(x,i+\Gamma(x))$. We also define a process $N^{y,\Gamma}$ in the same way, and we set $U_n^{y,\Gamma}$ as in Notation \ref{S:n:uniform}.
\end{itemize}
\end{definition}

Note that we could have restricted ourselves to an even smaller subset of $\NN^{\ZZ^2}$ for our set of histories. For instance, the support of a random walk's history has to be connected in $\ZZ^2$. Actually, we chose to define $\mathcal{H}$ as a simple countable subset of $\NN^{\ZZ^2}$, in order to sum over possible outcomes $\Gamma\in\mathcal{H}$ without worrying about uncountability.

Definition \ref{S:d:history} addresses the issue mentioned just before, for it ensures that for every $n_0\in\NN$, we have 
$$\forall n\in\NN,\;\;\;Z_{n_0+n}^{y}=Z_n^{Z_{n_0}^{y},N_{n_0}^{y}}.$$


Using Definition \ref{S:d:history}, we recover \eqref{S:e:jump_rule} by noticing that $Z_n^{y}=Z_n^{y,\mathbf{0}}.$ From now on, an omission of $\Gamma$ in any notation that is defined using a history superscript  $\Gamma$ will always mean that we are considering $\Gamma=\mathbf{0}$. Also, as mentioned before, the omission of the starting point superscript $y$ will mean that  $y=o$.

The rest of the paper is dedicated to showing Theorem \ref{S:t:LLN_spatial}. Its final proof using lemmas that will be shown later can be found at the end of Section \ref{S:ssss:global_lower_bound}.

\section{Key properties and tools}\label{S:ss:key_properties}
\subsection{2D simplification}\label{S:ss:2D_simplification}
As explained in the introduction, the idea of our proof is to adapt arguments used for one-dimensional dynamic environments in \cite{BHT}. The idea is to treat the vertical coordinate roughly as an equivalent of time. We will forget about the actual time variable and "hide" the time information by only considering hitting times of horizontal lines. In other words, we reduce the problem into a two-dimensional problem instead of a three-dimensional one.

\begin{definition}\label{S:d:tau}
Let $H\in\NN$, $y\in \ZZ^2$, $w\in\RR\times\ZZ$ and $\Gamma\in\mathcal{H}$.
The hitting time of height $\pi_2(w)+H$ by $Z^{y,\Gamma}$  is defined by
\begin{align*}
    \tau_{H,w}^{y,\Gamma}=\left\{\begin{array}{ll}\inf\{n\in\NN,\;Y_n^{y,\Gamma}=\pi_2(w)+H\}&\text{if $\pi_2(y)\leqslant \pi_2(w)+H$;}\\ 0&\text{otherwise,}\end{array}\right.
\end{align*}
where the infimum is in $\NN\cup\{+\infty\}.$ 
\end{definition}

In $\tau_{H,w}^{y,\Gamma}$, $w$ is a reference point whose horizontal coordinate does not play any role. It will be very useful in the future, because we will want to stop our random walks on a lattice centered at $w$, and we will have $\pi_2(y)$ slightly larger than $\pi_2(w)$ (see Definition \ref{S:d:box} and the proof of Lemma \ref{S:p:estimates_on_p}). Note that when $y=w$, $\tau_{H,y}^{y,\Gamma}$ is simply the time that $Z^{y,\Gamma}$ needs to go up $H$ times.

Notations can get really heavy. In an effort to make them more pleasant, we introduce an essential notation for the rest of the paper.
\begin{notation}
In the following notation, $A$ can be anything such that we defined $A^{y,\Gamma}$ with $y\in\ZZ^2$ and $\Gamma\in\mathcal{H}$. It can be an event (see Proposition \ref{S:p:independence}), a random variable (for example the times from Definition \ref{S:d:tau}, but also $Z$ itself, or $N$ as defined in Definition \ref{S:d:history}, or $V$ defined in Definition \ref{S:d:limiting_speeds}). For $H\in\NN$ and $w\in\RR\times\ZZ$, We set
\begin{align}\label{S:n:ThetaA}
\Theta_H^w A^{y,\Gamma}=A^{\mathcal{Z},\mathcal{N}},\text { where }\mathcal{Z}=Z^{y,\Gamma}_{\tau_{H,w}^{y,\Gamma}}\text{ and }\mathcal{N}=N^{y,\Gamma}_{\tau_{H,w}^{y,\Gamma}}.
\end{align} In particular, we have $\Theta_H^w Z^{y,\Gamma}_n=Z^{y,\Gamma}_{\tau^{y,\Gamma}_{H,w}+n}$ for every $n\in\NN$.
\end{notation}
We also adopt the following conventions:
\begin{itemize}[leftmargin=*,itemsep=0pt]
    \item Consistently with previous conventions, $\tau_H^{y,\Gamma}$ will mean $\tau_{H,o}^{y,\Gamma}$ (and not $\tau_{H,y}^{y,\Gamma}$!). $\tau_H$ will simply be $\tau^{o,\mathbf{0}}_{H,o}.$ In \eqref{S:n:ThetaA}, an omission of $w$ also means $w=o.$
    \item We write $Z^{y,\Gamma}_{\tau_{H,w}^{y,\Gamma}}$ without specifying what $Z^{y,\Gamma}_{\infty}$ means - any arbitrary value would work, since $\tau_{H,w}^{y,\Gamma}<\infty$ almost surely, using Assumption \ref{S:a:ballisticity}.
    \item We write $Z^{y,\Gamma}_{\tau_{H,w}}$ instead of $Z^{y,\Gamma}_{\tau_{H,w}^{y,\Gamma}}$ and $\Theta_H^w Z^{y,\Gamma}_{\tau_{H,w}}$ instead of $\Theta_H^w Z^{y,\Gamma}_{\tau_{H,w}^{y,\Gamma}}$. Mind that in these special cases, the omission of $y$ and $\Gamma$ does not mean $y=o$ and $\Gamma=\mathbf{0}$.
\end{itemize}

Note that Assumption \ref{S:a:ballisticity} can be translated in a simple way using $\tau_{H}$, as is stated in the following fact. Recall notation $\beta$ from the statement of Assumption \ref{S:a:ballisticity} and $\mathcal{A}$ from Notation \ref{S:n:A}.

\ncS{S:c:ball}
\begin{fact}\label{S:f:ball}
    There exists $\ucS{S:c:ball}>0$ such that for every $\omega\in\mathcal{A}$ and $H\in\NN$, we have
    $$\PP^\omega\left(\tau_{H}\geqslant \beta H\right)\leqslant \ucS{S:c:ball}^{-1} e^{-\ucS{S:c:ball} H}.$$
    The same estimate holds when replacing $\PP^\omega$ by $\PP$.
\end{fact}

\begin{proof}
    For every $\omega\in\mathcal{A}$ and $H\in\NN$, we have
    \begin{align*}
        \PP^\omega(\tau_H\geqslant \beta H)
        &\leqslant \PP^\omega(Y_{\lceil\beta H\rceil}\leqslant H)\\
    &\leqslant\PP^\omega(Y_{\lceil\beta H\rceil}\leqslant \beta^{-1} \lceil\beta H\rceil)\\
    &\leqslant \ucS{S:c:ballisticity}^{-1}e^{-\ucS{S:c:ballisticity} \lceil\beta H\rceil}.
    \end{align*}
    The result follows by choosing $\ucS{S:c:ball}$ properly (depending on $\ucS{S:c:ballisticity}$ and $\beta$) and integrating the quenched estimate to get the annealed one.
\end{proof}

In order to show Theorem \ref{S:t:LLN_spatial}, it will actually be sufficient to show an almost sure asymptotic estimate for $Z$ along the subsequence given by $(\tau_H)_{H\in\NN}$. This is what the following lemma is about.

\ncS{S:c:deviation}
\begin{lemma}\label{S:l:LLN_spatial_subsequence}
There exists $\nu\in\RR$ such that 
$$\PP\text{-almost surely},\;\;\;\frac{X_{\tau_H}}{H}\xrightarrow[H\to\infty]{} \nu.$$
Moreover we have a polynomial rate of convergence:
\begin{align*}
\forall \varepsilon>0,\;\exists\, \ucS{S:c:LLN}=\ucS{S:c:LLN}(\varepsilon)>0,\;\forall H\in\NN^*,\;\;\;\PP\left(\left\vert \frac{X_{\tau_H}}{H}-\nu\right\vert\geqslant\varepsilon\right)\leqslant \ucS{S:c:deviation}\,H^{-\alpha/4}.\end{align*}
\end{lemma}

The proof of Lemma \ref{S:l:LLN_spatial_subsequence} is the purpose of Sections \ref{S:ss:part1} and \ref{S:ss:part2}. The fact that Lemma \ref{S:l:LLN_spatial_subsequence} implies Theorem \ref{S:t:LLN_spatial} is shown at the end of Section \ref{S:ssss:global_lower_bound}.

\subsection{Markov-type properties}\label{S:ss:invariance_properties}

Our coupling makes the definition of our random walks more complex than they usually are. Yet, as we already said, a single random walk will behave just as in the usual framework. In particular, our random walks are Markov chains under the quenched law. This is actually a straightforward consequence of the following proposition. Recall Notation \ref{S:n:uniform}.

\begin{proposition}\label{S:p:independence_uniforms}
Let $y\in\ZZ^2$ and $\Gamma\in\mathcal{H}$. Under either $\PP$ or $\PP^\omega$, the $(U_n^{y,\Gamma})_{n\in \NN^*}$ are independent uniform random variables in $[0,1]$. 
\end{proposition}

\begin{proof}
    Let $y\in\ZZ^2$ and $\Gamma\in\mathcal{H}$. We want to show that for every $k\in\NN^*$ and $f_1,\ldots,f_k$ measurable non-negative functions on $[0,1]$, we have
\begin{align}\label{S:e:induction_form}
\EE\left[f_1(U_1^{y,\Gamma})\cdots f_k(U_k^{y,\Gamma})\right]=\int_0^1 f_1(u)\,\dd u\;\cdots\;\int_0^1 f_k(u)\,\dd u.
\end{align}
We show this by induction on $k$. The case $k=1$ simply follows from the fact that $U_1^{y,\Gamma}=U(y,\Gamma(y)+1)$. Assume \eqref{S:e:induction_form} is true for a fixed $k\in\NN^*$. Let $f_1,\ldots,f_{k+1}$ be measurable non-negative functions on $[0,1]$. Set $n_0=1$ and $x_0=y$. We have
\begin{align*}
    &\EE\left[f_1(U_1^{y,\Gamma})\cdots f_{k+1}(U_{k+1}^{y,\Gamma})\right]\\
    &=\sum_{x_1,\ldots,x_{k}\in\ZZ^2 \atop n_1,\ldots,n_{k}\in \NN^*} \EE\left[f_1(U(x_0,\Gamma(x_0)+n_0))\cdots f_{k+1}(U(x_k,\Gamma(x_k)+n_{k}))\,\prod_{j=1}^{k}\mathbf{1}_{Z^{y,\Gamma}_j=x_j}\,\mathbf{1}_{N_{j+1}^{y,\Gamma}(x_j)=n_j} \right].\nonumber
\end{align*}
Now in each term of this sum, the variable $f_{k+1}(U(x_k,\Gamma(x_k)+n_{k}))$ is independent from all the other variables that appear, for those are all measurable with respect to $\omega$ and $$\{U(x_0,\Gamma(x_0)+n_0),\ldots,U(x_{k-1},\Gamma(x_{k-1})+n_{k-1})\},$$
where, for every $j\in \llbracket 0,k-1\rrbracket$, either $x_k\neq x_j$ or $\Gamma(x_k)+n_k\neq \Gamma(x_j)+n_j.$
Now for any $x_k\in\ZZ^2$ and $n_k\in\NN^*$, $\EE[f_{k+1}(U(x_k,\Gamma(x_k)+n_k))]=\int_0^1 f_{k+1}(u)\,\dd u.$ Therefore
\begin{align*}
    \EE\left[f_1(U_1^{y,\Gamma})\cdots f_{k+1}(U_{k+1}^{y,\Gamma})\right]
    =\EE\left[f_1(U_1^{y,\Gamma})\cdots f_{k}(U_{k}^{y,\Gamma})\right]\;\int_0^1 f_{k+1}(u)\,\dd u,
\end{align*}
and then using the induction assumption allows us to conclude. The exact same arguments work when replacing $\PP$ by $\PP^\omega.$
\end{proof}

Recall the definition of the translated environment $\theta^y\omega$ from the beginning of Section \ref{S:ss:environment}.
\begin{corollary}\label{S:c:invariance}
Let $y\in\ZZ^2$ and $\Gamma\in\mathcal{H}$.
\begin{enumerate}
    \item The law under $\PP$ of $Z^{y,\Gamma}-y$ does not depend on $y$ and $\Gamma$, in particular it is the law of $Z$.
    \item For every environment $\omega$, the law under $\PP^\omega$ of $Z^{y,\Gamma}-y$ is the law under $\theta^y\omega$ of $Z$.
\end{enumerate}
\end{corollary}

\begin{proof}
This is a consequence of Assumption \ref{S:a:translation_inv}, \eqref{S:e:jump_rule_straightforward} and Proposition~\ref{S:p:independence_uniforms}.
\end{proof}

Note that for every $y\in\ZZ^2$, $\omega$ is in $\mathcal{A}$ if and only if $\theta^y \omega$ is in $\mathcal{A}$. Therefore, using the second point of the corollary, in order to show a quenched estimate on $Z^{y,\Gamma}-y$ valid for any $y\in\ZZ^2$, $\Gamma\in\mathcal{H}$ and $\omega\in\mathcal{A}$, it is sufficient to do it for $Z$. We will use this extensively in the rest of the paper. As for the first point, it expresses that averaging on all the possibles outcomes of the environment (that is, working under the annealed law) restores the translation-invariance of the law of the random walks.

Oftentimes, we will have to bound the probability of an event describing a random walk whose initial conditions (starting point and history) are random variables. To do this, we will need Markov-type properties.

As we said before, for any $y\in\ZZ^2$ and $\Gamma\in\mathcal{H}$, $Z^{y,\Gamma}$ is a Markov chain under the quenched law. Nonetheless, in general this Markov chain is obviously not time-homogeneous, since its transition matrices depend on the location of the random walk at each step. As a result, if $t\in\NN$, $\{Z_n,\,n\leqslant t\}$ is not necessarily independent from $\{Z_n-Z_t,\,n>t\}$ (neither under $\PP$ nor under $\PP^\omega$). This is an obstacle to studying the probability of an event describing a random walk whose initial conditions are given by its past. Nonetheless, we do have the following proposition, which will be very useful in the future. Recall \eqref{S:n:ThetaA} and Notation \ref{S:n:A}.

\begin{proposition}\label{S:p:independence}
Let $y_0\in\ZZ^2$, $\Gamma_0\in\mathcal{H}$, $H\in\NN$ and $w\in\RR\times\ZZ$. For $y\in\ZZ^2$ and $\Gamma\in\mathcal{H}$, let $A^{y,\Gamma}$ be an event that is measurable with respect to $\omega$ and $(U_n^{y,\Gamma})_{n\in\NN^*}$. Then, for every $\omega\in\mathcal{A}$,
\begin{align}\label{S:markov_type1}
&\PP^\omega\left(\Theta_H^w A^{y_0,\Gamma_0}\right)\leqslant \underset{y,\Gamma}{\sup}\;\PP^\omega(A^{y,\Gamma});\\\label{S:markov_type}
&\PP\left(\Theta_H^w A^{y_0,\Gamma_0}\right)\leqslant \underset{\omega\in\mathcal{A}}{\sup}\; \underset{y,\Gamma}{\sup}\;\PP^\omega(A^{y,\Gamma}).
\end{align}
In particular, if we have an upper bound of $\PP^\omega(A^{y,\Gamma})$ that is uniform in $\omega\in\mathcal{A}$, $y\in\ZZ^2$ and $\Gamma\in\mathcal{H}$, it is also an upper bound of $\PP(\Theta_H^w A^{y_0,\Gamma_0}).$
\end{proposition}

Mind that \textit{a priori} we may not replace $\sup_{y,\Gamma}\PP^\omega(A^{y,\Gamma})$ by $\PP^\omega(A^{o,\mathbf{0}})$. Indeed, although in the quenched setting $A^{y,\Gamma}$ is a measurable function of $(U_n^{y,\Gamma})_n$, whose law does not depend on $(y,\Gamma)$, the function itself may depend on $y$ and $\Gamma.$ However we will usually not use Proposition \ref{S:p:independence} in that case and so we will usually simply use a supremum over $\omega$: see for instance the proof of \eqref{S:e:equation_new}.

Note that the right hand-side in \eqref{S:markov_type} involves the quenched law $\PP^\omega$. This is why Assumptions \ref{S:a:ballisticity} and \ref{S:a:cut_line} have to be quenched and not annealed estimates, and why in the rest of this section, we will state mainly quenched estimates (see Propositions \ref{S:p:probability_of_going_down}, \ref{S:p:horizontal_bounds} and \ref{S:p:exit_box}).

\begin{proof}
For the sake of simplicity, we write the proof for $y_0=o$ and $\Gamma_0=\mathbf{0}$, and we assume that $w=o$ (up to changing $H$). Let us first show \eqref{S:markov_type1} for a fixed environment $\omega\in\mathcal{A}$. Let $y\in\ZZ^2$, $\Gamma\in\mathcal{H}$ and $H\in\NN$. First, note that the $(U_n^{y,\Gamma})_{n\in\NN^*}$ are measurable with respect to $\{U(x,i),\,x\in\ZZ^2,\,i>\Gamma(x)\}$. Furthermore, we claim that for every $t\in\NN$, $\{Z_t=y,\,N_t=\Gamma,\,\tau_H=t\}$ is measurable with respect to $\{U(x,i),\,x\in\ZZ^2,\,i\leqslant \Gamma(x)\}$. Let us prove this claim. First note that for every $t\in\NN$, there exists some $\{0,1\}$-valued measurable function $f_t$ on $[0,1]^t$ such that $\{\tau_H=t\}=\{f_t(U_1,\ldots,U_t)=1\}.$ Therefore, we can write
\begin{align*}
&\{Z_t=y,\,N_t=\Gamma,\,\tau_H=t\}\\
&=\bigcup_{ o=y_0,\ldots,y_t=y \atop 1=n_0,\ldots,n_{t-1}=\Gamma(y_{t-1})} \begin{array}{c} \{Z_1=y_1,\ldots, Z_t=y_t\}\cap \{N_1(y_0)=n_0,\ldots N_t(y_{t-1})=n_{t-1}\}\\
\cap\,\{N_t=\Gamma\}\cap\{f_t(U(y_0,n_0),\ldots,U(y_{t-1},n_{t-1}))=1\}.\end{array}
\end{align*}
In the union above, all the choices of $y_j$ and $n_j$ (where $0\leqslant j<t$) such that $n_j>\Gamma(y_j)$ give an empty contribution. Indeed, each choice of $y_j$ and $n_j$ corresponds to an event that is included in $\{N_{j+1}(y_j)=n_j,\,N_t=\Gamma\}$; therefore, if $n_j>\Gamma(y_j)$, then $N_{j+1}(y_j)>N_t(y_j)$, which is impossible since $s\mapsto N_s(x)$ is non-decreasing for any $x\in\ZZ^2$. Considering that each event in the union above is measurable with respect to $\{U(y_0,n_0),\ldots,U(y_{t-1},n_{t-1})\}$, the claim is proven.

Consequently, by Proposition \ref{S:p:independence_uniforms}, $\{Z_t=y,\,N_t=\Gamma,\,\tau_H=t\}$ is $\PP^\omega$-independent from the $\sigma$-algebra generated by $(U_n^{y,\Gamma})_{n\in\NN^*}$, so it is $\PP^\omega$-independent of $A^{y,\Gamma}$. As a result, we have
\begin{align*}
    \PP^\omega(\Theta_H A)=\PP^\omega\left(A^{Z_{\tau_H},N_{\tau_H}}\right)
    &=\displaystyle\sum_{y\in\ZZ^2,\,\Gamma\in\mathcal{H}} \sum_{t\in\NN}\, \PP^{\omega}(A^{y,\Gamma},\,Z_t=y,\,N_t=\Gamma,\,\tau_H=t)\\
    &=\displaystyle\sum_{y\in\ZZ^2,\,\Gamma\in\mathcal{H}} \PP^{\omega}(A^{y,\Gamma})\,\PP^\omega(Z_{\tau_H}=y,\,N_{\tau_H}=\Gamma)&\mbox{by independence}\\
    &\leqslant \underset{y,\Gamma}{\sup}\;\displaystyle\PP^\omega(A^{y,\Gamma}),
\end{align*}
concluding the proof of \eqref{S:markov_type1}. The annealed estimate \eqref{S:markov_type} is obtained by integrating over $\mathcal{A}$ and bounding the integral by the supremum of the function integrated.
\end{proof}

\subsection{Localization properties}\label{S:sss:localization_properties}

\subsubsection{Vertical lower bound}\label{S:ssss:global_lower_bound}

We define an event guaranteeing that $Z^{y,\Gamma}$ stays above some horizontal line by setting, for $H\in\NN$,
    \begin{align}\label{S:n:event_E}
    E_H^{y,\Gamma}=\left\{\forall n\in\NN,\;Y_n^{y,\Gamma}\geqslant\pi_2(y)-H\right\}.
    \end{align}

\begin{proposition}\label{S:p:probability_of_going_down}
\ncS{S:c:proba_of_going_down} There exists $\ucS{S:c:proba_of_going_down}>0$ such that for every $\omega\in\mathcal{A}$ and $H\in\NN$,
\begin{align*}
\PP^\omega\left((E_H^{y,\Gamma})^c\right)\leqslant \ucS{S:c:proba_of_going_down}^{-1} e^{-\ucS{S:c:proba_of_going_down}H}.
\end{align*}
This inequality is also true when replacing $\PP^\omega$ by $\PP$.
\end{proposition}

\begin{proof}
    Using Corollary \ref{S:c:invariance}, we may assume that $y=0$ and $\Gamma=\mathbf{0}$. Let $\omega\in\mathcal{A}$ and $H\in\NN$. For every $n\in\NN$, we have
    $$\PP^\omega( Y_n\leqslant -H)\leqslant \PP^\omega( Y_n\leqslant \beta^{-1} n)\leqslant \ucS{S:c:ballisticity}^{-1} e^{-\ucS{S:c:ballisticity} n},$$
    using Assumption \ref{S:a:ballisticity}. Therefore,
    \begin{align*}
        \PP^\omega(\exists\, n\in\NN,\, Y_n\leqslant -H)
        &=\PP^\omega(\exists\, n\geqslant H,\, Y_n\leqslant -H)\\
        &\leqslant \sum_{n\geqslant H} \ucS{S:c:ballisticity}^{-1} e^{-\ucS{S:c:ballisticity} n}\\
        &=\frac{\ucS{S:c:ballisticity}^{-1} e^{-\ucS{S:c:ballisticity} H}}{1-e^{-\ucS{S:c:ballisticity}}}.
    \end{align*}
    We get the result by setting $\ucS{S:c:proba_of_going_down}=\ucS{S:c:ballisticity}(1-e^{-\ucS{S:c:ballisticity}}).$ In order to get the annealed estimate, we integrate the quenched estimate.
\end{proof}

We will use Proposition \ref{S:p:probability_of_going_down} extensively in the rest of the paper. For now, it allows us to prove that Lemma \ref{S:l:LLN_spatial_subsequence} is sufficient to prove Theorem \ref{S:t:LLN_spatial}, using an argument of interpolation.

\begin{proof}[Proof of Theorem \ref{S:t:LLN_spatial} assuming Lemma \ref{S:l:LLN_spatial_subsequence}]\label{S:final_proof}
Let $\nu$ be as in Lemma \ref{S:l:LLN_spatial_subsequence}. Let $n\in\NN^*$ be such that $ Y_n>0$ (which, by Assumption \ref{S:a:ballisticity}, happens for $n$ large enough $\PP$-almost surely). Let $H_n\in\NN$ be such that $\tau_{H_n}\leqslant n<\tau_{H_n+1}.$ Note that, using Assumption \ref{S:a:ballisticity} again, \begin{align}\label{S:H_n_to_infty}
H_n\xrightarrow[n\to\infty]{a.s}+\infty.
\end{align}
Also, note that $Y_n\leqslant H_n$, so that, for $n$ large enough, we have
\begin{align}\label{S:e:bound_H_n}
    \PP(H_n\leqslant \beta^{-1} n)\leqslant \PP( Y_n\leqslant\beta^{-1} n)\leqslant \ucS{S:c:ballisticity}^{-1}\,e^{-\ucS{S:c:ballisticity} n}.
\end{align}
Now, note that
\begin{align}\label{S:e:three_terms}
\frac{ X_n}{ Y_n}=\frac{X_{\tau_{H_n}}}{H_n}+\frac{ X_n-X_{\tau_{H_n}}}{H_n}+ X_n\left(\frac{1}{ Y_n}-\frac{1}{H_n}\right).
\end{align}
First, using \eqref{S:H_n_to_infty} and Lemma \ref{S:l:LLN_spatial_subsequence}, we have
\begin{align}\label{S:e:1st_term}
    \frac{X_{\tau_{H_n}}}{H_n}\xrightarrow[n\to\infty]{a.s.} \nu.
\end{align}
As for the second term on the right-hand side of \eqref{S:e:three_terms}, let us fix $a>0$ and note that
\begin{align*}
    \PP\left(\vert  X_n-X_{\tau_{H_n}}\vert \geqslant a H_n\right)
    &\leqslant \PP(\tau_{H_n+1}-\tau_{H_n}\geqslant aH_n)\\
    &\leqslant \PP\left(\Theta_{H_n}\tau_{1,Z_{\tau_{H_n}}}\geqslant a\beta^{-1} n\right)+\PP(H_n\leqslant \beta^{-1} n).
\end{align*}
Now, if we fix $\omega\in\mathcal{A}$, for $n$ large enough we have, using Fact \ref{S:f:ball},
\begin{align*}
    \PP^\omega(\tau_1\geqslant a \beta^{-1} n)
    \leqslant \PP^\omega(\tau_{\lfloor a\beta^{-2}n\rfloor}\geqslant \beta \lfloor a\beta^{-2}n\rfloor)
    \leqslant c^{-1}e^{-c n}.
\end{align*}
Therefore, using Proposition \ref{S:p:independence} and \eqref{S:e:bound_H_n},
\begin{align}\label{S:e:equation_new}
    \PP\left(\vert  X_n-X_{\tau_{H_n}}\vert \geqslant a H_n\right)\leqslant c^{-1}e^{-c n},
\end{align}
which is summable in $n$. Using Borel-Cantelli, we obtain that
\begin{align}\label{S:e:2nd_term}
    \frac{ X_n-X_{\tau_{H_n}}}{H_n}\xrightarrow[n\to\infty]{a.s.}0.
\end{align}
In order to estimate the third term on the right-hand side of \eqref{S:e:three_terms}, first note that if $ Y_n\geqslant H_n-H_n^{1/2}$ and $H_n>\beta^{-1} n$, then we have, for $n$ large enough,
\begin{align*}
    \left\vert X_n\left(\frac{1}{ Y_n}-\frac{1}{H_n}\right)\right\vert
    =\left\vert\frac{ X_n}{ Y_n}\right\vert \frac{H_n- Y_n}{H_n}\leqslant \frac{n}{H_n/2}\frac{H_n^{1/2}}{H_n}=2n H_n^{-3/2}.
\end{align*}
In the first equality, we used that $H_n- Y_n\geqslant 0$, since $n<\tau_{H_n+1}$. In the inequality, we used that $\vert X_n\vert\leqslant n$, and that for $n$ large enough, $H_n>\beta^{-1} n\geqslant 4$, so $H_n-H_n^{1/2}\geqslant H_n/2$.
Therefore, if we fix $a>0$, we have, for $n$ large enough,
\begin{align*}
    \PP\left(\left\vert X_n\left(\frac{1}{ Y_n}-\frac{1}{H_n}\right)\right\vert\geqslant a\right)
    &\leqslant \PP( Y_n<H_n-H_n^{1/2})+\PP(H_n\leqslant (2n/a)^{2/3})+\PP(H_n\leqslant \beta^{-1} n)\\
    &\leqslant\PP\left(\left(\Theta_{H_n} E_{\lfloor (\beta^{-1} n)^{1/2}\rfloor}\right)^c\right)  +3\PP(H_n\leqslant \beta^{-1} n)\\
    &\leqslant c^{-1} e^{-c n^{1/2}}+3c^{-1}e^{-cn},
\end{align*}
using Propositions \ref{S:p:independence} and \ref{S:p:probability_of_going_down}, as well as \eqref{S:e:bound_H_n}. Applying the Borel-Cantelli lemma once more, we obtain that
\begin{align}\label{S:e:3rd_term}
     X_n\left(\frac{1}{ Y_n}-\frac{1}{H_n}\right)\xrightarrow[n\to\infty]{a.s.}0.
\end{align}
Putting together \eqref{S:e:three_terms}, \eqref{S:e:1st_term}, \eqref{S:e:2nd_term} and \eqref{S:e:3rd_term}, we obtain that $$\frac{ X_n}{ Y_n}\xrightarrow[n\to\infty]{a.s.}\nu,$$
concluding the proof of Theorem \ref{S:t:LLN_spatial}.
\end{proof}

\subsubsection{Horizontal bounds}\label{S:sss:horizontal_bounds}
It will also be essential to control the horizontal behavior of the random walk. The lack of strong information on the horizontal behavior of the random walks prevents us from getting a global horizontal bound as in Section \ref{S:ssss:global_lower_bound}. However, what we can do using Assumption \ref{S:a:ballisticity} is bound the horizontal displacement of the random walk by the time it reaches a certain height. To that end, we define the following event, for $y\in\ZZ^2$, $\Gamma\in\mathcal{H}$ and $H\in\NN^*$ (recall $\beta$ from Assumption \ref{S:a:ballisticity}):
\begin{align}\label{S:d:event_D}
    D^{y,\Gamma}_H=\left\{\forall n\in \left\llbracket 0,\tau_{H,y}^{y,\Gamma}\right\rrbracket,\;\left\vert X_n^{y,\Gamma}-\pi_1(y)\right\vert\leqslant \beta H\right\}.
\end{align}

\begin{proposition}\label{S:p:horizontal_bounds}
There exists $\ucS{S:c:ball}>0$ such that for every $\omega\in\mathcal{A}$, $y\in\ZZ^2$, $\Gamma\in\mathcal{H}$ and $H\in\NN^*$, 
\begin{align}\label{S:e:estimate_horizontal}
    \PP^\omega\left((D^{y,\Gamma}_H)^c\right)\leqslant \ucS{S:c:ball}^{-1}e^{-\ucS{S:c:ball} H},
\end{align}
and the same estimate holds when replacing $\PP^\omega$ by $\PP.$
\end{proposition}

\begin{proof}
    First note that using Corollary \ref{S:c:invariance}, we may assume that $y=o$ and $\Gamma=\mathbf{0}$. Now, observe that for every $\omega\in\mathcal{A}$, $\PP^\omega(D_H^c)\leqslant \PP^\omega(\tau_H\geqslant \beta H)$. Then we simply use Fact \ref{S:f:ball}.
\end{proof}

\subsubsection{Localization in boxes}\label{S:ssss:loc_in_boxes}

In order to apply Fact \ref{S:p:mixing}, we will have to localize events in boxes. In practice, this will be done by working on events of high probability that ensure that our random walks stay in some boxes before reaching a certain height, or, in other words, that they exit those boxes through the top side. This simply requires to put together the results of Sections \ref{S:ssss:global_lower_bound} and \ref{S:sss:horizontal_bounds}. However, we actually want something stronger: we want to control the behavior of a lot of random walks simultaneously. This will be instrumental for Section \ref{S:ss:part2}.


We will often use the following notation, for $H\in\RR_+^*$,
\begin{align}\label{S:n:H'}
H'=\left\lceil H^{1/2}\right\rceil.
\end{align}
We will also use this notation with specific values of $H$: for instance in the future we will write $H_k'$ for $\left\lceil H_k^{1/2}\right\rceil$ or $(hL_k)'$ for $\left\lceil (hL_k)^{1/2}\right\rceil.$

\begin{definition}\label{S:d:box}
Let $H\in\RR_+^*$ and $w\in\RR\times\ZZ$. We define
\begin{align}\label{S:d:interval_box}
\left\{\begin{array}{l}
    {\color{black}I_H(w)=(w+[0,H)\times [0,H'))\cap\ZZ^2;}\vspace{2pt}\\
    \mathcal{I}_H(w)=(w+[0,H)\times \{0\})\cap \ZZ^2;\\
    B_H(w)=w+[-\beta H, (\beta+1) H)\times [-H',H]\subseteq \RR^2.
\end{array}\right.\end{align}
We also define the following events, for $H\in\NN^*$ and $w\in\RR\times\ZZ$:
\begin{align}\label{S:d:event_F}
F_H(w)=\bigcap_{y\in I_H(w)}\left\{\forall n\in\llbracket 0,\tau_{H,w}^y\rrbracket,\, Z^y_n\in B_H(w)\right\}.
\end{align}
As usual, $I_H=I_H(o)$, $B_H=B_H(o)$ and $F_H=F_H(o)$.
\end{definition}

See Figure \ref{S:f:boxes} for an illustration of those definitions.

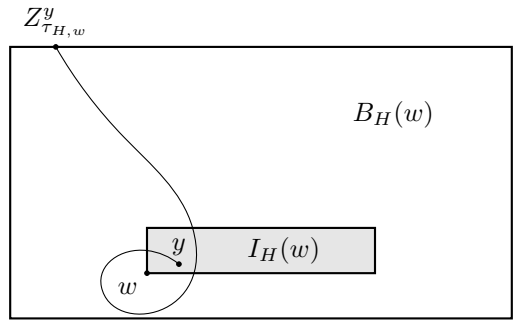
\begin{figure}[!h]
    \centering
    \begin{tikzpicture}[use Hobby shortcut, scale=0.6]
        \draw[fill, black!10!white] (0,0) rectangle (5,1);
        \draw[thick] (0,0) rectangle (5,1);
        \draw[above] (3,0) node {$I_H(w)$};
        \draw[fill,below left] (0,0) circle (.05) node {$w$};
        \draw[fill, above] (0.7,0.2) circle (.05) node {$y$};
        \draw (0.7,0.2) .. (0,0.5) .. (-1,-0.3) .. (1,1) .. (-0.5,3) .. (-2,5);
        \draw[fill, above] (-2,5) circle (.05) node {$Z^{y}_{\tau_{H,w}}$};
        \draw[thick, below left] (-3,5) rectangle (8,-1);
        \draw[below left] (6.5,4) node {$B_H(w)$};
    \end{tikzpicture}
    \caption{Illustration of Definition \ref{S:d:box}. On $F_H(w)$, random walk $Z^y$ has to exit $B_H(w)$ through the top side.}
    \label{S:f:boxes}
\end{figure}

Note that in Definition \ref{S:d:box}, we used a real parameter $H>0$, while $H$ is usually in $\NN$. This is because we will use the objects defined above with non-integer parameters as of Section \ref{S:ss:traps}.

\begin{proposition}\label{S:p:exit_box}
\ncS{S:c:box}
    There exists $\ucS{S:c:box}>0$ such that for every $w\in \R\times\ZZ$ and $H\in\NN^*$, we have
\begin{align}\label{S:e:to_be_shown}
\PP\left(F_H(w)^c\right)\leqslant \ucS{S:c:box}^{-1}e^{-\ucS{S:c:box} H^{1/2}}.\end{align}
This is also true when replacing $\PP$ by $\PP^\omega$ for $\omega\in\mathcal{A}.$
\end{proposition}

\begin{proof}
This is a direct consequence of Propositions \ref{S:p:probability_of_going_down} and \ref{S:p:horizontal_bounds}, along with a union bound.
\end{proof}


\begin{remark}\label{S:r:mixing_boxes}
\ncSH{H:size_box}
We will often use Fact \ref{S:p:mixing} with boxes $\{B_H(w),w\in\RR\times\ZZ\}$ and $h=H/2$. The diameter of $B_H(w)$ is precisely $(2\beta+1)H=2(2\beta+1)h$, and its height is equal to $H+H'$, and we want it to be at most $3h$, so we want $H'\leqslant H/2$. Furthermore, we will often encounter pairs of boxes $B,B'$ satisfying $\mathrm{sep}(B,B')\geqslant H-H'$, which is another reason to require $H'\leqslant H/2$, so that $\mathrm{sep}(B,B')\geqslant h$. Therefore, for the rest of the paper, we will work with $H\geqslant \ucSH{H:size_box}$, where $\ucSH{H:size_box}\in\NN$ is such that $$\forall H\geqslant \ucSH{H:size_box},\;H'\leqslant H/2.$$
\end{remark}

\subsection{The multi-scale renormalization method}\label{S:sss:renormalization_scheme}

The proofs of several major propositions in the rest of the paper are based on the fundamental idea of multi-scale renormalization, which gives a practical method for using mixing property \eqref{S:p:mixing}. We now give a general idea of how such a proof works, and we will often refer to it in the future.

Suppose we want to show an estimate for the probability of some family of "bad" events $(A_H)_{H\in\NN}.$

Step 1. We start by focusing on some subsequence $(A_{H_k})_{k}$. We set $p_{k}=\PP(A_{H_k})$. We show that when $A_{H_{k+1}}$ occurs, we can find two boxes $B,B'$ satisfying the assumptions of Fact \ref{S:p:mixing} with $h=h_k$ for some $h_k>0$, such that two events of probability $p_k$ supported by $B$ and $B'$ occur. There can be entropy for where theses two boxes are; let us say that there are $C_k$ possible couples of boxes.

Step 2. We deduce the desired estimate for $(p_k)_{k\in\N}$:
\begin{itemize}
\item Using Fact \ref{S:p:mixing} and a union bound, we get $$p_{k+1}\leqslant C_k\,(p_k^2+\ucS{S:c:mixing} h_k^{-\alpha}).$$
\item From this inequality we deduce the desired estimate of $p_k$ by induction on $k$, provided that the base case is satisfied. For this to work, the scales $(H_k)_k$ and the estimate to show have to be chosen properly, and some control on $C_k$ is necessary.
\end{itemize}

Step 3. We conclude using an interpolation argument to get a result for every $H\in\NN$ instead of only the subsequence.

In order to accommodate to the polynomial mixing, it will be useful to use the following scales.
\begin{definition}\label{S:d:scales}
We set $L_0=10^{10}$ and, for $k\geqslant 0$, $$L_{k+1}=l_k\,L_k,\;\;\;\text{where}\;\;\;l_k=\lfloor L_k^{1/4}\rfloor.$$
\end{definition}

The choice $10^{10}$ will become clearer in the proof of Proposition \ref{S:p:density}.

The rest of this paper will be dedicated to showing Lemma \ref{S:l:LLN_spatial_subsequence}. To do this, we strongly rely on methods developed in \cite{BHT}. First, in Section \ref{S:ss:part1}, we will show that there exist limiting directions $v_-$ and $v_+$ that bound the asymptotic behavior of our random walk with high probability. This requires to adapt the methods in \cite{BHT} by addressing two technical issues: the deterministic drift in the vertical direction is lost in the static framework, and the random walks can revisit their pasts. Then, in Section \ref{S:ss:part2}, we will show that these two limiting directions actually coincide, which will give us the limiting direction $\nu$ in Lemma~\ref{S:l:LLN_spatial_subsequence}. It is in this part of the proof that introducing a weaker "barrier" property as a replacement of the monotonicity property of \cite{BHT} will be instrumental; see Proposition \ref{S:p:barrier_property}.

\section{Limiting directions}\label{S:ss:part1}

\subsection{Definitions and main results}
Recall that we say that all points $x\in\RR^2$ satisfying $\pi_1(x)=a\pi_2(x)$ with $\pi_2(x)\neq 0$ have direction $a$. The goal of this section is to show that there exist two directions $v_-$ and $v_+$ that somehow bound the spatial behavior of our random walk in the long run. This property is made clearer in Lemma \ref{S:p:estimates_on_p}. It will consist of the first part of the proof of Lemma \ref{S:l:LLN_spatial_subsequence}, and we will show that in fact $v_-=v_+$ in Section \ref{S:ss:part2}, thus concluding the proof.

As a matter of fact, we aim at showing a stronger version of Lemma \ref{S:l:LLN_spatial_subsequence} by considering not only one fixed random walk but all the random walks starting simultaneously in $I_H(w)$ from Definition \ref{S:d:box}. This will be instrumental in Section \ref{S:ss:part2}, where we will need to control the directions of lots of random walks at once.

Recall notation $\tau^{y,\Gamma}_{H,w}$ from Definition \ref{S:d:tau}.

\begin{definition}\label{S:d:limiting_speeds}
Let $w\in\RR\times\ZZ$ and $H\in\NN^*$. Let $y\in I_H(w)$ and $\Gamma\in\mathcal{H}$. We define the direction of $Z^{y,\Gamma}$ at height $H$ with reference point $w$ to be
\begin{align*}\begin{array}{c}
    V_{H,w}^{y,\Gamma}=\displaystyle\frac{1}{H}\;\left(X^{y,\Gamma}_{\tau_{H,w}}-\pi_1(y)\right).\end{array}
\end{align*}
As usual, when $w$ or $y$ is omitted, it means that it is $o$, and an omission of $\Gamma$ means $\Gamma=\mathbf{0}.$ Now let $v\in\RR$. We consider the following events:
$$\begin{array}{l}
A_{H,w}(v)=\left\{\exists\, y\in I_H(w),\,V_{H,w}^{y}\geqslant v\right\};\\
\tilde{A}_{H,w}(v)=\left\{\exists\,y\in I_H(w),\,V_{H,w}^{y}\leqslant v\right\}.
\end{array}$$
As usual, $A_H(v)=A_{H,o}(v)$ and $\Tilde{A}_H(v)=\Tilde{A}_{H,o}(v)$. We set
$$\begin{array}{l} p_H(v)=\PP(A_{H}(v));\\
\tilde{p}_H(v)=\PP(\tilde{A}_{H}(v)).  \end{array}$$
We define the limiting directions by setting
$$\begin{array}{l}        v_+=\inf\left\{v\in\RR,\,\displaystyle\liminf_{H\to\infty} p_H(v)=0\right\};\\
        v_-=\sup\left\{v\in\RR,\,\displaystyle\liminf_{H\to\infty} \tilde{p}_H(v)=0\right\}.\end{array}$$
\end{definition}

\begin{figure}[!h]
    \centering
    \begin{tikzpicture}[use Hobby shortcut, scale=0.8]
        \draw[thick] (0,0) rectangle (5,1);
        \draw[above left] (3,0) node {$I_H(w)$};
        \draw[fill,below left] (0,0) circle (.05) node {$w$};
        \draw[fill, above] (0.7,0.5) circle (.05) node {$y$};
        \draw (0.7,0.5) .. (1,0.3) .. (1,0.5) .. (1.5,2) .. (-0.5,3) .. (-1.7,5);
        \draw[fill, above] (-1.7,5) circle (.05) node {$(\pi_1(y)+H V^y_{H,w},\pi_2(w)+H)$};
        \draw[dashed] (-2.3,5) -- (-2.3,-1);
        \draw[dashed] (3.7,5) -- (3.7,-1);
        \draw[thick, dotted] (0.7,0) -- (-2,5);
        \draw[->,>=latex] (0.6,-1.2) -- (-2.3,-1.2);
        \draw[->,>=latex] (0.8,-1.2) -- (3.7,-1.2);
        \draw[below left] (0.6,-1.2) node{$\beta H$};
        \draw[below right] (0.8,-1.2) node{$\beta H$};
        \draw[thick, below left] (-3,5) rectangle (8,-1);
        \draw[below left] (6.5,4) node {$B_H(w)$};
    \end{tikzpicture}
    \caption{Illustration of $A_{H,w}(v)$. The sample path started at $y$ reaches height $\pi_2(w)+H$ with $V_{H,w}^y$ larger than the direction $\nu$ given by the dotted line.}
    \label{S:f:boxes2}
\end{figure}
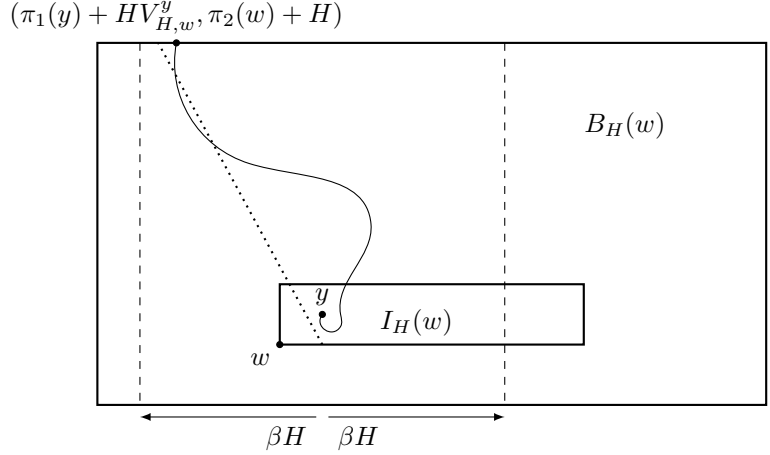

Note that when $\pi_2(y)=\pi_2(w)$, $V_{H,w}^{y,\Gamma}=V_{H,y}^{y,\Gamma}$ is simply the direction of $Z^{y,\Gamma}$. Mind that when $\pi_2(y)>\pi_2(w)$ however, this is not exactly true anymore.

Note that because of translation invariance, $\PP(A_{H,w}(v))$ and $\PP(\tilde{A}_{H,w}(v))$ actually do not depend on $w$, which is why we only considered the random walk starting at the origin for the definitions of $v_-$ and $v_+$. Indeed, we can first restrict ourselves to $w\in (-1,0]\times\{0\}$, using Corollary \ref{S:c:invariance}. Then, $H$ being in $\NN$ here, we observe that $I_H(w)=I_H(o)$ for every $w\in(-1,0]\times\{0\}$, so that $A_{H,w}(v)=A_H(v)$. This would be wrong if $H$ was any positive real number, in which case, depending on $w$, $I_H(w)$ can have cardinality $\lfloor H\rfloor H'$ or $\lceil H\rceil H'$. That is why we will have to be more careful later, in Lemma \ref{S:l:bound_trap}.

It may sound unclear why we use liminfs in the definitions of $v_-$ and $v_+$, instead of limsups. In fact, this will be required in order to get a much needed uniform lower bound on the probability for the random walk to attain directions greater than but close to $v_-$ over long time intervals (see Lemma \ref{S:l:bound_trap}).

Note that we never stated that $v_-\leqslant v_+$. In fact, it is not an obvious consequence of their definitions, but it will be a consequence of Lemma \ref{S:p:estimates_on_p}.

\begin{fact}
We have the following bounds on $v_-$ and $v_+$:
\begin{align}\label{S:e:bounds_on_v_+}\left\{\begin{array}{l}
-\beta\leqslant v_+\leqslant \beta;\\
-\beta\leqslant v_-\leqslant \beta.\end{array}\right.
\end{align}
\end{fact}

\begin{proof}
The proof being symmetric, let us just focus on the bounds for $v_+$.
\begin{itemize}[leftmargin=*]
    \item If $v<-\beta$, then using Proposition \ref{S:p:horizontal_bounds},
        \begin{align*}
            p_H(v)=\PP(\exists\, y\in I_H,\,V_{H}^y\geqslant v)
            \geqslant \PP(V_H\geqslant -\beta)
            \geqslant \PP(D_H)
            \geqslant 1-\ucS{S:c:ball}^{-1} e^{-\ucS{S:c:ball}H}\xrightarrow[H\to\infty]{} 1.
        \end{align*}
    \item If $v>\beta$, using Proposition \ref{S:p:horizontal_bounds} again,
    \begin{align*}
        p_H(v)
        &=\PP(\exists\,y\in I_H,\,V_{H}^y\geqslant v)\leqslant \PP(\exists\,y\in I_H,\,V_{H}^y>\beta)\\
        &\leqslant HH' \sup_{y\in I_H(w)} \PP\left((D_H^y)^c\right)\leqslant \ucS{S:c:ball}^{-1} HH' e^{-\ucS{S:c:ball}H^{1/2}}\xrightarrow[H\to\infty]{} 0.
    \end{align*}
\end{itemize}
\end{proof}

\begin{remark}\label{S:r:monotonicity_v_+}
Note that $v\in\RR\mapsto p_H(v)$ is a non-increasing function. Therefore, for $v>v_+$, we must have $\displaystyle\liminf_{H\to\infty} p_H(v)=0.$ Similarly, for $v<v_-$, $\displaystyle\liminf_{H\to\infty} \tilde{p}_H(v)=0$.
\end{remark}

In spite of Remark \ref{S:r:monotonicity_v_+}, the definitions that we gave for $v_-$ and $v_+$ are quite weak at first glance, because we only have information on the liminfs. Our goal now is to show that for $v>v_+$ and $v<v_-$, the liminfs given in Remark \ref{S:r:monotonicity_v_+} are actual limits, and we will even prove a precise estimate for $p_H(v)$ and $\tilde{p}_H(v)$ when $H$ goes to infinity.

\begin{lemma}[Deviation bounds]\label{S:p:estimates_on_p}
For every $\varepsilon>0$, there exists $\ucS{S:c:deviation}=\ucS{S:c:deviation}(\varepsilon)>0$ such that for every $H\in\NN^*$,
\begin{align}
    \label{S:deviation bounds_estimates}
    p_H(v_++\varepsilon)\leqslant \ucS{S:c:deviation} H^{-\alpha/4};\\
    \label{S:deviation bounds_estimates2}
    \tilde{p}_H(v_--\varepsilon)\leqslant \ucS{S:c:deviation} H^{-\alpha/4}
\end{align}
\end{lemma}

The proof of Lemma \ref{S:p:estimates_on_p} is the goal of Section \ref{S:ss:proof_est}.

\begin{corollary}\label{S:c:v_-<v_+}
The two limiting directions satisfy $v_-\leqslant v_+.$
\end{corollary}
\begin{proof}
    We now prove that this is indeed a corollary of Lemma \ref{S:p:estimates_on_p}. We argue by contradiction and assume that $v_->v_+$. Let $v=\frac{v_++v_-}{2}\in (v_+,v_-)$ and $\varepsilon=\frac{v_--v_+}{2}>0$. Using \eqref{S:deviation bounds_estimates} and \eqref{S:deviation bounds_estimates2}, we have, for every $H\in\NN^*$, $p_H(v)\leqslant \ucS{S:c:deviation}(\varepsilon) H^{-\alpha/4}$ and $\tilde{p}_H(v)\leqslant \ucS{S:c:deviation}(\varepsilon) H^{-\alpha/4}$. Therefore, $1\leqslant p_H(v)+\tilde{p}_H(v)\leqslant 2\ucS{S:c:deviation}(\varepsilon) H^{-\alpha/4}\xrightarrow[H\to\infty]{}0$, which is a contradiction.
\end{proof}

This corollary can be improved as follows.
\begin{lemma}\label{S:l:v_-=v_+}
We have $v_-=v_+$. We call this quantity $\nu$.
\end{lemma}

The proof of Lemma \ref{S:l:v_-=v_+} is the goal of Section \ref{S:ss:part2} and can be found more precisely in Section \ref{S:ss:final_proof}. For now, let us prove Lemma \ref{S:l:LLN_spatial_subsequence} as a consequence of Lemmas \ref{S:p:estimates_on_p} and \ref{S:l:v_-=v_+}.

\begin{proof}[Proof of Lemma \ref{S:l:LLN_spatial_subsequence}]
Let $\varepsilon>0$. Combining Lemma \ref{S:p:estimates_on_p} with Lemma \ref{S:l:v_-=v_+}, we have, for every $H\in\NN^*,$
$$\PP\left(\left\vert\frac{Z_{\tau_H}}{H}-\nu\right\vert\geqslant\varepsilon\right)\leqslant  2\,\ucS{S:c:deviation}(\varepsilon)\,H^{-\alpha/4},$$
which is summable, since $\alpha>4$. Therefore, Borel-Cantelli's lemma ensures that $$\PP\text{-almost surely,}\;\;\;\;\frac{Z_{\tau_H}}{H}\xrightarrow[H\to\infty]{} \nu,$$
concluding the proof of Lemma \ref{S:l:LLN_spatial_subsequence}.
\end{proof}

\subsection{Deviation bounds: proof of Lemma \ref{S:p:estimates_on_p}}\label{S:ss:proof_est}

\subsubsection{Ideas of the proof}
\ncSk{S:k:first_scale}
Let us first give some heuristic insight on how the proof is going to unfold.
\begin{itemize}[leftmargin=*]
    \item We will only show \eqref{S:deviation bounds_estimates}; \eqref{S:deviation bounds_estimates2} is shown in the same way, the proof being symmetric.
    \item We fix $v>v_+$. The road map to show the estimate on $p_H(v)$ is given by the renormalization method explained in Section \ref{S:sss:renormalization_scheme}, with a sequence of scales given by $(h_0L_k)_{k\geqslant \ucSk{S:k:first_scale}}$, where $h_0$ and $\ucSk{S:k:first_scale}$ will have to be chosen properly, depending on $v$. In the induction that will give an estimate on this sequence of scales, the choice of $\ucSk{S:k:first_scale}$ and the definition of $(L_k)_{k\in\NN}$ will be instrumental in the induction step, while $h_0$ is chosen for the base case to work.
    \item We are going to work with the sequence of events $(A_{h_0L_k}(v_k))_{k\geqslant \ucSk{S:k:first_scale}}$ with an appropriate choice of $(v_k)_{k\geqslant \ucSk{S:k:first_scale}}$. The goal is to show that with good probability, on $A_{h_0L_{k+1}}(v_{k+1})$, we can find events $A_{h_0L_k,w_1}(v_k)$ and $A_{h_0L_k,w_2}(v_k)$ with some base points $w_1,w_2$ located on a grid whose cardinality does not depend on $h_0$. The challenge is that we asked those two events to have everywhere-zero histories. One way to find them is to look for the two starting points $y_1$ and $y_2$ (from the definitions of $A_{h_0L_k,w_1}(v_k)$ and $A_{h_0L_k,w_2}(v_k)$) on cut lines that we ask to be at vertical distance at most $(h_0L_k)^{1/2}$ of two points $w_1$ and $w_2$ on our grid. This is the whole reason why in our paper, $I_H(w)$ is a rectangle, instead of being a horizontal interval as in \cite{BHT}.
\end{itemize}

\subsubsection{Estimate along a subsequence}
We are first going to show an estimate similar to \eqref{S:deviation bounds_estimates} along a subsequence. More precisely, we are going to show the following result.

\begin{lemma}\label{S:l:estimate_to_show}
Let $v>v_+$. There exist $\ucSk{S:k:first_scale}=\ucSk{S:k:first_scale}(v)\in\NN$ and $h_0=h_0(v)\in [1,\infty)$ such that for every $k\geqslant \ucSk{S:k:first_scale}$, $h_0L_k\in\NN$ and
\begin{align}\label{S:e:estimate_we_want}
p_{h_0L_k}(v)\leqslant L_k^{-\alpha/2}.\end{align}
\end{lemma}

\begin{proof}
We fix $v>v_+$ for the whole proof. Recall Definition \ref{S:d:scales}. We let $\ucSk{S:k:first_scale}=\ucSk{S:k:first_scale}(v)\in\NN$ be such that \begin{align}\label{S:e:condition_de_depart}
\displaystyle\sum_{k\geqslant \ucSk{S:k:first_scale}} \left(\displaystyle\frac{2\beta}{H_k'}+\displaystyle\frac{6\beta}{l_k}\right)<\displaystyle\frac{v-v_+}{2}.
\end{align}
We also set $v_{\ucSk{S:k:first_scale}}=\frac{v+v_+}{2}$. Using Remark \ref{S:r:monotonicity_v_+}, note that since $v_{\ucSk{S:k:first_scale}}>v_+$, $$\liminf_{H\to\infty} p_H(v_{k_0})=0.$$ Therefore there exists $H\geqslant L_{\ucSk{S:k:first_scale}}$ such that $p_H(v_{\ucSk{S:k:first_scale}})\leqslant L_{\ucSk{S:k:first_scale}}^{-\alpha/2}$. Let $h_0=H/L_{\ucSk{S:k:first_scale}}\in [1,\infty).$ Using that $L_k$ is a multiple of $L_{\ucSk{S:k:first_scale}}$ for any $k\geqslant \ucSk{S:k:first_scale}$, we have $h_0L_k\in\NN$ for all $k\geqslant \ucSk{S:k:first_scale}$, and
\begin{align*}
p_{h_0 L_{\ucSk{S:k:first_scale}}}(v_{k_0})\leqslant L_{\ucSk{S:k:first_scale}}^{-\alpha/2}.
\end{align*}
Now that $h_0$ is fixed, we let
\begin{align}\label{S:n:H_k}
    H_k=h_0L_k \text{ for all $k\geqslant k_0$}.
\end{align} Recall notation $H_k'$ from \eqref{S:n:H'}. We now define a sequence $(v_k)_{k\geqslant \ucSk{S:k:first_scale}}$ by setting, for $k\geqslant \ucSk{S:k:first_scale}$,
    \begin{align*}
    \left\{\begin{array}{l} \bar{v}_k=v_k+\displaystyle\frac{2\beta}{H_k'};\vspace{3pt}\\ v_{k+1}=\bar{v}_k+\displaystyle\frac{6\beta}{l_k}.\end{array}\right.\end{align*}

This definition combined with \eqref{S:e:condition_de_depart} and the fact that $h_0\geqslant 1$ ensures that $v_k\xrightarrow[k\to\infty]{\nearrow} v_\infty<v.$ Therefore, using Remark \ref{S:r:monotonicity_v_+}, in order to show \eqref{S:e:estimate_we_want}, we need only show that
\begin{align}\label{S:e:estimate_to_show} \forall k\geqslant \ucSk{S:k:first_scale},\;\;p_{H_k}(v_k)\leqslant L_k^{-\alpha/2}.\end{align}
The rest of the proof is dedicated to proving \eqref{S:e:estimate_to_show}.
\ncS{S:c:cardinality}
\paragraph{Localization at scale $L_k$.}
We first need to define an event $\mathcal{F}_k$ that guarantees that the random walks starting in $I_{H_{k+1}}$ will stay in $B_{H_{k+1}}$ and that their horizontal behaviours at scale $H_k$ are properly bounded.

Recall Definitions \eqref{S:d:event_D} and \eqref{S:d:event_F}. Let \begin{align}\label{S:e:eventF}
\mathcal{F}_k=F_{H_{k+1}}\cap \bigcap_{y\in I_{H_{k+1}}}\bigcap_{j=0}^{l_k-1}\; \Theta_{jH_k} D_{H_k}^y.\end{align}
For each $y$ and $j$, in order to bound $\PP\left((\Theta_{jH_k} D_{H_k}^y)^c\right)$, we use Proposition \ref{S:p:independence}. Using Proposition \ref{S:p:horizontal_bounds}, for every $\omega\in\mathcal{A}$, $z\in\ZZ^2$ and $\Gamma\in\mathcal{H}$, we have $\PP^\omega\left((D_{H_k}^{z,\Gamma})^c\right)\leqslant \ucS{S:c:ball}^{-1} e^{-\ucS{S:c:ball} H_k},$ which is uniform in $\omega$, $z$ and $\Gamma.$ So, by Proposition \ref{S:p:independence}, for every $y$ and $j$, we have $$\PP\left((\Theta_{jH_k} D_{H_k}^y)^c\right)\leqslant \ucS{S:c:ball}^{-1} e^{-\ucS{S:c:ball} H_k}.$$ In the end, using union bounds and Proposition \ref{S:p:exit_box},
\begin{align}\label{S:e:bound_F}
    \PP(\mathcal{F}_k^c)\leqslant \ucS{S:c:box}^{-1} e^{-\ucS{S:c:box} H_{k+1}^{1/2}}+H_{k+1} H_{k+1}' l_k \ucS{S:c:ball}^{-1} e^{-\ucS{S:c:ball} H_k}\leqslant c^{-1} e^{-cH_k^{1/2}},
\end{align}
where $c>0$ does not depend on $h_0$ and $\ucSk{S:k:first_scale}$.

\paragraph{Link between scales $H_k$ and $H_{k+1}$.}
Let us fix $k\geqslant \ucSk{S:k:first_scale}$. We are now moving on to the crucial idea of the proof: on event $A_{H_{k+1}}(v_{k+1})$, which is observable at scale $H_{k+1}$, several similar events occur at scale $H_k$. We first define a set of reference points in order to control the entropy when going from scale $H_{k+1}$ to scale $H_k$.

We define the set
\begin{align}\label{S:CCC}
    \CCC_k=\{(iH_k,jH_k)\in H_k \ZZ^2,\,-\lceil\beta l_k\rceil\leqslant j<\lceil(\beta+1)l_k\rceil,\,-\lfloor H_{k+1}'/H_k\rfloor \leqslant j<l_k\}.
\end{align}
Recalling \eqref{S:d:interval_box} and that $B_{H_{k+1}}=B_{H_{k+1}}(o)$, this set satisfies \begin{align}\label{S:d:CCC}
\bigcup_{w\in \CCC_k} \mathcal{I}_{H_k}(w)\supseteq B_{H_{k+1}}\cap \left(\ZZ\times H_k\ZZ\right),
\end{align}
where the union above is disjoint (note that boxes $B_{H_k}(w)$ with $w\in\CCC_k$ are not disjoint though). The cardinality of $\CCC_k$ can be bounded from above by $\ucS{S:c:cardinality} l_k^2$, where $\ucS{S:c:cardinality}>0.$

Roughly speaking, our goal is to show that up to an event of small probability, $A_{H_{k+1}}(v_{k+1})$ implies the existence of two points $w_1,w_2\in\CCC_k$ satisfying $|\pi_2(w_1)-\pi_2(w_2)|\geqslant 2H_k$ and such that $A_{H_k,w_1}(v_k)$ and $A_{H_k,w_2}(v_k)$ occur.

Let us fix $y\in I_{H_{k+1}}.$ For now, we claim to have the following inclusion of events:
\begin{align}\label{S:e:cascade}
\left\{V^y_{H_{k+1}}\geqslant v_{k+1}\right\}\,\cap\, \mathcal{F}_k
\subseteq \left\{\begin{array}{c}\text{there exist three $j\in\llbracket 1,l_k-1\rrbracket$}\\ \text{such that } X^y_{\tau_{(j+1)H_k}}\geqslant X^y_{\tau_{jH_k}}+\bar{v}_k H_k\end{array}\right\}.
\end{align}
Indeed, let us assume that $\mathcal{F}_k$ occurs and that $V^y_{H_{k+1}}\geqslant v_{k+1}$. We argue by contradiction and assume that $X^y_{\tau_{(j+1)H_k}}\geqslant X^y_{\tau_{jH_k}}+\bar{v}_kH_k$ for at most two $j\in\llbracket 1,l_k-1\rrbracket$. Now, the horizontal displacement of $Z^y$ between times $0$ and $\tau^y_{H_{k+1}}$ is the sum of $l_k$ horizontal displacements, $l_k-3$ of which we can now bound by $\bar{v}_kH_k$, and the three remaining ones can be bounded using $\Theta_{jH_k} D_{H_k}^y$ for the right $j$'s. More precisely, using $\mathcal{F}_k$,
    \begin{align*}
        X^y_{\tau_{H_{k+1}}}
        &=X^y_{\tau_{H_k}}+\sum_{j=1}^{l_k-1} \left(\pi_1(X^y_{\tau_{(j+1)H_k}})-\pi_1(X^y_{\tau_{jH_k}})\right)\\
        &<\pi_1(y)+(l_k-3)\bar{v}_kH_k+3\beta H_k\\
        &=\pi_1(y)+\left(\bar{v}_k+\frac{3(\beta-\bar{v}_k)}{l_k}\right)H_{k+1}\\
        &<\pi_1(y)+\left(\bar{v}_k+\frac{6\beta}{l_k}\right)H_{k+1}\\
        &=\pi_1(y)+v_{k+1}H_{k+1},
    \end{align*}
where in the last inequality, we used bounds \eqref{S:e:bounds_on_v_+} to get that $\bar{v}_k>v_k>v_+\geqslant -\beta$. This amounts to saying that $V^y_{H_{k+1}}<v_{k+1}$, which contradicts our assumption, thus concluding the proof of \eqref{S:e:cascade}.

Now, note that since $F_{H_{k+1}}\subseteq \mathcal{F}_k$, on the latter event we know that for every $j\in\llbracket 1,l_k-1\rrbracket$, $Z^y_{\tau_{jH_k}}\in B_{H_{k+1}}$. Furthermore $Y^y_{\tau_{jH_k}}\in H_k\ZZ$ by construction (this is not true for $j=0$, which is why we left it out). Therefore, $Z^y_{\tau_{jH_k}}$ belongs to $\mathcal{I}_{H_k}(w)$ for some $w\in\CCC_k$, using \eqref{S:d:CCC}. The issue however is that in \eqref{S:e:cascade}, events $\left\{X^y_{\tau_{(j+1)H_k}}\geqslant X^y_{\tau_{jH_k}}+\bar{v}_kH_k\right\}$ implicitly feature a non-zero history $N^y_{\tau_{jH_k}}$, while our goal is to get zero-history events of the form $A_{H_k,w}(v_k)$.


\paragraph{Removal of histories.} In order to make the histories disappear, we use cut lines as defined before Assumption \ref{S:a:cut_line}, which requires defining a new event of large probability that will fulfill the technical requirements for the rest of the argument, namely that we can find cut lines close enough, that before then the random walks do not go too far horizontally, and that they all stay in boxes allowing us to use Fact \ref{S:p:mixing}. Recall notation $H_k'$ defined in \eqref{S:n:H'}. We let
\begin{align}\label{S:e:eventG}
\mathcal{G}_k=\displaystyle\bigcap_{y\in I_{H_{k+1}}}\,\displaystyle\bigcap_{j=1}^{l_k-1}\,\left(\Theta_{jH_k} D_{H_k'}^y\cap \left\{\Xi(\Theta_{jH_k} Z^y)<H_k'\right\}\right)\cap\bigcap_{w\in\CCC_k} F_{H_k}(w).
\end{align}
In order to control the probability of $\mathcal{G}_k$, we use Proposition \ref{S:p:independence} again, as well as Proposition \ref{S:p:exit_box} and Assumption \ref{S:a:cut_line}. Using a union bound, we have
\begin{align}
    \PP\left(\mathcal{G}_k^c\right)
    &\leqslant H_{k+1}H_{k+1}' l_k \,\sup_{y,j}\left[\PP\left((\Theta_{jH_k} D_{H_k'}^y)^c\right)+\PP\left(\Xi(\Theta_{jH_k} Z^y)\geqslant H_k'\right)\right]+\ucS{S:c:cardinality}l_k^2 \ucS{S:c:box}^{-1}e^{-\ucS{S:c:box} H_k^{1/2}}\nonumber\\
    &\leqslant H_{k+1}H_{k+1}' l_k \left(\ucS{S:c:ball}^{-1} e^{-\ucS{S:c:ball} H_k^{1/2}} + \ucS{S:c:cut_line}^{-1} e^{-\ucS{S:c:cut_line} H_k^{1/2}} \right)+ c^{-1}e^{-c H_k^{1/2}}\nonumber\\
    &\leqslant c^{-1} e^{-c H_k^{1/2}},\label{S:e:bound_G}
\end{align}
where $c>0$ does not depend on $h_0$ and $\ucSk{S:k:first_scale}$.

Let $j\in\llbracket 1,l_k-1\rrbracket$, and let $\xi_j^y$ be the location of $Z^y$ on the first cut line reached after height $jH_k$, that is
\begin{align}\label{S:e:zeta}
\xi_j^y=\Theta_{jH_k} Z^y_{T(\Theta_{jH_k} Z^y)}
\end{align}
(recall notations from Definition \ref{S:d:cut_line}). Note that, by construction, we have the following equality (mind that the left-hand side features no history now!): $$Z^{\xi_j^y}_{\tau_{H_k,Z^y_{\tau_{jH_k}}}}=Z^y_{\tau_{(j+1)H_k}}.$$
Now, on $\mathcal{G}_k\,\cap\, \left\{X^y_{\tau_{(j+1)H_k}}\geqslant X^y_{\tau_{jH_k}}+\bar{v}_kH_k\right\}$, we have
\begin{align*}
    X^y_{\tau_{(j+1)H_k}}
    &\geqslant X^y_{\tau_{jH_k}}+\bar{v}_kH_k\\
    &\geqslant \pi_1(\xi_j^y)+\bar{v}_kH_k-\beta H_k'&\text{using $\Theta_{jH_k} D^y_{H_k'}$}\\
    &\geqslant \pi_1(\xi_j^y)+v_kH_k&\mbox{by definition of $\bar{v}_k$,}
\end{align*}
so, in other words, $V^{\xi_j^y}_{H_k,Z^y_{\tau_{jH_k}}}\geqslant v_k$. Using \eqref{S:e:cascade}, this means that we have found three points (given by $\xi_j^y$ for three values of $j\in\llbracket 1,l_k-1\rrbracket$) with the right lower bound on their directions and with everywhere-zero initial histories. Furthermore, on $\mathcal{F}_k\cap\mathcal{G}_k$, for $j\in\llbracket 1,l_k-1\rrbracket$, we have $\pi_2(\xi_j^y)<Y^y_{\tau_{jH_k}}+H_k'=jH_k+H_k'$. Therefore the $\xi_j^y$ are located in three rectangles $I_{H_k}(w_i)$ for $w_i\in\CCC_k$ satisfying $|\pi_2(w_i)-\pi_2(w_j)|\geqslant H_k$ for $i\neq j$. As a result,
\begin{align}\label{S:pourplustard}
\mathcal{F}_k\cap\mathcal{G}_k\cap A_{H_{k+1}}(v_{k+1})
&\subseteq \bigcup_{\underset{|\pi_2(w_1)-\pi_2(w_2)|\geqslant 2H_k}{w_1,w_2\in \mathcal{C}_k}}\begin{array}{c} (A_{H_k,w_1}(v_k)\,\cap\,F_{H_k}(w_1))\\\cap\,(A_{H_k,w_2}(v_k)\,\cap\,F_{H_k}(w_2)).\end{array}
\end{align}
Now, events $A_{H_k,w_1}(v_k)\,\cap\,F_{H_k}(w_1)$ and $A_{H_k,w_2}(v_k)\,\cap\,F_{H_k}(w_2)$ above are respectively measurable with respect to boxes $B_{H_k}(w_1)$ and $B_{H_k}(w_2)$, and these two boxes satisfy the assumptions of Fact \ref{S:p:mixing} with $h=H_k/2$, by Definition \ref{S:d:box} and the fact that $|\pi_2(w_1)-\pi_2(w_2)|\geqslant 2H_k$ (recall Remark \ref{S:r:mixing_boxes}). Therefore,
\ncS{S:c:end_proof}
\begin{align*}
    \PP\left(\mathcal{F}_k\cap\mathcal{G}_k\cap A_{H_{k+1}}(v_{k+1})\right)
    &\leqslant |\CCC_k|^2 \left(p_{H_k}(v_k)^2+\ucS{S:c:mixing}(H_k/2)^{-\alpha}\right)\\
    &\leqslant \ucS{S:c:cardinality}^2\,l_k^{4}\, \left(p_{H_k}(v_k)^2+\ucS{S:c:mixing}(H_k/2)^{-\alpha}\right).
\end{align*}
In the end, using bounds \eqref{S:e:bound_F} and \eqref{S:e:bound_G}, we get
\begin{align}\label{S:e:constant_end_proof}
    \PP(A_{H_{k+1}}(v_{k+1}))
    &\leqslant \ucS{S:c:cardinality}^2\,l_k^{4}\, \left(p_{H_k}(v_k)^2+\ucS{S:c:mixing}(H_k/2)^{-\alpha}\right)+\PP\left(\mathcal{F}_k^c\right)+\PP\left(\mathcal{G}_k^c\right)\nonumber\\
    &\leqslant \ucS{S:c:end_proof}\,l_k^{4}\left(p_{H_k}(v_k)^2+\ucS{S:c:end_proof} L_k^{-\alpha}\right),
\end{align}
for some constant $\ucS{S:c:end_proof}>0$ that does not depend on $h_0$ and $\ucSk{S:k:first_scale}$ (to which we gave a name because we will need it again for the proof of Proposition \ref{S:p:density} at the end of our paper). By induction, we can conclude that if $p_{H_k}(v_k)\leqslant  L_k^{-\alpha/2},$ then
\begin{align*}
    \frac{p_{H_{k+1}}(v_{k+1})}{L_{k+1}^{-\alpha/2}}\leqslant c\,L_{k+1}^{\alpha/2}\,l_k^{4}\,L_k^{-\alpha}\leqslant c\,L_k^{\frac{-3\alpha+8}{8}},
    \end{align*}
for a well-chosen constant $c>0$ that does not depend on $h_0$ and $k_0$. Since $\alpha>3$, up to taking an even larger $\ucSk{S:k:first_scale}$ (independently on $h_0$), we can assume that this is at most $1$, which concludes the induction and the proof of estimate \eqref{S:e:estimate_to_show}.
\end{proof}

\subsubsection{Interpolation}\label{S:sss:interpolation}
\ncSk{S:k:first_scale_large}\ncSk{S:k:scale_interpolation}
We now want to interpolate \eqref{S:e:estimate_we_want} to show Lemma \ref{S:p:estimates_on_p}. Let $v>v_+$. Recall that constants $h_0$ and $\ucSk{S:k:first_scale}$ from Lemma \ref{S:l:estimate_to_show} actually depend on the direction in the estimate. We let $v'=\frac{v_++v}{2}$, $v''=\frac{v'+v}{2}$, $h_0=h_0(v')$ and define $\ucSk{S:k:first_scale_large}\geqslant \ucSk{S:k:first_scale}(v')$ such that
\begin{align}
& L_{\ucSk{S:k:first_scale_large}}^{1/10}>\frac{2\beta}{v-v'};\label{S:e:cond1}\\
& \frac{2\beta}{L_{\ucSk{S:k:first_scale_large}}^{1/2}}+\frac{2|v'|}{L_{\ucSk{S:k:first_scale_large}}^{1/10}}\leqslant v''-v'.\label{S:e:cond3}
\end{align}
Let $H\geqslant (h_0L_{\ucSk{S:k:first_scale_large}})^{11/10}$, and let $\ucSk{S:k:scale_interpolation}\geqslant \ucSk{S:k:first_scale_large}$ be such that 
\begin{align}\label{S:turlututu}
    (h_0L_{\ucSk{S:k:scale_interpolation}})^{11/10}\leqslant H<(h_0L_{\ucSk{S:k:scale_interpolation}+1})^{11/10}.
\end{align}
Since $\ucSk{S:k:scale_interpolation}\geqslant \ucSk{S:k:first_scale}(v')$ and $v'>v_+$, using Lemma \ref{S:l:estimate_to_show},
\begin{align}\label{S:chapeau_pointu}
    p_{h_0 L_{\ucSk{S:k:scale_interpolation}}}(v')\leqslant L_{\ucSk{S:k:scale_interpolation}}^{-\alpha/2}.
\end{align}

Recall \eqref{S:n:H_k}. Note that \eqref{S:turlututu} implies that $H'<H_{\ucSk{S:k:scale_interpolation}}$, therefore every $y\in I_H$ satisfies $\tau^y_{H_{\ucSk{S:k:scale_interpolation}}}>0.$ Now let $\bar{H}=\left\lfloor H/H_{\ucSk{S:k:scale_interpolation}}\right\rfloor\,H_{\ucSk{S:k:scale_interpolation}}$ be the last multiple of $H_{\ucSk{S:k:scale_interpolation}}$ before $H$.

In order to link what happens at scales $H$ and $H_{\ucSk{S:k:scale_interpolation}}$, similarly to \eqref{S:CCC}, we define a set
$$\hat{\CCC}=\left\{(iH_{\ucSk{S:k:scale_interpolation}},jH_{\ucSk{S:k:scale_interpolation}})\in H_{\ucSk{S:k:scale_interpolation}} \ZZ^2,\,-\lceil \beta H/H_{\ucSk{S:k:scale_interpolation}}\rceil\leqslant i<\lceil (\beta+1)H/H_{\ucSk{S:k:scale_interpolation}}\rceil,\, -\lfloor H'/H_{\ucSk{S:k:scale_interpolation}}\rfloor\leqslant j<\lfloor H/H_{\ucSk{S:k:scale_interpolation}}\rfloor\right\},$$
which satisfies 
\begin{align}\label{S:d:hatC}
\bigcup_{w\in \hat{\CCC}} \mathcal{I}_{H_{\ucSk{S:k:scale_interpolation}}}(w)\supseteq B_H\cap \left(\ZZ\times H_{\ucSk{S:k:scale_interpolation}}\ZZ\right).
\end{align}
and whose cardinality can be bounded from above by $c\,(H/H_{\ucSk{S:k:scale_interpolation}})^2$.

We will work with the following events:
\begin{align*}
&\mathcal{A}_1=\displaystyle \bigcap_{w\in \hat{\mathcal{C}}} A_{H_{\ucSk{S:k:scale_interpolation}},w}(v')^c;\\
&\mathcal{A}_2=\displaystyle\bigcap_{y\in I_H}\left\{X^y_{\tau_{H}}-X^y_{\tau_{\bar{H}}}<(v-v'')H-\beta H_{\ucSk{S:k:scale_interpolation}}\right\};\\
&\mathcal{F}= F_H\cap \displaystyle\bigcap_{y\in I_H} D^y_{ H_{\ucSk{S:k:scale_interpolation}}}\\
&\mathcal{G}=\displaystyle\bigcap_{y\in I_H}\displaystyle\bigcap_{j=1}^{\left\lfloor H/H_{\ucSk{S:k:scale_interpolation}}\right\rfloor-1} \left( \Theta_{jH_{\ucSk{S:k:scale_interpolation}}} D_{H_{\ucSk{S:k:scale_interpolation}}'}^y\,\cap\,\left\{\Xi\left(\Theta_{jH_{\ucSk{S:k:scale_interpolation}}}Z^y\right)<H_{\ucSk{S:k:scale_interpolation}}'\right\} \right).
\end{align*}
For any $y\in I_H$, using \eqref{S:d:hatC}, events $\mathcal{A}_1$ and $\mathcal{G}$ (as well as $F_H$) allow us to bound the displacement of $Z^y$ between times $\tau_{jH_{\ucSk{S:k:scale_interpolation}}}^y$ and $\tau_{(j+1)H_{\ucSk{S:k:scale_interpolation}}}^y$ for $j\in \llbracket 1,\lfloor H/H_{\ucSk{S:k:scale_interpolation}}\rfloor-1\rrbracket$. Indeed, the conditions on cut lines given by $\mathcal{G}$ allow us to find a point inside $I_{H_{\ucSk{S:k:scale_interpolation}}}(w)$, for some $w\in\hat{\CCC}$, for which we can use $A_{H_{\ucSk{S:k:scale_interpolation}},w}(v')^c$ given by $\mathcal{A}_1.$ More precisely, on $\mathcal{A}_1\cap \mathcal{F}\cap\mathcal{G}$ and for every $y\in I_H$, we have
\begin{align*}
    X^y_{\tau_{\bar{H}}}
    &=X^y_{\tau_{H_{\ucSk{S:k:scale_interpolation}}}}+\sum_{j=1}^{\lfloor H/H_{\ucSk{S:k:scale_interpolation}}\rfloor-1} \left( X^y_{\tau_{(j+1)H_{\ucSk{S:k:scale_interpolation}}}}-X^y_{\tau_{jH_{\ucSk{S:k:scale_interpolation}}}} \right)\\
    &\leqslant X^y_{\tau_{H_{\ucSk{S:k:scale_interpolation}}}}+\left(\lfloor H/H_{\ucSk{S:k:scale_interpolation}}\rfloor-1\right)\,\left(\beta H_{\ucSk{S:k:scale_interpolation}}'+v'H_{\ucSk{S:k:scale_interpolation}}\right)\\
    &\leqslant X^y_{\tau_{H_{\ucSk{S:k:scale_interpolation}}}}+\frac{H}{H_{\ucSk{S:k:scale_interpolation}}}(\beta H_{\ucSk{S:k:scale_interpolation}}'+v'H_{\ucSk{S:k:scale_interpolation}})+2|v'| H_{\ucSk{S:k:scale_interpolation}}\\
    &\leqslant X^y_{\tau_{H_{\ucSk{S:k:scale_interpolation}}}}+v''H,
\end{align*}
where in the last line we used \eqref{S:e:cond3} and \eqref{S:turlututu}. Therefore, on $\mathcal{A}_1\cap\mathcal{A}_2\cap\mathcal{F}\cap\mathcal{G}$, we have, for every $y\in I_H$,
\begin{align*}
    X^y_{\tau_{H}}-\pi_1(y)
    &=\left(X^y_{\tau_{H_{\ucSk{S:k:scale_interpolation}}}}-\pi_1(y) \right)+\left(X^y_{\tau_{\bar{H}}}-X^y_{\tau_{H_{\ucSk{S:k:scale_interpolation}}}} \right)+\left(X^y_{\tau_{H}}-X^y_{\tau_{\bar{H}}}\right)\\
    &< \beta H_{\ucSk{S:k:scale_interpolation}}+v''H + (v-v'')H-\beta H_{\ucSk{S:k:scale_interpolation}}=vH.
\end{align*}

As a result, $V_H^y<v$. This being true for any $y\in I_H$, this means that $A_H(v)$ cannot occur, so we have shown that
\begin{align}\label{S:e:inclusion_big}
A_H(v)\subseteq \mathcal{A}_1^c\cup \mathcal{A}_2^c \cup \mathcal{F}^c\cup\mathcal{G}^c.
\end{align}
Now, note that
\begin{align}\label{S:maj1}\PP\left(\mathcal{A}_1^c\right)\leqslant c \left(H/H_{\ucSk{S:k:scale_interpolation}}\right)^2 L_{\ucSk{S:k:scale_interpolation}}^{-\alpha/2},
\end{align}
using Lemma \ref{S:l:estimate_to_show}. Also, by \eqref{S:e:cond1} and the fact that $\ucSk{S:k:scale_interpolation}\geqslant \ucSk{S:k:first_scale_large}$, as well as \eqref{S:turlututu}, we have $(v-v')H>2\beta H_{\ucSk{S:k:scale_interpolation}}$, therefore \begin{align}
    \PP\left(\mathcal{A}_2^c\right)
    &\leqslant HH' \,\PP\left(\Theta_{\bar{H}}X^y_{\tau_{H-\bar{H},Z^y_{\tau_{\bar{H}}}}}-X^y_{\tau_{\bar{H}}}>\beta H_{\ucSk{S:k:scale_interpolation}}\right)\nonumber\\
    &\leqslant HH' \sup_{\omega\in\mathcal{A}}\,\sup_{z,\Gamma}\, \PP^\omega\left(\bigl(D_{H_{\ucSk{S:k:scale_interpolation}}}^{z,\Gamma}\bigr)^c\right)\nonumber\\
    &\leqslant HH' \ucS{S:c:ball}^{-1}e^{-\ucS{S:c:ball} H_{\ucSk{S:k:scale_interpolation}}},\label{S:maj2}
\end{align}
using Propositions \ref{S:p:horizontal_bounds} and \ref{S:p:independence}.
We also have
\begin{align}
&\PP(\mathcal{F}^c)\leqslant \ucS{S:c:box}^{-1} e^{-\ucS{S:c:box} H^{1/2}}+HH' \ucS{S:c:ball}^{-1} e^{-\ucS{S:c:ball} H_{\ucSk{S:k:scale_interpolation}}};\label{S:maj3}\\
&\PP(\mathcal{G}^c)\leqslant HH' \frac{H}{H_{\ucSk{S:k:scale_interpolation}}} \left( \ucS{S:c:ball}^{-1} e^{-\ucS{S:c:ball} H_{\ucSk{S:k:scale_interpolation}}^{1/2}}+\ucS{S:c:cut_line}^{-1} e^{-\ucS{S:c:cut_line} H_{\ucSk{S:k:scale_interpolation}}^{1/2}}\right).\label{S:maj4}
\end{align}
Using \eqref{S:turlututu}, we can see that the upper bounds given by \eqref{S:maj2}, \eqref{S:maj3} and \eqref{S:maj4} are all negligible with respect to that given by \eqref{S:maj1}, so, using \eqref{S:e:inclusion_big}, we get
\begin{align*}
\PP(A_H(v))
&\leqslant c \left(H/H_{\ucSk{S:k:scale_interpolation}}\right)^2 L_{\ucSk{S:k:scale_interpolation}}^{-\alpha/2}\\
&\leqslant c H^2 L_{\ucSk{S:k:scale_interpolation}}^{-\alpha/2}\\
&\leqslant c H^2 H^{-5\alpha/7} &\mbox{using \eqref{S:turlututu}}\\
&\leqslant c H^{-\alpha/4} &\mbox{using that $\alpha\geqslant 5$.}
\end{align*}
By adjusting $c$ to accommodate small values of $H$, this concludes the proof of Lemma \ref{S:p:estimates_on_p}.

\section{Equality of the limiting directions: proof of Lemma \ref{S:l:v_-=v_+}}\label{S:ss:part2}

The goal of this section is to show Lemma \ref{S:l:v_-=v_+}. The heuristic idea behind the proof is the following. The definition of $v_-$ ensures that the random walk often has a direction close to $v_-$. On the other hand, the probability that the random walk has a direction larger than $v_++\varepsilon$ (where $\varepsilon>0$ is fixed) decreases quickly, as was shown in Lemma \ref{S:p:estimates_on_p}. Therefore, assuming by contradiction that $v_+>v_-$, the moments when its direction stays close to $v_-$ may prevent it from reaching a direction close to $v_+$ in the future, which would be a contradiction. However, the random walk might be able to compensate by going faster than $v_++ \varepsilon(H)$ for some well-chosen $\varepsilon(H)>0$. This is why we need precise estimates, and these will be given by a notion of trap that we will introduce further on.

We start by presenting a major property of our model, which comes from the coupling that we built. The choice of the coupling is actually made in order to get this property, which is inspired by the arguments from \cite{BHT}. Roughly, it says that random walks block each other in some weak sense: a random walk can always get around another random walk, but this happens with low probability.

\subsection{Barrier property}\label{S:ss:barrier_property}

\begin{proposition}\label{S:p:barrier_property}
    Let $x_0,x_0'\in\ZZ^2$ with $\pi_1(x_0)<\pi_1(x_0')$. Let $H\in\NN^*$ such that $H>\pi_2(x_0')-\pi_2(x_0)$. Let $\Gamma\in\mathcal{H}$ such that $\mathrm{Supp}\,\Gamma \cap \{Z_n^{x_0'},n\leqslant \tau^{x_0'}_{H,x_0}\}=\emptyset.$ Assume that $\tau^{x_0,\Gamma}_{H,x_0}<\infty$ and $\tau^{x_0'}_{H,x_0}<\infty$ (which happens almost surely). Then at least one of the following scenarios occurs:
    \begin{enumerate}
        \item $Z^{x_0,\Gamma}$ visits $x_0'+\{0\}\times (-\infty,0)$ before time $\tau^{x_0,\Gamma}_{H,x_0}$;
        \item $Z^{x_0'}$ visits $x_0+\{0\}\times (-\infty,0)$ before time $\tau^{x_0'}_{H,x_0}$;
        \item We have $X^{x_0,\Gamma}_{\tau_{H,x_0}}\leqslant X^{x_0'}_{\tau_{H,x_0}}.$
    \end{enumerate}
    The statement is also true when we replace $x_0$, $x_0'$ and $\Gamma$ by random variables satisfying the same assumptions.
\end{proposition}

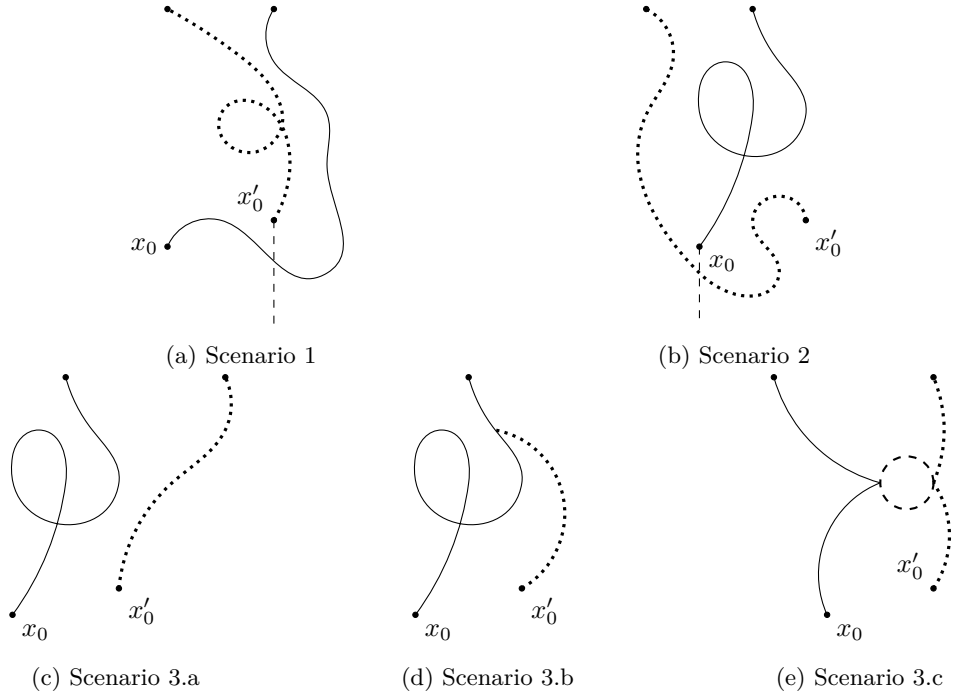
\begin{figure}[!h]
  \begin{center}
       \begin{subfigure}[b]{0.4\textwidth}
         \centering
        \begin{tikzpicture}[use Hobby shortcut, scale=0.7]
            \draw[dotted,very thick] (2,-1) .. (2.3,0) .. (2,1) .. (1,1) .. (2,0.5) .. (1.5, 2) .. (0,3) ;
            \draw[fill] (0,3) circle (.05);
            \draw[fill, above left] (2,-1) circle (.05) node {$x_0'$};
            \draw[dashed] (2,-1) -- (2,-3);
            \draw (0,-1.5) .. (1,-1) .. (3,-2) .. (3,0) .. (3,1) .. (2,2) .. (2,3);
            \draw[fill] (2,3) circle (.05);
            \draw[fill,left] (0,-1.5) circle (.05) node {$x_0$};
        \end{tikzpicture}
        \caption{Scenario 1}
     \end{subfigure}
     \begin{subfigure}[b]{0.4\textwidth}
         \centering
        \begin{tikzpicture}[use Hobby shortcut, scale=0.7]
            \draw[dotted,very thick] (2,-1) .. (1,-1) .. (1.5,-2) .. (0,-2) .. (-1,1) .. (-0.5,2) .. (-1,3) ;
            \draw[fill] (-1,3) circle (.05);
            \draw[fill, below right] (2,-1) circle (.05) node {$x_0'$};
            \draw[dashed] (0,-1.5) -- (0,-3) ;
            \draw (0,-1.5) .. (1,1) .. (0.5,2) .. (0,1.5) .. (2,1) .. (1.5,2) .. (1,3) ;
            \draw[fill] (1,3) circle (.05);
            \draw[fill,below right] (0,-1.5) circle (.05) node {$x_0$};
        \end{tikzpicture}
         \caption{Scenario 2}
     \end{subfigure}
     \begin{subfigure}[b]{0.3\textwidth}
         \centering
        \begin{tikzpicture}[use Hobby shortcut, scale=0.7]
            \draw[dotted,very thick] (2,-1) .. (3,1) .. (4,2) .. (4,3) ;
            \draw[fill] (4,3) circle (.05);
            \draw[fill, below right] (2,-1) circle (.05) node {$x_0'$};
            \draw (0,-1.5) .. (1,1) .. (0.5,2) .. (0,1.5) .. (2,1) .. (1.5,2) .. (1,3) ;
            \draw[fill] (1,3) circle (.05);
            \draw[fill,below right] (0,-1.5) circle (.05) node {$x_0$};
        \end{tikzpicture}
         \caption{Scenario 3.a}
     \end{subfigure}     
     \begin{subfigure}[b]{0.3\textwidth}
         \centering
        \begin{tikzpicture}[use Hobby shortcut, scale=0.7]
            \draw[dotted,very thick] (2,-1) .. (2.7,1) .. (1.5,2) ;
            \draw[fill, below right] (2,-1) circle (.05) node {$x_0'$};
            \draw (0,-1.5) .. (1,1) .. (0.5,2) .. (0,1.5) .. (2,1) .. (1.5,2) .. (1,3) ;
            \draw[fill] (1,3) circle (.05);
            \draw[fill,below right] (0,-1.5) circle (.05) node {$x_0$};
        \end{tikzpicture}
         \caption{Scenario 3.b}
     \end{subfigure}  
     \begin{subfigure}[b]{0.3\textwidth}
         \centering
        \begin{tikzpicture}[use Hobby shortcut, scale=0.7]
            \draw[dotted,very thick] (2,-1) .. (2.3,0) .. (2,1);
            \draw[dashed, thick] (2,1) .. (1.5,1.5) .. (1,1) .. (1.5,0.5) .. (2,1);
            \draw[dotted,very thick] (2,1) .. (2.2, 2) .. (2,3) ;
            \draw[fill] (2,3) circle (.05);
            \draw[fill, above left] (2,-1) circle (.05) node {$x_0'$};
            \draw (0,-1.5) .. (0.5,0.7) .. (1,1) ;
            \draw (1,1) .. (0,1.5) .. (-0.5,2) .. (-1,3);
            \draw[fill] (-1,3) circle (.05);
            \draw[fill,below right] (0,-1.5) circle (.05) node {$x_0$};
        \end{tikzpicture}
         \caption{Scenario 3.c}
     \end{subfigure}
        \caption{Illustrations of what can happen in each of the three scenarios of Proposition \ref{S:p:barrier_property}. In scenarios 1 and 2, one random walk goes around the other one. In scenarios 3, they end up with the same horizontal order because either they never meet nor go around each other (scenario 3.a), or they meet and stay together (scenario 3.b), or they meet on a loop that they both visit entirely before carrying on in the same situation as the initial one (scenario 3.c).}
        \label{S:f:barrier_property}
    \end{center}
\end{figure}

\begin{proof} Replacing $x_0$, $x_0'$ and $\Gamma$ by random variables does not change the proof, since the statement is deterministic. Even so, the proof is subtle. A lot of different things can happen, as is illustrated by Figure \ref{S:f:barrier_property}. We start by discussing the issues and ideas of the proof, in order to give some motivation.
\paragraph{Heuristics of the proof.} The main problem of the proof is that even if two random walks meet, they may not coalesce (\textit{i.e.} they may not stay together forever from then on), since the two random walks do not necessarily look at the same uniform variables every step of the way afterwards (for instance, if they are both located at a point that one has never visited before but the other one already has). However, we will show that when two random walks meet, either they coalesce (that is, at least restricted to the time intervals at hand), as in scenario 3.b), or they end up splitting up without having swapped their initial horizontal order, as in scenario 3.c). In the latter case, what actually happens is that the two random walks visit the same loop, and removing this loop is tantamount to adding the same history to both random walks. This prompts us to show a stronger version of the proposition by adding a common history $\Gamma_0$ to both random walks, which will allow us to apply our line of reasoning inductively by removing loops one by one.
\paragraph{Stronger claim.} We fix $x_0,x_0', H$ and $\Gamma$ as in the statement of the proposition, and we let $\Gamma_0\in\mathcal{H}$. 
From now on, we use simpler notations: $\tau$ for $\tau_{H,x_0}^{x_0,\Gamma+\Gamma_0}$, $\tau'$ for $\tau_{H,x_0}^{x_0',\Gamma_0}$, $Z$ for $Z^{x_0,\Gamma+\Gamma_0}\vert_{\llbracket 0,\tau\rrbracket}$ and $Z'$ for $Z^{x_0',\Gamma_0}\vert_{\llbracket 0,\tau'\rrbracket}$. Assume that both $\tau$ and $\tau'$ are finite. We argue by contradiction and assume that
\begin{align}\label{S:e:ass_for_contradiction}
\left\{\begin{array}{ll}
    \text{$Z$ does not visit $x_0'+\{0\}\times (-\infty,0)$}& (1) \\
    \text{$Z'$ does not visit $x_0+\{0\}\times (-\infty,0)$}& (2)\\
    \text{$X_\tau> X'_{\tau'}$}& (3)
\end{array}\right.\end{align}
We want to get to a contradiction from this, and then choosing $\Gamma_0=\mathbf{0}$ implies the desired result.

\paragraph{Loop removal algorithm.} We now define an algorithm allowing us to remove all loops from the paths of our random walks, so that we can focus on the zero-loop case later on. We consider a path defined by a parametrization $f:P\rightarrow \ZZ^2$, where $P$ is a bounded subset of $\NN$, satisfying for every $s,t\in P$, $|t-s|=1\Rightarrow \vert f(t)-f(s)\vert=1$, where $\vert\cdot\vert$ is the Euclidean norm on $\RR^2$. If $f$ is not injective, we define
    $$\begin{array}{l}
    T^{out}_1(f)=\min\left\{t\in P,\,f(t)\in \{f(s),s\in P,s<t\}\right\};\\
    T^{in}_1(f)=\min\left\{t\in P,\,f(t)=f(T_1^{out}(f))\right\};\\
    P_1(f)=\llbracket T_1^{in}(f),T_1^{out}(f)-1\rrbracket\cap P;\\
    L_1(f)=f(P_1(f)).
    \end{array}
    $$
    We call $L_1(f)$ the first loop of $f$. The times $T^{in}_1(f)$ and $T^{out}_1(f)$ are called the entry and exit times of $L_1(f)$. We define by induction the other loops of $f$, if they exist, by defining, for $i\geqslant 2$,
    $$\begin{array}{l}
    T^{out}_{i}(f)=T^{out}_1\left(f\vert_{P\setminus \cup_{j<i} P_j(f)}\right);\\
    T^{in}_i(f)=T^{in}_1\left(f\vert_{P\setminus \cup_{j<i} P_j(f)}\right);\\
    P_i(f)=\llbracket T_i^{in}(f),T_i^{out}(f)-1\rrbracket\cap\left(P\setminus\cup_{j<i} P_j(f)\right)\\
    L_i(f)=f(P_i(f)).
    \end{array}
    $$
    Mind that here we consider functions defined on subsets of $P$ that are not necessarily connected in $P$, which is why we did not assume $P$ to be connected in $\NN$ in the first place.
    
    If there are no more loops, we just set $P_i(f)=\emptyset$, $L_i(f)=\emptyset$ and $T_i^{in}(f)=T_i^{out}(f)=\infty$.
    
    We also define, for such a function $f$, its interpolated sample path as the curve in $\RR\times\ZZ\,\cup\,\ZZ\times\RR$ obtained by joining each pair of points $\{f(t),f(t+1)\}$ (for $t$ and $t+1\in P$) by a segment. We denote it by $\mathrm{int}(f)$.

\paragraph{The two sample paths meet.} Recall our assumption \eqref{S:e:ass_for_contradiction}, from which we want to get to a contradiction. Let $C=\mathrm{int}\left(Z\vert_{\llbracket 0,\tau\rrbracket}\right)$ and $C'=\mathrm{int}\left(Z'\vert_{\llbracket 0,\tau'\rrbracket}\right)$. For now, we want to show that $C$ and $C'$ meet at some point of $\RR^2$, which implies, by construction, that the two sets $\{Z_n,\,0\leqslant n\leqslant \tau\}$ and $\{Z'_n,\,0\leqslant n\leqslant \tau'\}$ meet at some point of $\ZZ^2$. To show that, we first form a closed simple curve $C_0$ of $\RR^2$ as shown in Figure \ref{S:f:barrier_property_2}. First we consider $$C'_*=\mathrm{int}\left(Z'\vert_{\llbracket 0,\tau'\rrbracket\setminus\cup_{i\in\NN^*} P_i(Z)}\right),$$ which is $C'$ from which we removed all the loops and which we interpolated. Then we join the two extreme points $x_0'$ and $x_1'$ of $C'$ (note that they both have to be on $C'_*$) using horizontal and vertical segments that go low enough and left enough so that they do not meet $C\cup C'$ except at $x_0'$ and $x_1'$ (if $C'_*$ intersects $x_0'+\{0\}\times(-\infty,0)$, we remove the initial part of $C'_*$ so that $x_0'$ is replaced by the lowest point on $C'_*\cap (x_0'+\{0\}\times(-\infty,0))$). This is possible because $C\cup C'$ is a compact set and because of $(1)$ and $(3)$ in \eqref{S:e:ass_for_contradiction}. By construction, $C_0$ is a closed simple curve, so we can apply Jordan's theorem to $C_0$. Point $x_1=Z_{\tau}$ is in the unbounded component, because the half-line $x_1+(0,\infty)\times\{0\}$ cannot meet $C_0$. On the contrary, point $x_0$ has to be in the bounded component, because the vertical segment joining $x_0$ to a lower point $x_2$ in the unbounded component meets $C_0$ only once, because of $(2)$ in \eqref{S:e:ass_for_contradiction} {\color{black}(here we use the sometimes called even-odd rule that can be found in \cite{Jordan})}. Therefore, curve $C$ has to meet $C_0$, and by construction of $C_0$, it has to meet it on $C'_*$. Therefore $C$ and $C'$ intersect.
    
\begin{figure}[!h]
  \begin{center}
        \begin{tikzpicture}[use Hobby shortcut, scale=0.8]
            \draw[dotted,very thick] (2,-1) .. (2.3,0) .. (1.8,1);
            \draw[dotted] (1.8,1).. (1,1) .. (1.5,0.5) .. (1.8,1);
            \draw[dotted,very thick] (1.8,1).. (1.5, 2) .. (0,3) ;
            \draw[fill] (0,3) circle (.05);
            \draw[fill, above left] (2,-1) circle (.05) node {$x_0'$};
            \draw[dashed] (2,-1) -- (2,-2) -- (-2,-2) -- (-2,3) -- (0,3);
            \draw (0,-1.5) .. (-0.5,-1) .. (1,2) .. (2,3);
            \draw[fill] (2,3) circle (.05);
            \draw[fill,right] (0,-1.5) circle (.05) node {$x_0$};
            \draw[left] (-2,1) node {$C_0$};
            \draw[above left] (0.2,-1) node {$C$};
            \draw[right] (2.3,0.2) node {$C'_*$};
            \draw[right] (2,3) node {$x_1$};
            \draw[above] (0,3) node {$x_1'$};
            \draw[fill, right] (0,-2.3) circle (.05) node {$x_2$};
        \end{tikzpicture}
        \caption{}
        \label{S:f:barrier_property_2}
    \end{center}
\end{figure}

\paragraph{Zero-loop case.} First consider the simpler case where $C'$ has no loops intersecting $C$. By the previous point, $C'$ meets $C$, so we can consider $$\hat{t}=\max\left\{t\in \llbracket 0,\tau'\rrbracket,\,Z'_t\in C\right\},$$ and $\hat{x}=Z'_{\hat{t}}$. Because of $(3)$ in \eqref{S:e:ass_for_contradiction}, $\hat{t}<\tau'$. The uniform variable that $Z'$ uses to jump at time $\hat{t}$ is $U(\hat{x},\Gamma_0(\hat{x})+1)$ (for there is no loop on $C'$ intersecting $C$, so $\hat{x}$ cannot be on a loop of $C'$). The same uniform variable is used by $Z$ when it gets to $\hat{x}$ for the first time, since by assumption $\mathrm{Supp}\,\Gamma\cap C'=\emptyset$. Therefore, $Z'_{\hat{t}+1}\in C$, which contradicts the definition of $\hat{t}$. At the end of the day, we have a contradiction, so our assumption \eqref{S:e:ass_for_contradiction} was false.

\paragraph{General case.} Consider now the case where $C'$ has a loop intersecting $C$. For $i\geqslant 1$, let $P_i=P_i(Z')$, $L_i=L_i(Z')$, $T_i^{in}=T_i^{in}(Z')$ and $T_i^{out}=T_i^{out}(Z')$. We let $$n_1=\min\left\{i\in \NN^*, L_i \cap C\neq\emptyset\right\}.$$
    Let $T$ be the first hitting time of $L_{n_1}$ by $Z$ and $T'\in \llbracket T^{in}_{n_1},T^{out}_{n_1}-1\rrbracket$ be the first time such that $Z'_{T'}=Z_T$. Let us show by induction on $t\leqslant T_{n_1}^{out}-T'$ that
    \begin{align}\label{S:e:first_equality}
    \left(Z_T,Z_{T+1},\ldots,Z_{T+t}\right)=\left(Z'_{T'},Z'_{T'+1},\ldots, Z'_{T'+t}\right).
    \end{align}
    The case $t=0$ follows from the fact that $Z'_{T'}=Z_T$. Suppose \eqref{S:e:first_equality} is true for $t<T_{n_1}^{out}-T'$. We need to show that $Z_{T+t+1}=Z'_{T'+t+1}.$
    
    \begin{itemize}[leftmargin=*]
    \item First, note that $Z'$ cannot have visited $Z'_{T'+t}$ before time $T'+t$. Indeed, suppose that it has; then there exists $i<n_1$ such that $Z'_{T'+t}\in L_i$, now $Z'_{T'+t}=Z_{T+t}$ so $L_i\cap C\neq \emptyset$, which contradicts the definition of $n_1.$ Therefore, the uniform variable $Z'$ uses to jump at time $T'+t$ is $U(Z'_{T'+t},\Gamma_0(Z'_{T'+t})+1)$.
    \item This also means that $Z'_{T'+t}=Z_{T+t}$ is not among $\{Z'_{T'},\ldots, Z'_{T'+t-1}\}$, which is the same set as $\{Z_T,\ldots, Z_{T+t-1}\}$ by the induction assumption. Therefore, using also the definition of $T$, $Z$ has not visited site $Z_{T+t}$ before time $T+t$, so the uniform variable it uses to jump at time $T+t$ is $U(Z_{T+t},\Gamma_0(Z_{T+t})+1)=U(Z'_{T'+t},\Gamma_0(Z'_{T'+t})+1)$ (we also use the fact that $\mathrm{Supp}\,\Gamma\cap C'=\emptyset$).
    \end{itemize}
    
    Both random walks use the same uniform variable, therefore $Z_{T+t+1}=Z'_{T'+t+1}$, which shows \eqref{S:e:first_equality} for $t+1$ and ends the induction. Applying equality \eqref{S:e:first_equality} with $t=T_{n_1}^{out}-T'$ yields
    \begin{align}\label{S:e:second_equality}
        \left(Z_T,Z_{T+1},\ldots,Z_{T+T_{n_1}^{out}-T'}\right)=\left(Z'_{T'},Z'_{T'+1},\ldots, Z'_{T_{n_1}^{out}}\right).
    \end{align}
    With the same arguments, we can show that we also have
    \begin{align}\label{S:e:third_equality}
        \left(Z_{T+T_{n_1}^{out}-T'},\ldots, Z_{T+T_{n_1}^{out}-T_{n_1}^{in}}\right)=\left(Z'_{T_{n_1}^{in}},Z'_{T_{n_1}^{in}+1},\ldots,Z'_{T'}\right).
    \end{align}
    Also, remark that in $(Z'_{T^{in}_{n_1}},\ldots, Z'_{T^{out}_{n_1}-1})$, we have $T^{out}_{n_1}-T^{in}_{n_1}+1$ distinct points of $L_{n_1}$ (by definition of $T_{n_1}^{out}$), so we have all the points in $L_{n_1}$ exactly once.
    Therefore, putting together \eqref{S:e:second_equality} and \eqref{S:e:third_equality}, and considering that $Z'_{T_{n_1}^{in}}=Z'_{T_{n_1}^{out}}$, we see that between times $T$ and $T+T_{n_1}^{out}-T_{n_1}^{in}$, $Z$ visits all the sites in $L_{n_1}$ exactly once too.
    
Set $\Gamma_1=\Gamma_0+\sum_{x\in L_{n_1}} \delta_{\{x\}}$ and $\tilde{P}_1=\llbracket T, T+T_{n_1}^{out}-T_{n_1}^{in}-1\rrbracket.$ Separating what happens before time $T$ resp. $T_{n_1}^{in}$ and what happens after time $T+T_{n_1}^{out}-T_{n_1}^{in}$ resp. $T_{n_1}^{out}$, we get
    $$\begin{array}{l}
    \left(Z^{x_0,\Gamma+\Gamma_1}_n,\,n\in\llbracket 0, \tau_{H,x_0}^{x_0,\Gamma+\Gamma_1}\rrbracket\right)=\left(Z^{x_0,\Gamma+\Gamma_0}_n,\, n\in \llbracket 0,\tau_{H,x_0}^{x_0,\Gamma+\Gamma_0}\rrbracket\setminus \tilde{P}_1\right);\\
    \left(Z^{x_0',\Gamma_1}_n,\, n\in\llbracket 0,\tau_{H,x_0}^{x_0',\Gamma_1}\rrbracket\right)=\left(Z^{x_0',\Gamma_0}_n,\,n\in\llbracket 0, \tau_{H,x_0}^{x_0',\Gamma_0}\rrbracket\setminus P_{n_1}\right).
    \end{array}$$
    Let $K$ be the number of loops of $C'$ that intersect $C$ ($K$ is finite). By applying the same line of reasoning inductively on the next loops of $C'$ that intersect $C$ (which we denote by $L_{n_2},\ldots,L_{n_K}$), we can construct a history $\Gamma_K$ such that $$\begin{array}{l}
    \left(Z^{x_0,\Gamma+\Gamma_K}_n,\,n\in\llbracket 0,\tau_{H,x_0}^{x_0,\Gamma+\Gamma_K}\rrbracket\right)=\left(Z^{x_0,\Gamma+\Gamma_0}_n,\,n\in\llbracket 0,\tau_{H,x_0}^{x_0,\Gamma+\Gamma_0}\rrbracket\setminus (\tilde{P}_1\cup\ldots \cup \tilde{P}_K)\right);\\
    \left(Z^{x_0',\Gamma_K}_n,\,n\in\llbracket 0,\tau_{H,x_0}^{x_0',\Gamma_K}\rrbracket\right)=\left(Z^{x_0',\Gamma_0}_n,\,n\in\llbracket 0,\tau_{H,x_0}^{x_0',\Gamma_0}\rrbracket\setminus (P_{n_1}\cup\ldots\cup P_{n_K})\right).
    \end{array}$$
    Therefore, by construction, the sample path of $Z^{x_0',\Gamma_K}$ on $\llbracket 0,\tau_{x_0,H}^{x_0',\Gamma_K}\rrbracket$ has no loop intersecting that of $Z^{x_0,\Gamma+\Gamma_K}$ on $\llbracket 0,\tau_{x_0,H}^{x_0,\Gamma+\Gamma_K}\rrbracket$, so we can apply the previous zero-loop case by replacing $\Gamma_0$ by $\Gamma_K$. Assumptions from \eqref{S:e:ass_for_contradiction} are still satisfied, and $\mathrm{Supp}\,\Gamma_K \cap \{Z_n^{x_0',\Gamma_K},\,n\in\NN\}=\emptyset$ by construction, so we do recover a contradiction.
\end{proof}

\subsection{Trapped points}\label{S:ss:traps}
Let us move on to the proof of Lemma \ref{S:l:v_-=v_+}. Recall that $v_-\leqslant v_+$, on account of Corollary \ref{S:c:v_-<v_+}, so we now argue by contradiction and assume that $v_-<v_+$. In the rest of this section, we set  
\begin{align}\label{S:d:delta}
\delta=\frac{v_+-v_-}{5(\beta+1)}.
\end{align}
Note that $\delta \in (0,1/2],$ using the bounds in \eqref{S:e:bounds_on_v_+}.

The crucial idea of our proof is given by Proposition \ref{S:p:barrier_property}, which implies that a random walk can be "trapped" by another random walk. We want to ensure that trapped random walks will experience a delay with respect to $v_+$, which motivates Definition \ref{S:d:traps} below.

Let $H\in\NN^*$ and $w\in\R\times\ZZ.$ Recall notation $H'$ from \eqref{S:n:H'}. We define
\begin{align}\label{S:d:z_w}
    z_w=w+(\delta H+4\beta H',2H')\in \RR\times\ZZ.
\end{align}

\begin{definition}[Trap]\label{S:d:traps}
We say that $w$ is $H$-trapped if there exists $y\in I_{\delta H}(z_w)$ such that:
\begin{enumerate}
    \item $V_{H-2H',z_w}^y\leqslant v_-+\delta$;
    \item For every $n\in\llbracket 0,\tau_{H-2H',z_w}^y\rrbracket$, $Y^y_n\geqslant\pi_2(w)+H'$.
\end{enumerate}
\end{definition}

The term "trap" will become clearer in Section \ref{S:sss:delay_trap}. See Figure \ref{S:f:traps} for an illustration of Definition \ref{S:d:traps}.

\begin{remark}\label{S:r:trap_mesurability}
Note that the definition of a trap implies that we can define an algorithm to decide if $w$ is $H$-trapped or not, only looking at the environment and the uniform variables between heights $\pi_2(w)+H'$ and $\pi_2(w)+H.$
\end{remark}

\begin{remark}\label{S:r:condition2}
    Condition 2 in Definition \ref{S:d:traps} is satisfied whenever $E_{H'}^y$ occurs. However it is more convenient to use condition 2, otherwise Remark \ref{S:r:trap_mesurability} would not hold, and it is sufficient, since we are not interested in the behaviour of $Z^y$ after time $\tau_{H-2H',z_w}^y.$
\end{remark}

\subsubsection{Probability of being trapped}
Of course, we will not be able to show that a point is trapped with a high probability, for point 1 in Definition \ref{S:d:traps} is very demanding. However, the definition of $v_-$ will allow us to show that we can reach any direction close to but at least $v_-$ with a positive probability, so we will be able to get a uniform lower bound on the probability of being trapped. This is what the following lemma expresses. Recall the definition of $\mathbf{H}_0$ from Remark \ref{S:r:mixing_boxes}.

\begin{lemma}\label{S:l:bound_trap}
\ncSH{H:traps}
There exists an integer constant $\ucSH{H:traps}\geqslant \ucSH{H:size_box}$, depending on $v_-$ and $v_+$, such that \begin{align*}
    \inf_{H\geqslant \ucSH{H:traps}} \inf_{w\in \R\times\ZZ} \PP(w\text{ is }H\text{-trapped})>0.    
    \end{align*}
\end{lemma}

\begin{proof}\ncS{S:c:numb}\ncS{S:c:prob2}\ncS{S:c:prob1} \ncS{S:c:prob} \ncS{S:c:inf}

Let us first study condition 1 in Definition \ref{S:d:traps}. Recall notation $\tilde{p}$ from Definition \ref{S:d:limiting_speeds}. We claim that there exist two positive constants $\ucS{S:c:numb}$ and $\ucS{S:c:prob1}$ such that for $H$ large enough,
\begin{align}\label{S:tetete}
    \inf_{w\in\R\times\ZZ} \PP\left(\exists\,y\in I_{\delta H}(z_w),\;V_{H-2H',z_w}^y\leqslant v_-+\delta\right)\geqslant \ucS{S:c:numb}^{-1}\,\tilde{p}_{H}(v_-+\delta/4) -\ucS{S:c:prob2}^{-1}e^{-\ucS{S:c:prob2} H^{1/2}}.
\end{align}

Let us prove this claim. Let us fix $H>4/\delta$. We have
\begin{align*}
    &\inf_{w\in\R\times\ZZ} \PP\left(\exists\,y\in I_{\delta H}(z_w),\;V_{H-2H',z_w}^y\leqslant v_-+\delta\right)\\
    &=\inf_{w\in\R\times\ZZ} \PP\left(\exists\,y\in I_{\delta H}(w),\;V_{H-2H',w}^y\leqslant v_-+\delta\right)\\
    &=\inf_{w\in[-1,0)\times\{0\}} \PP\left(\exists\,y\in I_{\delta H}(w),\;V_{H-2H',w}^y\leqslant v_-+\delta\right)\\
    &\geqslant \sup_{w\in[0,1)\times\{0\}} \PP\left(\exists\,y\in I_{\delta H/2}(w),\;V_{H-2H',w}^y\leqslant v_-+\delta\right)\\
    &= \sup_{w\in \RR\times\ZZ} \PP\left(\exists\,y\in I_{\delta H/2}(w),\;V_{H-2H',w}^y\leqslant v_-+\delta\right).
\end{align*}
In the first equality, we used that $w\mapsto z_w$ is a bijection of $\RR\times\ZZ$. In the second and last equalities, we used Corollary \ref{S:c:invariance}. In the inequality, we used that since $\delta H>4$, for any $w\in[-1,0)\times\{0\}$ and $w'\in [0,1)\times\{0\},$ $I_{\delta H/2}(w')$ is included in $I_{\delta H}(w)$.

Now, we want to replace $H-2H'$ by $H$ in the parameter of the direction. Indeed, the information we have on $v_-$ is a liminf when $H$ goes to infinity, and it could be that we are unlucky and this liminf is the limit on a subsequence that is not eventually in the image of $H\mapsto H-2H'$. In order to do this, we work on
$$\bigcap_{y\in I_{\delta H/2}(w)} \Theta_{H-2H'}^w D^y_{2H'},$$
which, using Propositions \ref{S:p:independence} and \ref{S:p:horizontal_bounds}, has probability at least $1-\ucS{S:c:prob1}^{-1}e^{-\ucS{S:c:prob1}H^{1/2}}$, where $\ucS{S:c:prob1}$ is a positive constant that does not depend on $H$. On this event, provided that $2\left(\beta+\left(v_-+\delta/2\right)\right) H'\leqslant \delta (H-2H')/2$, we have that if $V^y_{H,w}\leqslant v_-+\delta/2$, then
\begin{align*}
    X^y_{\tau_{H-2H',w}}-\pi_1(y)
    &\leqslant \left(X^y_{\tau_{H,w}}-\pi_1(y)\right)-\left(X^y_{\tau_{H,w}}-X^y_{\tau_{H-2H',w}}\right)\\
    &\leqslant \left(v_-+\delta/2\right)H+2\beta H'\\
    &\leqslant \left(v_-+\delta/2\right)(H-2H')+2\left(\beta+\left(v_-+\delta/2\right)\right) H'\\
    &\leqslant (v_-+\delta) (H-2H'),
\end{align*}
or, in other words, $V^y_{H-2H',w}\leqslant v_-+\delta$. Therefore, for $H$ large enough,
\begin{align*}
    &\sup_{w\in \RR\times\ZZ} \PP\left(\exists\,y\in I_{\delta H/2}(w),\;V_{H-2H',w}^y\leqslant v_-+\delta\right)\\
    &\geqslant \sup_{w\in \RR\times\ZZ} \PP\left(\exists\,y\in I_{\delta H/2}(w),\;V_{H,w}^y\leqslant v_-+\delta/2\right)-\ucS{S:c:prob1}^{-1}e^{-\ucS{S:c:prob1}H^{1/2}}.
\end{align*}
Now, in order to recover parameter $H$ in the size of the rectangle too, we consider $I_{H}$ and split it into rectangles $I_{\delta H/2}(w)$ for some number $\ucS{S:c:numb}$ (which does not depend on $H$) of values of $w\in\RR\times\ZZ$ satisfying $0\leqslant \pi_2(w)<H'$. Let us fix such a $w$ and $y\in I_{\delta H/2}(w)$. In order to link $V^y_{H}$ with $V^y_{H,w}$, we work on $$\bigcap_{y\in I_{H}} \Theta_H D_{H'}^y,$$
which, using Propositions \ref{S:p:independence} and \ref{S:p:horizontal_bounds}, has probability at least $1-\ucS{S:c:prob}^{-1}e^{-\ucS{S:c:prob}H^{1/2}}$ for some constant $\ucS{S:c:prob}>0$ that does not depend on $H$. On this event, the displacement of $Z^y$ between times $\tau_{H}^y$ and $\tau_{H,w}^y$ is at most $\beta H'$, which is at most $\delta H/4$ for $H$ large enough. In the end, using a union bound, we have
\begin{align*}
    \sup_{w\in \RR\times\ZZ} &\PP\left(\exists\,y\in I_{\delta H/2}(w),\;V_{H,w}^y\leqslant v_-+\delta/2\right) \nonumber\\
    &\geqslant \ucS{S:c:numb}^{-1}\, \left(\PP\left(\exists\,y\in I_{H},\;V_{H}^y\leqslant v_-+\delta/4\right)-\ucS{S:c:prob}^{-1}e^{-\ucS{S:c:prob} H^{1/2}}\right)\nonumber\\
    &= \ucS{S:c:numb}^{-1}\, \left(\tilde{p}_{H}(v_-+\delta/4) -\ucS{S:c:prob}^{-1}e^{-\ucS{S:c:prob} H^{1/2}}\right).
\end{align*}

Putting all inequalities together, we derive our claim \eqref{S:tetete} with $\ucS{S:c:prob2}$ depending on $\ucS{S:c:numb}$, $\ucS{S:c:prob1}$ and  $\ucS{S:c:prob}$. Now, $\liminf_{H\to\infty} \tilde{p}_{H}(v_-+\delta/4)>0,$
since $v_-+\delta/4>v_-$ (recall Definition \ref{S:d:limiting_speeds}). Therefore, using \eqref{S:tetete},
$$\liminf_{H\to\infty} \inf_{w\in\R\times\ZZ} \PP\left(\exists\,y\in I_{\delta H}(z_w),\;V_{H-2H',z_w}^y\leqslant v_-+\delta\right)>0.$$
This implies that for any large enough $\ucSH{H:traps}$, we have
$$\ucS{S:c:inf}:=\inf_{H\geqslant \ucSH{H:traps}} \inf_{w\in\R\times\ZZ} \PP\left(\exists\,y\in I_{\delta H}(z_w),\;V_{H-2H',z_w}^y\leqslant v_-+\delta\right)>0.$$

In the end, for every $H\geqslant \ucSH{H:traps}$ and $w\in\R\times\ZZ,$ we have, using Remark \ref{S:r:condition2} and Proposition \ref{S:p:probability_of_going_down},
$$\PP(\text{$w$ is $H$-trapped})\geqslant \ucS{S:c:inf}-cH^{3/2} e^{-\ucS{S:c:proba_of_going_down}H^{1/2}}\geqslant \frac{\ucS{S:c:inf}}{2}>0,$$
provided that $\ucSH{H:traps}$ is large enough.
\end{proof}

We stated the above result as a lemma because it will later appear as a mere first step towards a stronger result, Proposition \ref{S:p:proba_threats}. The same holds for the next lemma, which is the first step towards Proposition \ref{S:p:delay_threats}.

\subsubsection{Delay near a trapped point}\label{S:sss:delay_trap}

The following lemma explains why the name "trap" was chosen: heuristically speaking, when we start a random walk near an $H$-trapped point $w$, with high probability it is delayed by the time it reaches height $\pi_2(w)+H$. Indeed, the trap gives a sample path that acts as a barrier, which pushes our random walk to the left unless it gets around the barrier from below. See Figure \ref{S:f:traps} for an illustration of this idea.

\ncSH{H:delay_traps}
Recall the definition of $\ucSH{H:traps}$ in Lemma \ref{S:l:bound_trap}. For the next lemma, we need another technical requirement on $H$. Note that, since $\beta \delta>0$, there exists $\ucSH{H:delay_traps}\geqslant\ucSH{H:traps}$, which depends on $v_-$ and $v_+$, such that
\begin{align}\label{S:d:H2}
    \forall H\geqslant \ucSH{H:delay_traps},\;-5\beta\delta H+(4\beta-2(v_-+\delta))H'\leqslant 0.
\end{align}

\begin{lemma}\label{S:l:delay_traps}
Let $H\geqslant \ucSH{H:delay_traps}$, $w\in\RR\times\ZZ$, $x_0\in w+(-\infty,\delta H)\times [0,H')$ and $\Gamma\in\mathcal{H}$ whose support satisfies $\pi_2(\mathrm{Supp}\,\Gamma) \subseteq (-\infty, \pi_2(w)+H')$. Suppose that
\begin{enumerate}
    \item $w$ is $H$-trapped;
    \item $D^{x_0,\Gamma}_{4H'}$ occurs;
    \item $\Theta_{4H'}^{x_0} E^{x_0,\Gamma}_{H'}$ occurs.
\end{enumerate}
Then, we have $$X^{x_0,\Gamma}_{\tau_{H,w}}\leqslant  \pi_1(w)+(v_+-2\delta)H.$$
Again, the statement is also true when we replace $x_0$, $w$ and $\Gamma$ by values of random variables satisfying the same assumptions.
\end{lemma}

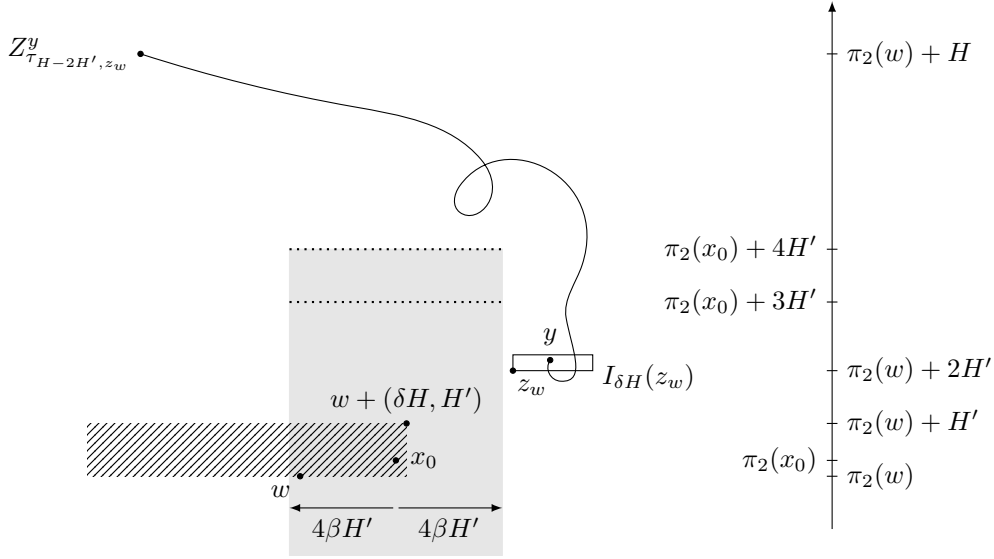
\begin{figure}[!h]
  \begin{center}
    \begin{tikzpicture}[use Hobby shortcut,scale=0.7]
        \draw[fill,black!10!white] (-1.2,-2.5) rectangle (2.8,3.3);
        \path[pattern=north east lines] (-5,0) rectangle (1,-1);
        \draw[thick,dotted] (-1.2,2.3) -- (2.8,2.3);
        \draw[thick, dotted] (-1.2,3.3) -- (2.8,3.3);
        \draw (3,1) rectangle (4.5,1.3);
        \draw[fill, below right] (3,1) circle (.05);
        \draw[fill, below right] (2.9,1) circle node {$z_w$};
        \draw[below right] (4.5,1.3) node {$I_{\delta H}(z_w)$};
        \draw[above] (3.7,1.3)  node {$y$};
        \draw[fill, right] (3.7,1.2) circle (.05);
        \draw[fill, below left] (-1,-1) circle (.05) node {$w$};
        \draw[fill, above] (1,0) circle (.05) node {$w+(\delta H,H')$};
        \draw[fill, below right] (0.8,-0.7) circle (.05);
        \draw[fill, right] (0.9,-0.7) node {$x_0$};
        \draw[->,>=latex] (0.85,-1.6) -- (2.8,-1.6);
        \draw[below] (1.8,-1.6) node {$4\beta H'$};
        \draw[->,>=latex] (0.75,-1.6) -- (-1.2,-1.6);
        \draw[below] (-0.2,-1.6) node {$4\beta H'$};
        \draw (3.7,1.2) .. (3.9,0.8) .. (4,2) .. (4.3,3) .. (2,4.5) .. (2,4) .. (2.4,5) .. (0,6) .. (-4,7);
        \draw[fill, left] (-4,7) circle (.05) node {$Z^y_{\tau_{H-2H',z_w}}$};
        \draw[->,>=latex] (9,-2) -- (9,8);
        \draw[right] (8.9,7) -- (9.1,7) node {$\pi_2(w)+H$};
        \draw[left] (9.1,3.3) -- (8.9,3.3) node {$\pi_2(x_0)+4H'$};
        \draw[left] (9.1,2.3) -- (8.9,2.3) node {$\pi_2(x_0)+3H'$};
        \draw[right] (8.9,1) -- (9.1,1) node {$\pi_2(w)+2H'$};
        \draw[right] (8.9,0) -- (9.1,0) node{$\pi_2(w)+H'$};
        \draw[right] (8.9,-1) -- (9.1,-1) node{$\pi_2(w)$};
        \draw[left] (9.1,-0.7) -- (8.9,-0.7) node{$\pi_2(x_0)$};
    \end{tikzpicture}
    \caption{Illustration of $w$ being $H$-trapped. The line connecting $(\pi_1(y),\pi_2(z_w))$ and $Z^y_{\tau_{H-2H',z_w}}$ has direction at most $v_-+\delta$. A random walk starting at $x_0$ in the hatched area, is delayed by the sample path given by $Z^y$, as soon as it stays in the grey area and does not go back under the lower dotted line once it has reached the upper one.}
    \label{S:f:traps}
    \end{center}
\end{figure}

\begin{proof}
Let $H$, $w$, $x_0$ and $\Gamma$ be as in the statement of the lemma. Suppose $w$ is $H$-trapped. By definition, there exists $y\in I_{\delta H}(z_w)$ such that $V_{H-2H',z_w}^{y}\leqslant v_-+\delta$ and
\begin{align}\label{S:a:Y^y}
    \forall n\in\llbracket 0,\tau_{H-2H',z_w}^y\rrbracket,\;Y^y_n\geqslant\pi_2(w)+H'.
\end{align}
Let us apply Proposition \ref{S:p:barrier_property} with $x_0$, $\Gamma$ and $x_0'=y$, after replacing $H$ by $H-(\pi_2(x_0)-\pi_2(w))$.
\begin{itemize}[leftmargin=*]
    \item Using \eqref{S:a:Y^y} and the assumption on the support of $\Gamma$, we have $\mathrm{Supp}\,\Gamma\,\cap\,\{Z_n^{y},n\leqslant \tau^y_{H,w}\}=\emptyset.$ By \eqref{S:a:Y^y}, we also note that $Z^{y}$ cannot visit $x_0+\{0\}\times (-\infty,0)$ before time $\tau^y_{H,w}=\tau^y_{H-2H',z_w}.$
    \item Assumptions 2 and 3 in the proposition ensure that $Z^{x_0,\Gamma}$ does not visit $y+\{0\}\times (-\infty,0)$ before time $\tau_{H,x_0}^{x_0,\Gamma}$ either (see Figure \ref{S:f:traps} for an illustration of this fact), since $\pi_1(y)-\pi_1(x_0)>4\beta H'$ by construction.
\end{itemize}
Therefore we must have
\begin{align*}
X_{\tau_{H,w}}^{x_0,\Gamma}\leqslant X^y_{\tau_{H-2H',z_w}}
&=X_{\tau_{H,w}}^{y}\\
&\leqslant \pi_1(y)+(v_-+\delta)(H-2H')&\mbox{since $V_{H-2H',z_w}^{y}\leqslant v_-+\delta$}\\
&\leqslant \pi_1(z_w)+\delta H+(v_-+\delta)(H-2H')&\mbox{since $y\in I_{\delta H}(z_w)$}\\
&\leqslant \pi_1(w)+2\delta H+4\beta H'+(v_-+\delta)(H-2H')&\mbox{using \eqref{S:d:z_w}}\\
&=\pi_1(w)+(v_-+3\delta)H+(4\beta-2(v_-+\delta)H'\\
&=\pi_1(w)+(v_+-(5\beta+2)\delta) H +(4\beta-2(v_-+\delta))H'&\mbox{using \eqref{S:d:delta}}\\
&\leqslant \pi_1(w)+(v_+-2\delta)H&\mbox{using \eqref{S:d:H2}.}
\end{align*}
\end{proof}

The interest of traps becomes clear with Lemma \ref{S:l:delay_traps}. The issue however is that the probability of being trapped cannot be made arbitrarily close to $1$ when $H$ goes to infinity; we only know, thanks to Lemma \ref{S:l:bound_trap}, that it is uniformly positive. Therefore, we need to introduce another notion in which we will allow some entropy on where to find a trap.

\subsection{Threatened points}\label{S:ss:threats}
The problem with traps is that the probability of being trapped may be very small; however we will see that it is sufficient to have a trapped point along a line segment of slope $v_+$ in order to experience the delay, which motivates the new definition below.

\begin{definition}[Threat]\label{S:d:threats}
Let $H\in\NN^*$, $r\in\NN^*$ and $w\in\R\times\ZZ.$ $w$ is said to be $(H,r)$-threatened if one of the points $w_j=w+jH(v_+,1)$, where $j\in\llbracket 0,r-1\rrbracket$, is $H$-trapped.
\end{definition}

\subsubsection{Probability of being threatened}

When $r$ increases (keep in mind that $rH$ is the vertical length of the line segment along which we look for trapped points), it is clear that the probability that $w$ is threatened increases. We now show that it goes to $1$ when $r\to\infty$, and quantify the convergence using $\alpha.$ This is the major asset of the notion of threats. Recall constant $\ucSH{H:traps}$ from Lemma \ref{S:l:bound_trap}.

\begin{proposition}\label{S:p:proba_threats}
\ncS{S:c:threats}
There exists $\ucS{S:c:threats}=\ucS{S:c:threats}(\delta)>0$ such that for every $H\geqslant \ucSH{H:traps}$ and $r\in\NN^*$, \begin{align}\label{S:e:estimate_final_threat}
\sup_{w\in\RR\times\ZZ}\PP(w\text{ is not $(H,r)$-threatened})\leqslant \ucS{S:c:threats}\,r^{-\alpha}.\end{align}
\end{proposition}

\ncS{S:c:proof_threats} \ncSk{S:k:proof_threats} \ncSk{S:k:proof_threats_2}
In order to prove this result, we follow again the structure of proof given in Section \ref{S:sss:renormalization_scheme} (only here the scale parameter is $r$ and not $H$). Mind that here we will need to apply the renormalization method twice to get the desired estimate.

We start by considering only $r=3^k$ for $k\in\NN$. We set $$q_k=q_k(H)=\sup_{w\in \RR\times\ZZ}\PP(w\text{ is not }(H,3^k)\text{-threatened}).$$ Let us start by showing that $q_k$ converges to $0$ when $k\to\infty$, uniformly in $H$ large enough. More precisely, we show the following result.

\begin{lemma}\label{S:l:first_est}
    There exists $\ucS{S:c:proof_threats}\in [1/3,1)$ and $\ucSk{S:k:proof_threats}\in\NN$ such that
\begin{align}\label{S:estimate}
\forall k\geqslant 2,\;\forall H\geqslant \ucSH{H:traps},\;q_{\ucSk{S:k:proof_threats}+k}\leqslant \ucS{S:c:proof_threats}^k.
\end{align}
\end{lemma}
Note that this bound is not the final bound that we want (which involves $\alpha$). That is why we will need to show a second estimate after this one.
\begin{proof}
To prove \eqref{S:estimate}, we use induction on $k\geqslant 2$. Let us fix $\ucSk{S:k:proof_threats}\in\NN$ (we will choose it later in the proof).
\paragraph{Base case.} If a point is not $(H,r)$-threatened, in particular it is not $H$-trapped, so, by Lemma \ref{S:l:bound_trap}, \begin{align*}\sup_{H\geqslant \ucSH{H:traps}} q_{\ucSk{S:k:proof_threats}+2}\leqslant \sup_{H\geqslant \ucSH{H:traps}}\sup_{w\in \R\times\ZZ} \PP(\text{$w$ is not $H$-trapped})<1.\end{align*}
Therefore there exists $\ucS{S:c:proof_threats}\in[1/3,1)$ such that the case $k=2$ in (\ref{S:estimate}) is satisfied, namely $q_{\ucSk{S:k:proof_threats}+2}\leqslant \ucS{S:c:proof_threats}^2$ for all $H\geqslant \ucSH{H:traps},$ and the choice of $\ucS{S:c:proof_threats}$ can be made independently of $\ucSk{S:k:proof_threats}$.
\paragraph{Induction step.} Fix $k\geqslant 2$ and suppose that \begin{align}\label{S:e:induction_ass}
\displaystyle\sup_{H\geqslant \ucSH{H:traps}} q_{\ucSk{S:k:proof_threats}+k}\leqslant \ucS{S:c:proof_threats}^k.
\end{align}
Fix an integer $H\geqslant \ucSH{H:traps}$ and $w\in \R\times\ZZ.$ Note that event $\{\text{$w$ is not $(H,3^{\ucSk{S:k:proof_threats}+k+1})$-threatened}\}$ is included in the events given by $$\mathcal{A}_k=\bigcap_{j=0}^{3^{\ucSk{S:k:proof_threats}+k}-1} \{\text{$w_j$ is not $H$-trapped}\}\;\;\;\text{and}\;\;\;\mathcal{A}_k'=\bigcap_{j=2\cdot 3^{\ucSk{S:k:proof_threats}+k}}^{3^{\ucSk{S:k:proof_threats}+k+1}-1} \{\text{$w_j$ is not $H$-trapped}\}.$$
Using Remark \ref{S:r:trap_mesurability}, those events are measurable with respect to horizontal strips separated in time by $3^{\ucSk{S:k:proof_threats}+k}H$. In order to replace those strips by boxes, anticipating the use of Fact \ref{S:p:mixing}, we introduce the following events that will give us additional horizontal bounds:
\begin{align*}
    \mathcal{O}_k=\bigcap_{j=0}^{3^{\ucSk{S:k:proof_threats}+k}-1} G_j(w)\;\;\text{and}\;\;\mathcal{O}_k'=\bigcap_{j=2\cdot 3^{\ucSk{S:k:proof_threats}+k}}^{3^{\ucSk{S:k:proof_threats}+k+1}-1} G_j(w_{2\cdot 3^{\ucSk{S:k:proof_threats}+k}})
\end{align*}
where, for $\bar{w}\in\RR\times\ZZ,$ $$G_j(\bar{w})=\bigcap_{y\in I_{\delta H}(z_{w_j})} \left\{Z^y_{\bigl[0,\tau_{H-2H',z_{w_j}}\bigr]}\subseteq B_{3^{\ucSk{S:k:proof_threats}+k}H}(\bar{w}-(3^{\ucSk{S:k:proof_threats}+k} H/2,0))\right\}.$$

\begin{figure}[!h]
  \begin{center}
    \begin{tikzpicture}[use Hobby shortcut,scale=0.6]
        \draw[fill, above left] (0,0) circle (.05) node {$w$};
        \draw[dotted] (0,0) -- (4.5,9);
        \draw[fill] (0.5,1) circle (.05);
        \draw[fill] (1,2) circle (.05);
        \draw[fill] (1.5,3) circle (.05);
        \draw[fill] (2,4) circle (.05);
        \draw[fill] (2.5,5) circle (.05);
        \draw[fill, above left] (3,6) circle (.05) node {$w_{2\cdot 3^{\ucSk{S:k:proof_threats}+k}H}$};
        \draw[fill] (3.5,7) circle (.05);
        \draw[fill] (4,8) circle (.05);
        \draw[fill, above left] (4.5,9) circle (.05) node {$w_{3^{\ucSk{S:k:proof_threats}+k+1}H}$};
        \draw (-3.5,-0.4) rectangle (3.5,3);
        \draw[right] (3.7,1) node {$B_{3^{\ucSk{S:k:proof_threats}+k}H}(w-(3^{\ucSk{S:k:proof_threats}+k} H/2,0))$};
        \draw (-0.5,5.6) rectangle (6.5,9);
        \draw[below] (8,5.6) node {$B_{3^{\ucSk{S:k:proof_threats}+k}H}(w_{2\cdot 3^{\ucSk{S:k:proof_threats}+k}H}-(3^{\ucSk{S:k:proof_threats}+k} H/2,0))$};
    \end{tikzpicture}
    \caption{}
    \label{S:f:threats2}
    \end{center}
\end{figure}

Now, $\mathcal{A}_k\cap \mathcal{O}_k$ is measurable with respect to box $B_{3^{\ucSk{S:k:proof_threats}+k}H}(w-(3^{\ucSk{S:k:proof_threats}+k} H/2,0))$, and $\mathcal{A}_k'\cap \mathcal{O}_k'$ is measurable with respect to box $B_{3^{\ucSk{S:k:proof_threats}+k}H}(w_{2\cdot 3^{\ucSk{S:k:proof_threats}+k}}-(3^{\ucSk{S:k:proof_threats}+k} H/2,0))$. Those two boxes satisfy the assumptions of Fact \ref{S:p:mixing} with $h=3^{\ucSk{S:k:proof_threats}+k}H/2$.

Let us now bound the probability of $\bigl(G_j(w)\bigr)^c$. For $j\in\llbracket 0, 3^{\ucSk{S:k:proof_threats}+k}-1\rrbracket$, we have \begin{align}\label{S:e:inclusion_weird}
        \bigcap_{y\in I_{\delta H}(z_{w_j})} D^y_{3^{\ucSk{S:k:proof_threats}+k-1} H/\beta}\subseteq G_j(w).
    \end{align}
Indeed, let us assume that $\ucSk{S:k:proof_threats}$ is large enough so that for every $H$ we have $3^{\ucSk{S:k:proof_threats}+k-1}H/\beta\geqslant H-2H'$. Let $j\in\llbracket 0, 3^{\ucSk{S:k:proof_threats}+k}-1\rrbracket$, $y\in I_{\delta H}(z_{w_j})$ and $n\in \llbracket 0,\tau^y_{H-2H',z_{w_j}}\rrbracket$. On the event on the left-hand side of \eqref{S:e:inclusion_weird}, we have
    \begin{align*} 
        \left\vert X^y_n-\pi_1(w)\right\vert
        &\leqslant \left\vert X^y_n-\pi_1(y)\right\vert+\left\vert \pi_1(y)-\pi_1(w)\right\vert\\
        &\leqslant 3^{\ucSk{S:k:proof_threats}+k-1} H +jH\vert v_+\vert +\delta H + 4\beta H'+\delta H\\
        &\leqslant 3^{\ucSk{S:k:proof_threats}+k-1} H + 3^{\ucSk{S:k:proof_threats}+k}\beta H+ 2\delta H + 4\beta H'&\mbox{using \eqref{S:e:bounds_on_v_+}}\\
        &<3^{\ucSk{S:k:proof_threats}+k}\beta H+ \frac{3^{\ucSk{S:k:proof_threats}+k} H}{2},    \end{align*}
    provided that $\ucSk{S:k:proof_threats}$ is large enough (independently of $H$), which gives a first condition to choose $\ucSk{S:k:proof_threats}$. From these horizontal bounds, noting that the vertical bounds are always satisfied by construction, we have that 
    $$\forall n\leqslant \tau^y_{H+2H',z_{w_j}}, \;Z^y_n\in B_{3^{\ucSk{S:k:proof_threats}+k}H}(w-(3^{\ucSk{S:k:proof_threats}+k} H/2,0)),$$
    which ends the proof of \eqref{S:e:inclusion_weird}. Similarly, with the exact same arguments, for $j\in\llbracket 2\cdot 3^{\ucSk{S:k:proof_threats}+k},3^{\ucSk{S:k:proof_threats}+k+1}-1\rrbracket$, we have
    \begin{align}\label{S:e:inclusion_weird2}
        \bigcap_{y\in I_{\delta H}(z_{w_j})} D^y_{3^{\ucSk{S:k:proof_threats}+k-1} H/\beta}\subseteq G_j(w_{2\cdot 3^{\ucSk{S:k:proof_threats}+k}}).
    \end{align}
Using \eqref{S:e:inclusion_weird} and \eqref{S:e:inclusion_weird2} along with union bounds and Proposition \ref{S:p:horizontal_bounds}, we have $$\PP\left(\mathcal{O}_k^c\right)\leqslant \ucS{S:c:ball}^{-1}\, 3^{\ucSk{S:k:proof_threats}+k} (\delta H)^2  e^{-\ucS{S:c:ball} 3^{\ucSk{S:k:proof_threats}+k-1}H/\beta}\leqslant c^{-1}e^{-c \,3^{\ucSk{S:k:proof_threats}+k}H},$$
where $c>0$ does not depend on $H$, and the same holds for $\mathcal{O}_k'$. So, using Fact \ref{S:p:mixing},
\begin{align*}
        q_{\ucSk{S:k:proof_threats}+k+1}&\leqslant \PP\left((\mathcal{A}_k\cap \mathcal{O}_k) \cap  (\mathcal{A}_k'\cap \mathcal{O}_k')\right)+\PP(\mathcal{O}_k^c)+\PP((\mathcal{O}_k')^c)\\
        &\leqslant q_{\ucSk{S:k:proof_threats}+k}^2+c_1\,\left(\frac{3^{\ucSk{S:k:proof_threats}+k}H}{2}\right)^{-\alpha}+2 c^{-1}e^{-c \,3^{\ucSk{S:k:proof_threats}+k}H}\\
        &\leqslant q_{\ucSk{S:k:proof_threats}+k}^2+c\,3^{-(\ucSk{S:k:proof_threats}+k)\alpha}&\mbox{using that $H\geqslant 1$.}
    \end{align*}
    In the end, using induction assumption \eqref{S:e:induction_ass} as well as the fact that $1/3\leqslant\ucS{S:c:proof_threats}<1$, $\alpha\geqslant 1$ and $k\geqslant 2$, \begin{align*}
        \frac{q_{\ucSk{S:k:proof_threats}+k+1}}{\ucS{S:c:proof_threats}^{k+1}}
        &\leqslant \ucS{S:c:proof_threats}^{k-1}+c\,\ucS{S:c:proof_threats}^{\ucSk{S:k:proof_threats}-1}\leqslant \ucS{S:c:proof_threats}+c\,\ucS{S:c:proof_threats}^{\ucSk{S:k:proof_threats}-1} \leqslant 1,
    \end{align*}
    provided that $\ucSk{S:k:proof_threats}$ is chosen large enough (recall that the choice of $\ucS{S:c:proof_threats}$ was independent of $\ucSk{S:k:proof_threats}$). This concludes the induction and therefore the proof of Lemma \ref{S:l:first_est}.
\end{proof}

We now prove the desired estimate along the subsequence; more precisely, we prove the following lemma.

\begin{lemma}\label{S:l:second_est}
    There exists $\ucSk{S:k:proof_threats_2}\in\NN^*$ such that \begin{align} \forall k\in \NN^*,\;\forall H\geqslant \ucSH{H:traps},\;q_{\ucSk{S:k:proof_threats_2}+k}\leqslant \frac{1}{2}\,3^{-\alpha k}.\label{S:pataprout}\end{align}
\end{lemma}
\begin{proof}
We use exactly the same method as in the proof of the first estimate \eqref{S:estimate}. Since $q_k$ goes to $0$ uniformly in $H\geqslant \ucSH{H:traps}$ (on account of \eqref{S:estimate}), we have, for any $\ucSk{S:k:proof_threats_2}\in\NN$ large enough and $H\geqslant \ucSH{H:traps}$,
$$q_{\ucSk{S:k:proof_threats_2}+1}\leqslant \frac{1}{2}\;3^{-\alpha}.$$ We now show by induction on $k\geqslant 1$ that $q_{\ucSk{S:k:proof_threats_2}+k}\leqslant \frac{1}{2}\,3^{-\alpha k}.$ For the induction step, we obtain with the same arguments as before,
$$\frac{q_{\ucSk{S:k:proof_threats_2}+k+1}}{\frac{1}{2}3^{-\alpha(k+1)}}\leqslant 2\cdot 3^{\alpha(k+1)}\,\left(\frac{1}{4}3^{-2\alpha k}+c\,3^{-\alpha(\ucSk{S:k:proof_threats_2}+k)}\right),$$
which is at most $1$ provided that $\ucSk{S:k:proof_threats_2}$ is large enough. This gives a second condition to choose $\ucSk{S:k:proof_threats_2}$. This constant being properly chosen, we get \eqref{S:pataprout}.
\end{proof}

\begin{proof}[Proof of Proposition \ref{S:p:proba_threats}] It remains to interpolate \eqref{S:pataprout} to get the desired result. Let $H\geqslant\ucSH{H:traps}$, $r\geqslant 3^{\ucSk{S:k:proof_threats_2}+1}$ and $k\in\NN^*$ such that $3^{\ucSk{S:k:proof_threats_2}+k}\leqslant r<3^{\ucSk{S:k:proof_threats_2}+k+1}.$ Then 
\begin{align*}
\PP(w\text{ is not $(H,r)$-threatened})&\leqslant\PP(w\text{ is not $(H,3^{\ucSk{S:k:proof_threats_2}+k})$-threatened})&\mbox{by definition}\\
&\leqslant \frac{1}{2} 3^{-\alpha k}&\mbox{by \eqref{S:pataprout}}\\
&\leqslant \frac{3^{\alpha(\ucSk{S:k:proof_threats_2}+1)}}{2} \,r^{-\alpha}.
\end{align*}
It remains to tailor constant $\ucS{S:c:threats}$ in order for \eqref{S:e:estimate_final_threat} to hold for every $r\in\NN^*$.
\end{proof}

\subsubsection{Delay near a threatened point}

Now that we have shown that every point is threatened with a high probability, we need to quantify the delay caused by threats for the random walk, just as we did for traps with Lemma \ref{S:l:delay_traps}.

First a technical definition is required, because we do not want to look for threats among too many points for entropy reasons (see the proof of Lemma \ref{S:l:paths}).

\begin{definition}\label{S:d:rounded_points}
Let $y\in\ZZ^2$ and $H\in\NN^*$ such that $H \geqslant 4/\delta$. We denote by $\lfloor y\rfloor_H$ the point of $\left\lfloor \delta H/4\right\rfloor \ZZ\times H' \ZZ$ given by
$$\lfloor y \rfloor_H=\left(\left\lfloor \frac{\pi_1(y)}{\tilde{H}}\right\rfloor \tilde{H},\,\left\lfloor \frac{\pi_2(y)}{H'}\right\rfloor H'\right),\;\;\text{where } \tilde{H}=\lfloor \delta H/4\rfloor.$$
\end{definition}

\ncSH{H:delay_threats}
Let $\ucSH{H:delay_threats}\geqslant \ucSH{H:delay_traps}$ be an integer (depending on $v_-$ and $v_+$) satisfying 
\begin{align}\label{S:d:H:delay_threats}
\text{for every $H\geqslant \ucSH{H:delay_threats}$, }
\left\{\begin{array}{l}
H\geqslant 4/\delta;\\
-\delta H +\beta H'\leqslant 0.
\end{array}\right.
\end{align}

The first condition ensures that Definition \ref{S:d:rounded_points} can be used, and the second one is a technical requirement that will appear later on.

\begin{proposition}\label{S:p:delay_threats}
Let $H\geqslant \ucSH{H:delay_threats}$, $r\in\NN^*$ and $y\in\ZZ^2$. We set $w=\lfloor y\rfloor_H.$ Let $\Gamma\in\mathcal{H}$ be such that $\pi_2(\mathrm{Supp}\,\Gamma)\subseteq (-\infty,\pi_2(w)+H')$. We assume the following conditions are met:
\begin{enumerate}
    \item $w$ is $(H,r)$-threatened;
    \item For every $j\in\llbracket 0,r-1\rrbracket$,  $\Theta_{jH}^y V_{H,Z^{y,\Gamma}_{\tau_{jH,y}}}^{y,\Gamma}\leqslant v_++\frac{\delta}{2r}$;
    \item For every $j\in\llbracket 0,r-1\rrbracket$, $\Theta_{jH}^y D_{4H'}^{y,\Gamma}$ and $\Theta_{jH+4H'}^y E_{H'}^{y,\Gamma}$ occur;
    \item For every $j\in\llbracket 1,r\rrbracket$, $\Theta_{jH}^w D_{H'}^{y,\Gamma}$ occurs.
\end{enumerate}
Then we have $X^{y,\Gamma}_{\tau_{rH,y}}\leqslant \pi_1(y)+\left( v_+-\displaystyle\frac{\delta}{2r}\right) rH.$

Again, $y$ and $\Gamma$ can be replaced by random variables satisfying the same assumptions.
\end{proposition}

To simplify the reading of the proof, we set, for every $j\in\llbracket 0,r\rrbracket$, 
$$\begin{array}{lll}
\mathcal{Z}_j=\mathcal{Z}_j^{y,\Gamma}=Z^{y,\Gamma}_{\tau_{jH,y}}&\text{and}& \mathcal{N}_j=\mathcal{N}_j^{y,\Gamma}=N^{y,\Gamma}_{\tau_{jH,y}};\\
\tilde{\mathcal{Z}}_j=\tilde{\mathcal{Z}}_j^{y,\Gamma}=Z^{y,\Gamma}_{\tau_{jH,w}}&\text{and}& \tilde{\mathcal{N}}_j=\tilde{\mathcal{N}}_j^{y,\Gamma}=N^{y,\Gamma}_{\tau_{jH,w}}.
\end{array}$$

Let us explain the proposition heuristically. Suppose that a point $w$ close to $y$ is threatened (condition 1). Divide the strip $\pi_2^{-1}([\pi_2(y),\pi_2(y)+rH])$ into $r$ strips of height $H$, and assume that the random walk started at $y$ does not go too fast on each of these $r$ strips (condition 2). The fact that $w$ is threatened and that $Z^y$ does not go too fast will imply that $Z^y$ meets a potential barrier associated to a trapped point $w_{j_0}=w+j_0 H(v_+,1)$, and it cannot get around this barrier because of condition 3 (recall Lemma \ref{S:l:delay_traps}). Therefore, combining the upper bound we have on each of the $r$ strips, and the new upper bound that the barrier gives us on this particular strip, we get a global upper bound. Mind that in this proposition two grids coexist, one with lines at heights $\pi_2(y)+jH$ ($j\in\llbracket 0,r\rrbracket$), where the $\mathcal{Z}_j$ are, and one with lines at heights $\pi_2(w)+jH$ ($j\in\llbracket 1,r\rrbracket$), where the $\tilde{\mathcal{Z}}_j$ are. Condition 4 allows us to control the error of displacement between $\tilde{\mathcal{Z}}_j$ and $\mathcal{Z}_j$.

\begin{figure}[!h]
  \begin{center}
    \begin{tikzpicture}[use Hobby shortcut,scale=0.6]
        \draw (3,0.3) rectangle (5,0.7);
        \draw[fill, left] (3,0.3) circle (.05) node {$z_{w_{j_0}}$};
        \draw[fill] (3.5,0.5) circle (.05);
        \draw[above] (3.5,0.7) node {$x_0'$};
        \draw[fill, above left] (-2,-2) circle (.05) node {$w_{j_0}$};
        \draw[dotted] (-2,-2) -- (-3.5,-4);
        \draw[->, >=latex, left] (-3.6,-3.6) -- (-3.8,-4) node {$w$};
        \draw[fill,right] (1,-1) circle (.05) node {$w_{j_0}+(\delta H,H')$};
        \draw [dotted, thick] (3.5,0.5) .. (3.7,0.3) ..  (5,2) .. (4.7,3) .. (2,4.5) .. (2,4) .. (2.4,5) .. (0,6) .. (-4,7);
        \draw[fill] (-4,7) circle (.05);  
        \draw[fill, below left] (0,-2) circle (.05) node {$\tilde{\mathcal{Z}}_{j_0}$};
        \draw[fill, left] (-0.5,-1.5) circle (.05) node {$\mathcal{Z}_{j_0}$};
        \draw (-2.5,-4) .. (-1,-3) .. (1.5,-3) .. (0,-2) .. (-0.5,-1.5)  .. (0,-1) .. (-1,0) .. (0.2,2) .. (-1,4) .. (0.3,5) .. (0,6);
        \draw[->, >=latex, below left] (-2.7,-3.6) -- (-2.8,-4) node {$y$};
        \draw[fill, below left] (0,6) circle (.05) node {$x$};
    \end{tikzpicture}
    \caption{Illustration of $w=\lfloor y\rfloor_H$ being $(H,r)$-threatened and the random walk starting from $y$ experiencing a delay on account of the trap. In this example, at the point labeled $x$, the sample path started at $y$ coalesces with the dotted curve, which causes a delay.}
    \label{S:f:threats}
    \end{center}
\end{figure}
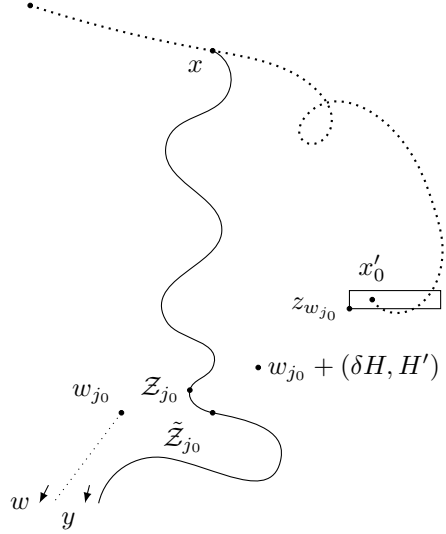

\begin{proof}
Let $H\geqslant \ucSH{H:delay_threats}$, $r\in\NN^*$, $y\in\ZZ^2$, $w=\lfloor y\rfloor_H$, $\Gamma\in\mathcal{H}$ such that $\pi_2(\mathrm{Supp}\,\Gamma)\subseteq (-\infty,\pi_2(w)+H')$. Assume all four assumptions from the statement are satisfied. Because of condition 1, there exists $j_0\in \llbracket 0,r-1\rrbracket$ such that $w_{j_0}$ is $H$-trapped. We want to apply Lemma \ref{S:l:delay_traps} replacing $w$ in the statement by $w_{j_0}$, $x_0$ by $\mathcal{Z}_{j_0}$ and $\Gamma$ by $\mathcal{N}_{j_0}$ (recall that the lemma was also true for a random choice of $x_0$ and $\Gamma$), which requires justifying that
\begin{align}\label{S:e:claim1}
\mathcal{Z}_{j_0}\in w_{j_0}+(-\infty,\delta H)\times [0,H').
\end{align}
Note that the localization of the history is obtained by construction, and assumptions 2 and 3 in Lemma \ref{S:l:delay_traps} are given by condition 3 in the statement of the proposition. In order to prove \eqref{S:e:claim1}, note that
\begin{align*}
    \pi_1(\mathcal{Z}_{j_0})
    &=\pi_1(y)+\sum_{j=0}^{j_0-1}\left(\pi_1(\mathcal{Z}_{j+1})-\pi_1(\mathcal{Z}_j)\right)\\
    &\leqslant \pi_1(y)+\left(v_++\frac{\delta}{2r}\right)j_0H&\mbox{using condition 2}\\
    &\leqslant \pi_1(w_{j_0})+\frac{\delta H}{4}+\frac{\delta H}{2}&\mbox{by definition of $w=\lfloor y\rfloor_H$ and $w_{j_0}=w+j_0H(v_+,1)$}\\
    &<\pi_1(w_{j_0})+\delta H.
\end{align*}
As for the second coordinate, by definition of $w$ we have
\begin{align*}
    \pi_2(\mathcal{Z}_{j_0})
    &=\pi_2(y)+j_0H\in \pi_2(w)+j_0H+[0,H')=\pi_2(w_{j_0})+[0,H')
\end{align*}
This proves \eqref{S:e:claim1}. Now, by applying Lemma \ref{S:l:delay_traps}, we get that
\begin{align}\label{S:e:localization_x_3}
    \pi_1(\tilde{\mathcal{Z}}_{j_0+1})=X^{\mathcal{Z}_{j_0},\mathcal{N}_{j_0}}_{\tau_{H,w_{j_0}}}\leqslant \pi_1(w_{j_0})+(v_+-2\delta)H.
\end{align}
Therefore, using that $D_{H'}^{\tilde{\mathcal{Z}}_{j_0+1},\tilde{\mathcal{N}}_{j_0+1}}$ occurs (condition 4), we get
\begin{align*}
    \pi_1(\mathcal{Z}_{j_0+1})
    &\leqslant \pi_1(\tilde{\mathcal{Z}}_{j_0+1})+\beta H'\\
    &\leqslant \pi_1(w_{j_0})+(v_+-2\delta)H+\beta H'\\
    &\leqslant \pi_1(w_{j_0})+(v_+-\delta)H,
\end{align*}
using \eqref{S:d:H:delay_threats}. Consequently, we have
\begin{align*}
    \pi_1(\mathcal{Z}_r)
    &=\pi_1(\mathcal{Z}_{j_0+1})+\sum_{j=j_0+1}^{r-1} \left(\pi_1(\mathcal{Z}_{j+1})-\pi_1(\mathcal{Z}_j)\right)\\
    &\leqslant \pi_1(w_{j_0})+(v_+-\delta)H+(r-j_0-1)\left(v_++\frac{\delta}{2r}\right)H&\mbox{using \eqref{S:e:localization_x_3} and condition 2}\\
    &\leqslant \pi_1(y)+j_0v_+H+(v_+-\delta)H+(r-j_0-1)\left(v_++\frac{\delta}{2r}\right)H&\mbox{}\\
    &\leqslant\pi_1(y)+rv_+H-\frac{\delta}{2}H\\
    &= \pi_1(y)+\left(v_+-\frac{\delta}{2r}\right)rH.
\end{align*}
\end{proof}

\subsection{Threatened paths}\label{S:ss:threatened_paths}
\ncSk{S:k:paths} \ncS{S:c:paths}

We now know that when a random walk passes near a threatened point, it will be delayed to the left with a high probability (Proposition \ref{S:p:delay_threats}), and that each point has a high probability of being threatened (Proposition \ref{S:p:proba_threats}). The goal of this section is to improve the latter result, by showing that with a high probability, every random walk meets a lot of threatened points along its way. Mind that this is not a direct consequence of Proposition \ref{S:p:proba_threats}, because the random walk might go precisely to areas where there are few threats. From now on, we will focus on specific values of the parameters introduced before: $H=hL_k$ with $k>\ucSk{S:k:paths}$ for a wise choice of $h\in\NN^*$ and $\ucSk{S:k:paths}\in\NN$, and $r=l_{\ucSk{S:k:paths}}.$

\begin{lemma}\label{S:l:paths}
There exists $\ucSk{S:k:paths}\in\NN$ and $\ucS{S:c:paths}=\ucS{S:c:paths}(\delta)>0$ such that the following conditions are met:
\begin{itemize}[leftmargin=*]
    \item $L_{\ucSk{S:k:paths}}\geqslant \ucSH{H:delay_threats};$
    \item For every $h\in\NN^*$,
    \begin{align}\label{S:e:inequality1}
        \PP\left(\exists\,y\in I_{hL_{\ucSk{S:k:paths}+1}},\;\lfloor y \rfloor_{hL_{\ucSk{S:k:paths}}}\text{ is not $(hL_{\ucSk{S:k:paths}},l_{\ucSk{S:k:paths}})$-threatened}\right)\leqslant \ucS{S:c:paths} L_{\ucSk{S:k:paths}+1}^{-(2\alpha-3)/10};
    \end{align}
        \item The following two technical requirements are satisfied
    \begin{align}\label{S:e:condition_supp}
    &49\beta l_{\ucSk{S:k:paths}}\leqslant \delta l_{\ucSk{S:k:paths}+1};\\ \label{S:e:condition_less_than_1}
    &\ucS{S:c:end_proof} (\ucS{S:c:end_proof}+\ucS{S:c:paths}^2)\,L_{\ucSk{S:k:paths}}^{-(6\alpha-49)/40}\leqslant \ucS{S:c:paths},
    \end{align}
    where $\ucS{S:c:end_proof}$ was defined in \eqref{S:e:constant_end_proof}.
\end{itemize}
\end{lemma}

\begin{proof}
Let $h\in\NN^*$ and $\ucSk{S:k:paths}\in\NN$ satisfying $L_{\ucSk{S:k:paths}}\geqslant \ucSH{H:delay_threats}$ and \eqref{S:e:condition_supp}. Then, using Proposition \ref{S:p:proba_threats} and the fact that $\ucSH{H:delay_threats}\geqslant\ucSH{H:traps}$, we have
\begin{align*}
    \PP&\left(\exists\,y\in I_{hL_{\ucSk{S:k:paths}+1}},\;\lfloor y \rfloor_{hL_{\ucSk{S:k:paths}}}\text{ is not $(hL_{\ucSk{S:k:paths}},l_{\ucSk{S:k:paths}})$-threatened}\right)\\
    &\leqslant \left\lceil \frac{hL_{\ucSk{S:k:paths}+1}}{\left\lfloor \frac{\delta hL_{\ucSk{S:k:paths}}}{4}\right\rfloor}\right\rceil\left\lceil \frac{(hL_{\ucSk{S:k:paths}+1})'}{(hL_{\ucSk{S:k:paths}})'}\right\rceil\,\ucS{S:c:threats}\,l_{\ucSk{S:k:paths}}^{-\alpha}\\
    &\leqslant c\,L_{\ucSk{S:k:paths}+1}^{-(2\alpha-3)/10}.
\end{align*}
Therefore, we do get inequality \eqref{S:e:inequality1} with some constant $\ucS{S:c:paths}=\ucS{S:c:paths}(\delta)>0.$ Now that $\ucS{S:c:paths}$ is fixed, it suffices to take a larger $\ucSk{S:k:paths}$ so that inequality \eqref{S:e:condition_less_than_1} holds as well, which is possible because $\alpha\geqslant 9.$
\end{proof}

Conditions \eqref{S:e:condition_supp} and \eqref{S:e:condition_less_than_1} will appear naturally later in the proof. Also, note that considering only rounded points $\lfloor y\rfloor_{hL_{\ucSk{S:k:paths}}}$ was crucial here to obtain a bound that is uniform in $h$. 

\begin{definition}\label{S:d:density}
Let $\ucSk{S:k:paths}$ be defined as in Lemma \ref{S:l:paths}. Let $k>\ucSk{S:k:paths}$, $w\in\RR\times\ZZ$, $h\in\NN^*$ and $y\in \ZZ^2$ such that $\pi_2(y)\in [\pi_2(w),\pi_2(w)+(hL_k)')$. We set the threatened density of random walk $Z^{y}$ to be
\begin{align}\label{S:e:def_density}
    D_{h,k}^{y}(w)=\frac{L_{\ucSk{S:k:paths}+1}}{L_k} \,\#\left\{0\leqslant j<\frac{L_k}{L_{\ucSk{S:k:paths}+1}},\;\left\lfloor Z^{y}_{\tau_{jhL_{\ucSk{S:k:paths}+1},w}}\right\rfloor_{hL_{\ucSk{S:k:paths}}} \text{ is $(hL_{\ucSk{S:k:paths}},l_{\ucSk{S:k:paths}})$-threatened}\right\}.\end{align}
\end{definition}

As usual, we also set $D_{h,k}^y=D_{h,k}^y(o).$

\begin{remark}\label{S:r:redundancies}
    Mind that contrary to what is depicted in Figure \ref{S:f:density}, we could have $\pi_1(y)-\pi_1(w)>hL_{\ucSk{S:k:paths}+1}.$ In that case, whenever $jhL_{\ucSk{S:k:paths}+1}\leqslant \pi_1(y)-\pi_1(w)$, $\tau^y_{jhL_{\ucSk{S:k:paths}+1},w}=0$, so $Z^{y}_{\tau_{jhL_{\ucSk{S:k:paths}+1},w}}=y$, which causes redundancies. The number of such indices will have to be upper-bounded.
\end{remark}

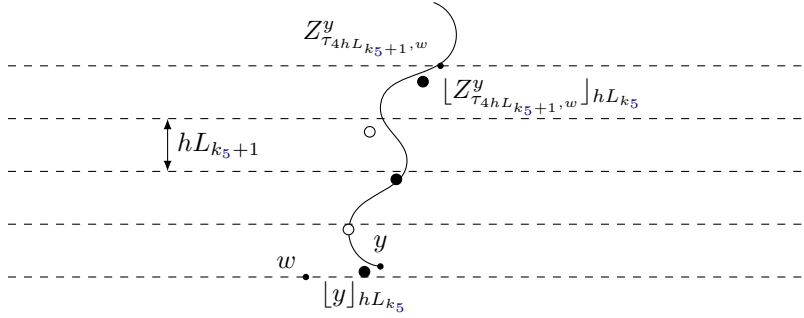
\begin{figure}[!h]
  \begin{center}
    \begin{tikzpicture}[use Hobby shortcut, scale=0.7]
        \draw[fill] (2,0) circle (.05);
        \draw[above] (2,0.1) node {$y$};
        \draw (2,0) .. (1.5,1) .. (2.5,2) .. (2,3) .. (3.3,4) .. (3,5);
        \draw[dashed] (-5,-0.2) -- (10,-0.2);
        \draw[dashed] (-5,0.8) -- (10,0.8);
        \draw[dashed] (-5,1.8) -- (10,1.8);
        \draw[dashed] (-5,2.8) -- (10,2.8);
        \draw[dashed] (-5,3.8) -- (10,3.8);
        \draw[fill] (3.13,3.8) circle (.05);
        \draw[above left] (3.13,3.8) node {$Z^y_{\tau_{4hL_{\ucSk{S:k:paths}+1},w}}$};
        \draw[<->,>=latex] (-2,1.8) -- (-2,2.8);
        \draw[right] (-2,2.3) node {$hL_{\ucSk{S:k:paths}+1}$};
        
        \draw[fill, below] (1.7,-0.1) circle (.1) node{$\lfloor y\rfloor_{hL_{\ucSk{S:k:paths}}}$};
        \draw[fill, above left] (0.6,-0.2) circle (.05) node {$w$};
        \draw[fill=white, above left] (1.4,0.7) circle (.1);
        \draw[fill, below left] (2.3,1.65) circle (.1);
        \draw[fill=white, below left] (1.8,2.55) circle (.1);
        \draw[fill, below right] (2.8,3.5) circle (.1);
        \draw[fill, below right] (2.95,3.75) node {$\lfloor Z^y_{\tau_{4hL_{\ucSk{S:k:paths}+1},w}}\rfloor_{hL_{\ucSk{S:k:paths}}}$};
    \end{tikzpicture}
    \caption{Illustration of a point $y\in I_{hL_k}(w)$ whose density is more than $1/2$: there are more threatened points (the large filled circles) than non-threatened points (the empty large circles) along the way.}
    \label{S:f:density}
    \end{center}
\end{figure}

Let us now state our final proposition before ending the proof of Lemma \ref{S:l:v_-=v_+}: with a high probability, our random walks encounter threats more than half of the time along the way.
\begin{proposition}\label{S:p:density}
For every $k>\ucSk{S:k:paths}$ and $h\in\NN^*,$
$$\PP(\exists\, y\in I_{hL_k},\; D_{h,k}^y<1/2)\leqslant \ucS{S:c:paths}\,L_k^{-(2\alpha-3)/10}.$$
\end{proposition}

\begin{proof}
The proof uses again the renormalization method presented in Section \ref{S:sss:renormalization_scheme} and is very similar to that of Lemma \ref{S:p:estimates_on_p}. Let us fix $k>\ucSk{S:k:paths}$ and $h\in\NN^*$. We define a sequence of densities $(\rho_k)_{k\geqslant \ucSk{S:k:paths}}$ by setting 
$$\left\{\begin{array}{l} \rho_{\ucSk{S:k:paths}}=1;\\
\forall k\geqslant \ucSk{S:k:paths},\;\rho_{k+1}=\rho_k-\frac{5}{l_k}.
\end{array}\right.$$
One can check, using a computational knowledge engine, that since $L_0\geqslant 10^{10},$ we have $\sum_{k\geqslant 1} \frac{5}{l_k}\leqslant \frac{1}{2},$ therefore we have $\rho_k\geqslant 1/2$ for every $k\geqslant \ucSk{S:k:paths}.$ We define, for $w\in\RR\times\ZZ$, $$S_{h,k}(w)=\{\exists\,y\in I_{hL_k}(w),\;D_{h,k}^y(w)\leqslant\rho_k\}.$$
Since $\rho_k\geqslant 1/2$, it suffices to show that $s_{h,k}=\PP(S_{h,k}(o))$ satisfies
\begin{align}\label{S:e:estimate_density_to_show}
    s_{h,k}\leqslant \ucS{S:c:paths}\,L_k^{-(2\alpha-3)/10}.
\end{align}
To do this, we use induction on $k>\ucSk{S:k:paths}$.

\textsc{Base case.} When $k=\ucSk{S:k:paths}+1$, the result follows directly from Definition \ref{S:d:density} and Lemma \ref{S:l:paths}.

\textsc{Induction step.} Assume that \eqref{S:e:estimate_density_to_show} has been shown for a fixed $k>\ucSk{S:k:paths}$, and fix $y\in I_{hL_{k+1}}.$ Recall the definitions of $\CCC_k$, $\mathcal{F}_k$ and $\mathcal{G}_k$, from \eqref{S:d:CCC}, \eqref{S:e:eventF} and \eqref{S:e:eventG}, where $H_k$ is replaced by $hL_k$. Recall also notation $\xi_j^y$ used in \eqref{S:e:zeta}. Note that in the definitions of $\mathcal{F}_k$ and $\mathcal{G}_k$, we will not use the part with events $D$ (because contrary to Lemma \ref{S:p:estimates_on_p}, here we are not looking at horizontal displacements).

We set $\mathcal{Z}_j=Z^y_{\tau_{jH_k}}$ to make notations lighter. We claim that
\begin{align}\label{S:e:inclusion37}
    \mathcal{G}_k\cap \left\{D_{h,k+1}^y\leqslant \rho_{k+1}\right\}
    \subseteq \left\{\text{there exist three $j\in\llbracket 1,l_k-1\rrbracket$ such that $D^{\xi_{j}^y}_{h,k}(\mathcal{Z}^y_j)\leqslant \rho_k$}\right\}
\end{align}
(note that here again, we put aside $j=0$). Indeed, suppose that $\mathcal{G}_k$ occurs but it is not the case that there exist three $j\in\llbracket 1,l_k-1\rrbracket$ such that $D^{\xi_{j}^y}_{h,k}(\mathcal{Z}^y_j)\leqslant \rho_k$. This means that for $l_k-3$ values of $j\in \llbracket 1,l_k-1\rrbracket$, we have $D^{\xi_j^y}_{h,k}(\mathcal{Z}^y_j)>\rho_k$, which means
    $$\#\left\{0\leqslant i<\frac{L_k}{L_{\ucSk{S:k:paths}+1}},\,\left\lfloor Z^{\xi^y_j}_{\tau_{ihL_{\ucSk{S:k:paths}+1},\mathcal{Z}_j}}\right\rfloor_{hL_{\ucSk{S:k:paths}}}\text{ is $(hL_{\ucSk{S:k:paths}},l_{\ucSk{S:k:paths}})$-threatened}\right\}>\rho_k\frac{L_k}{L_{\ucSk{S:k:paths}+1}}.$$
    Recall Remark \ref{S:r:redundancies}: we need to control redundancies. Now, on $\mathcal{G}_k$, there are at most $\frac{(hL_k)'}{hL_{\ucSk{S:k:paths}+1}}\leqslant 2\frac{L_k^{1/2}}{L_{\ucSk{S:k:paths}+1}}$ indices $i\in\llbracket 0,L_k/L_{\ucSk{S:k:paths}+1}\rrbracket$ such that $Z^{\xi_j^y}_{\tau_{ihL_{\ucSk{S:k:paths}+1},\mathcal{Z}_j}}=\xi_j^y$ (indeed, this occurs only when $ihL_{\ucSk{S:k:paths}+1}\leqslant \pi_2(\xi^y_j)-\pi_2(\mathcal{Z}^y_j)\leqslant (hL_k)')$. Therefore, 
    $$\#\left\{0\leqslant i<\frac{L_k}{L_{\ucSk{S:k:paths}+1}},\,\left\lfloor \Theta_{jhL_k} Z^y_{\tau_{ihL_{\ucSk{S:k:paths}+1},\mathcal{Z}_j}}\right\rfloor_{hL_{\ucSk{S:k:paths}}}\text{ is $(hL_{\ucSk{S:k:paths}},l_{\ucSk{S:k:paths}})$-threatened}\right\}>\rho_k\frac{L_k}{L_{\ucSk{S:k:paths}+1}}-2\frac{L_k^{1/2}}{L_{\ucSk{S:k:paths}+1}}.$$
    In the end,
    \begin{align*}
        D_{h,k+1}^y&=\frac{L_{\ucSk{S:k:paths}+1}}{L_{k+1}}\,\#\left\{0\leqslant j<\frac{L_{k+1}}{L_{\ucSk{S:k:paths}+1}},\;\left\lfloor Z^y_{\tau_{jhL_{\ucSk{S:k:paths}+1},\mathcal{Z}_j}}\right\rfloor_{hL_{\ucSk{S:k:paths}}}\text{ is $(hL_{\ucSk{S:k:paths}},l_{\ucSk{S:k:paths}})$-threatened}\right\}\\
        &\geqslant \frac{L_{\ucSk{S:k:paths}+1}}{L_{k+1}}\,(l_k-3)\left(\rho_k\frac{L_k}{L_{\ucSk{S:k:paths}+1}}-2\frac{L_k^{1/2}}{L_{\ucSk{S:k:paths}+1}}\right)\\
        &=\left(1-\frac{3}{l_k}\right)\, \left(\rho_k-\frac{2}{L_k^{1/2}}\right)\\
        &>\rho_k-\frac{5}{l_k}=\rho_{k+1},
    \end{align*}
    which proves \eqref{S:e:inclusion37}.
    
    Now, note that on $\mathcal{F}_k\cap \mathcal{G}_k$, for every $j\in\llbracket 0,l_k-1\rrbracket$, $\xi_j^y$ is in a $I_{hL_k}(w)$ with $w\in\CCC_k.$ Therefore, in a similar way as in \eqref{S:pourplustard}, we get
    $$\mathcal{F}_{k}\cap \mathcal{G}_k\cap S_{h,k+1}\subseteq \bigcup_{\underset{|\pi_2(w_1)-\pi_2(w_2)|\geqslant 2 hL_k}{w_1,w_2\in \mathcal{C}_k}}\begin{array}{c} (S_{h,k}(w_1)\,\cap\,F_{hL_k}(w_1))\\\cap\,(S_{h,k}(w_2)\,\cap\,F_{hL_k}(w_2)).\end{array}$$
    Recall constant $\ucS{S:c:end_proof}$ from \eqref{S:e:constant_end_proof}. Here again we get
    $$s_{h,k+1}\leqslant \ucS{S:c:end_proof}\, l_k^{4}\,(s_{h,k}^2+\ucS{S:c:end_proof} L_k^{-\alpha}),$$
    and so, using the induction assumption and \eqref{S:e:condition_less_than_1}, we get
    \begin{align*}
        \frac{s_{h,k+1}}{L_{k+1}^{-(2\alpha-3)/10}} &\leqslant \ucS{S:c:end_proof} (\ucS{S:c:end_proof}+\ucS{S:c:paths}^2)\,L_k^{(2\alpha-3)/8}\, l_k^{4}\,L_k^{-(2\alpha-3)/5}\\
        &\leqslant \ucS{S:c:end_proof} (\ucS{S:c:end_proof}+\ucS{S:c:paths}^2)\,L_k^{-(6\alpha-49)/40}\leqslant \ucS{S:c:paths}.
    \end{align*}
    This concludes the induction and thus the proof of \eqref{S:e:estimate_density_to_show}.
\end{proof}

\subsection{Final proof of Lemma \ref{S:l:v_-=v_+}.}\label{S:ss:final_proof}
Recall that we argued by contradiction and assumed that $v_-<v_+$, therefore $\delta=\frac{v_+-v_-}{5(\beta+1)}>0$. Let $\eta=\frac{\delta}{4 l_{\ucSk{S:k:paths}}}>0$ where $\ucSk{S:k:paths}$ is defined as in Lemma \ref{S:l:paths}. We are going to show that
\begin{align}\label{S:e:contradiction}
p_{L_k^2}\left(v_+-\frac{\eta}{6}\right)\xrightarrow[k\to\infty]{} 0,
\end{align}
which contradicts the definition of $v_+$. From now on, we fix $k>\ucSk{S:k:paths}+1$, and we work with $h=h_k=L_k,$ which is why it was important for our previous estimates to hold uniformly on $h\geqslant 1.$ We let $H_k=h_k L_{\ucSk{S:k:paths}}=L_k L_{\ucSk{S:k:paths}}$ and $r=l_{\ucSk{S:k:paths}}$.
In order to prove \eqref{S:e:contradiction}, we consider the large box $B_{L_k^2}$, which we pave using small sub-boxes $B_{H_k}(w)$ for $w\in\hat{\CCC}_k$, where 
$$\hat{\CCC}_k=\{(iH_k,jH_k)\in H_k \ZZ^2,\,-\lceil\beta L_k/L_{\ucSk{S:k:paths}}\rceil\leqslant j<\lceil(\beta+1)L_k/L_{\ucSk{S:k:paths}}\rceil,\,0 \leqslant j<L_k/L_{\ucSk{S:k:paths}}\}.$$
As usual, we have
$$\underset{w\in\hat{\CCC}_k}{\bigcup} \mathcal{I}_{H_k}(w)\supseteq B_{L_k^2}\cap (\ZZ\times H_k \ZZ).$$
Recall Proposition \ref{S:p:delay_threats}. In order to use it, we define the following events:
\begin{align*}
&\hat{\mathcal{A}}_k= \bigcap_{w\in\hat{\CCC}_k} A_{H_k,w}(v_++\eta)^c;\\
&\hat{\mathcal{F}}_k= F_{L_k^2}\cap \bigcap_{y\in I_{L_k^2}} \left(D_{rH_k}^y\cap\bigcap_{i=1}^{L_k/L_{\ucSk{S:k:paths}}-1}\,\bigcap_{j=0}^{r} \Theta_{(ir+j)H_k}^y D_{4H_k'}^y \cap \Theta_{(ir+j)H_k+4H_k'}^y E_{H_k'}^y\cap \Theta^w_{(ir+j)H_k} D_{H_k'}^y\right);\\
&\hat{\mathcal{G}}_k= \bigcap_{y\in I_{L_k^2}} \bigcap_{i=1}^{L_k/L_{\ucSk{S:k:paths}}-1}\,\bigcap_{j=0}^{r-1} \left(\Theta_{(ir+j)H_k}^y D_{H_k'}^y\cap\left\{\Xi\left(\Theta_{(ir+j)H_k}^y Z^y\right)<H_k'\right\}\right);\\
&\hat{\mathcal{H}}_k= \bigcap_{y\in I_{L_k^2}}\left\{D^y_{L_k,k}\geqslant 1/2\right\}.
\end{align*}

Using Lemma \ref{S:p:estimates_on_p} and the fact that $\alpha>8$, we have
\begin{align}\label{S:e:proba_A_0}
\PP(\hat{\mathcal{A}}_k^c)\leqslant c \left(\frac{L_k}{L_{\ucSk{S:k:paths}}}\right)^2 \ucS{S:c:deviation}(\eta) H_k^{-\alpha/4}\leqslant c L_k^{-(\alpha-8)/4}\xrightarrow[k\to\infty]{} 0.
\end{align}
Using Propositions \ref{S:p:independence}, \ref{S:p:horizontal_bounds} and \ref{S:p:probability_of_going_down}, we have
\begin{align}\label{S:e:proba_F_0}
\PP(\hat{\mathcal{F}}_k^c)
\leqslant \ucS{S:c:box}^{-1} e^{-\ucS{S:c:box} L_k}+ c L_k^4 \left(\ucS{S:c:ball}^{-1} e^{-\ucS{S:c:ball} H_k^{1/2}}+\ucS{S:c:ball}^{-1} e^{-4\ucS{S:c:ball} H_k^{1/2}} +   e^{-\ucS{S:c:proba_of_going_down} H_k^{1/2}}\right)\xrightarrow[k\to\infty]{}0.
\end{align}
Using Propositions \ref{S:p:independence}, \ref{S:p:horizontal_bounds} and Assumption \ref{S:a:cut_line}, we have
\begin{align}\label{S:e:proba_G_0}
    \PP(\hat{\mathcal{G}}_k^c)\leqslant cL_k^4 \left(\ucS{S:c:ball}^{-1}e^{-2\ucS{S:c:ball} H_k^{1/2}}+ \ucS{S:c:cut_line}^{-1} e^{-\ucS{S:c:cut_line} H_k^{1/2}} \right)\xrightarrow[k\to\infty]{} 0.
\end{align}

By Proposition \ref{S:p:density}, we have, using that $\alpha\geqslant 2,$
\begin{align}\label{S:e:proba_H_0}
\PP(\hat{\mathcal{H}}_k^c)\leqslant \ucS{S:c:paths}\,L_k^{-(2\alpha-3)/10}\xrightarrow[k\to\infty]{} 0.
\end{align}
To make things easier to reed, for $y\in I_{L_k^2}$ and $1\leqslant i\leqslant L_k/L_{\ucSk{S:k:paths}+1}$, we set $$\mathscr{Z}^y_i=Z^y_{\tau_{irH_k}}\text{ and }\mathscr{N}^y_i=N^y_{\tau_{irH_k}}.$$
The goal now is to show that on the four events defined above, we have, for every $y\in I_{L_k^2}$,
\begin{align}\label{S:e:inclusionFG}
    V^y_{L_k^2}=\frac{1}{L_k^2}\left(\pi_1\left(\mathscr{Z}^y_{L_k/L_{\ucSk{S:k:paths}+1}}\right)-\pi_1(y)\right)<v_+-\eta/3.
\end{align}
First note that since $\lceil (L_k^2)^{1/2}\rceil=L_k<H_k<rH_k$, we have $\tau_{irH_k}>0$ for every $i\geqslant 1$. Therefore, we will only isolate $i=0$ and simply use $D_{rH_k}^y$ to bound the displacement $\pi_1(\mathscr{Z}^y_{1})-\pi_1(y)$. Let us now focus on bounding  $\pi_1(\mathscr{Z}^y_{i+1})-\pi_1(\mathscr{Z}^y_i)$ where $1\leqslant i<L_k/L_{\ucSk{S:k:paths}+1}$. First note that $\hat{\mathcal{A}}_k$ along with $F_{L_k^2}$ and $\hat{\mathcal{G}}_k$ allow us to bound the displacements of the random walk, similarly to what we did in Section \ref{S:sss:interpolation}. We have, for $1\leqslant i<L_k/L_{\ucSk{S:k:paths}+1}$ and $0\leqslant j<r$,
\begin{align}
X^y_{\tau_{(ir+j+1)H_k,y}}-X^y_{\tau_{(ir+j)H_k,y}}&\leqslant \beta H_k'+ (v_++\eta) H_k\nonumber\\
&\leqslant \left(v_++\frac{3\eta}{2}\right) H_k&\mbox{using \eqref{S:e:condition_supp}}\label{S:new_bound3}\\
&\leqslant \left(v_++\frac{\delta}{2r}\right) H_k&\mbox{by definition of $\eta$.}\label{S:new_bound}
\end{align}
We can use \eqref{S:new_bound3} for indices $i$ such that $\lfloor \mathscr{Z}^y_i\rfloor_{H_k}$ is not $(H_k,r)$-threatened. As for the remaining indices $i$, we will use Proposition \ref{S:p:delay_threats}, replacing in the statement of the proposition $H$ by $H_k$, $y$ by $\mathscr{Z}^y_i$ and $\Gamma$ by $\mathscr{N}^y_{i}$. Assumption 2 in Proposition \ref{S:p:delay_threats} is satisfied using \eqref{S:new_bound}. Assumptions 3 and 4 in Proposition \ref{S:p:delay_threats} are satisfied using event $\hat{\mathcal{F}}_k$. Again, the assumption on the history is automatically obtained. At the end of the day, Proposition \ref{S:p:delay_threats} ensures that for indices $i$ such that  $\lfloor \mathscr{Z}^y_i\rfloor_{H_k}$ is $(H_k,r)$-threatened, we have
\begin{align}\label{S:e:boundd}
\pi_1(\mathscr{Z}^y_{i+1})-\pi_1(\mathscr{Z}^y_{i})\leqslant \left(v_+-\frac{\delta}{2r}\right)rH_k.
\end{align}

Denote by $J^y_k$ the set of indices $i\in\left\llbracket 1,L_k/L_{\ucSk{S:k:paths}+1}-1\right\rrbracket $ such that $\left\lfloor \mathscr{Z}^y_i\right\rfloor_{H_k}$ is $(H_k,r)$-threatened. By Definition \ref{S:d:density}, we have the inclusion of events
\begin{align}\label{S:e:inclusion_density}
\hat{\mathcal{H}}_k\subseteq \bigcap_{y\in I_{L_k^2}} \left\{|J^y_k|\geqslant\frac{L_k}{2 L_{\ucSk{S:k:paths}+1}}-1\right\}.
\end{align}

Therefore,
\begin{align*}
    &\pi_1\left(\mathscr{Z}^y_{L_k/L_{\ucSk{S:k:paths}+1}}\right)-\pi_1(y)\\
    &=\pi_1(\mathscr{Z}^y_1)-\pi_1(y)+\sum_{i=1}^{L_k/L_{\ucSk{S:k:paths}+1}-1} \left(\pi_1( \mathscr{Z}^y_{i+1})- \pi_1(\mathscr{Z}^y_{i})\right)\\
    &\leqslant \beta r H_k +\sum_{i\in J^y_k} \left(\pi_1( \mathscr{Z}^y_{i+1})- \pi_1(\mathscr{Z}^y_{i})\right)+\sum_{i\notin J_k^y} \left(\pi_1( \mathscr{Z}^y_{i+1})- \pi_1(\mathscr{Z}^y_{i})\right)&\mbox{using $D^y_{rH_k}$ inside $\hat{\mathcal{F}}_k$}\\
    &\leqslant \beta r H_k + |J^y_k| \left(v_+-\frac{\delta}{2r}\right) r H_k +\left(\frac{L_k}{L_{\ucSk{S:k:paths}+1}}-1-|J_k^y|\right)\left(v_++\frac{3\eta}{2}\right)r H_k&\mbox{using \eqref{S:e:boundd} and \eqref{S:new_bound3}}\\
    &\leqslant \beta r H_k+ v_+ L_k^2 + \frac{3\eta}{2} L_k^2-\vert J_k^y\vert \left(\frac{\delta}{2r}+\frac{3\eta}{2}\right)r H_k\\
    &=\left(v_+-\frac{\eta}{4}+\left(\frac{7\eta}{2}+\beta\right) \frac{L_{\ucSk{S:k:paths}+1}}{L_k}\right)\,L_k^2&\mbox{by \eqref{S:e:inclusion_density}, $r=l_{\ucSk{S:k:paths}}$, $\eta=\frac{\delta}{4l_{\ucSk{S:k:paths}}}$}\\
    &<\left(v_+-\frac{\eta}{6}\right)\,L_k^2&\mbox{using $k>\ucSk{S:k:paths}+1$ and \eqref{S:e:condition_supp}}.
\end{align*}
See Figure \ref{S:f:J_y} for an illustration of the above bounds. In the end, we do have \eqref{S:e:inclusionFG}, which is true for any $y\in I_{L_k^2}$, so
\begin{align*}
    p_{L_k^2}\left(v_+-\frac{\eta}{6}\right)=\PP\left(\exists\,y\in I_{L_k^2},\,V_{L_k^2}^y\geqslant v_+-\frac{\eta}{6}\right)\leqslant \PP(\hat{\mathcal{A}}_k^c)+\PP(\hat{\mathcal{F}}_k^c)+\PP(\hat{\mathcal{G}}_k^c)+\PP(\hat{\mathcal{H}}_k^c)\xrightarrow[k\to\infty]{}0,
\end{align*}
using \eqref{S:e:proba_A_0}, \eqref{S:e:proba_F_0}, \eqref{S:e:proba_G_0} and \eqref{S:e:proba_H_0}. Therefore $\liminf_{H\to\infty} p_H(v_+-\eta/6)=0,$ where $v_+-\eta/6<v_+$. This contradicts the definition of $v_+$, therefore, $v_-=v_+$.

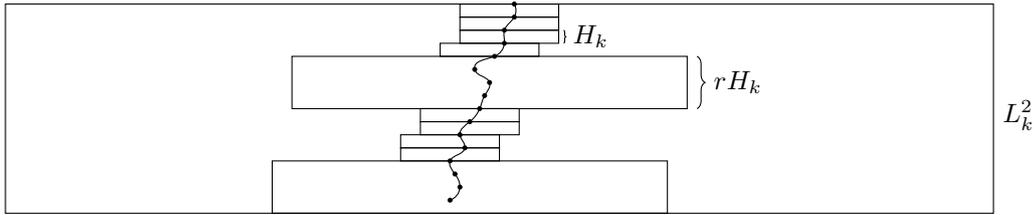
\begin{figure}
    \centering
      \begin{tikzpicture}[use Hobby shortcut,scale=1.3]
        \draw (0, 0) rectangle (10, 2.1333);
        \draw (4.50, 0.1333) .. (4.60, 0.2666) .. (4.55, 0.4) .. (4.50, 0.5333) ..          (4.65, 0.6666) .. (4.60, 0.8) .. (4.70, 0.9333) .. (4.80, 1.0666) .. (4.85, 1.2) .. (4.90, 1.3333) .. (4.75, 1.4666) .. (4.95, 1.6) .. (5.05, 1.7333) .. (5.05, 1.8666) .. (5.15, 2) .. (5.15, 2.1333);
        \draw[fill] (4.50, 0.1333) circle (0.02);
        \draw[fill] (4.60, 0.2666) circle (0.02);
        \draw[fill] (4.55, 0.4) circle (0.02);
        \draw[fill] (4.50, 0.5333) circle (0.02);
        \draw[fill] (4.65, 0.6666) circle (0.02);
        \draw[fill] (4.60, 0.8) circle (0.02);
        \draw[fill] (4.70, 0.9333) circle (0.02);
        \draw[fill] (4.80, 1.0666) circle (0.02);
        \draw[fill] (4.85, 1.2) circle (0.02);
        \draw[fill] (4.90, 1.3333) circle (0.02);
        \draw[fill] (4.75, 1.4666) circle (0.02);
        \draw[fill] (4.95, 1.6) circle (0.02);
        \draw[fill] (5.05, 1.7333) circle (0.02);
        \draw[fill] (5.05, 1.8666) circle (0.02);
        \draw[fill] (5.15, 2) circle (0.02);
        \draw[fill] (5.15, 2.1333) circle (0.02);
        \draw (2.70, 0) rectangle (6.70, 0.5333); 
        \draw (4.00, 0.5333) rectangle (5.00, 0.6666);
        \draw (4.00, 0.6666) rectangle (5.00, 0.8);
        \draw (4.20, 0.8) rectangle (5.20, 0.9333);
        \draw (4.20, 0.9333) rectangle (5.20, 1.0666);
        \draw (2.90, 1.0666) rectangle (6.90, 1.6); 
        \draw (4.40, 1.6) rectangle (5.40, 1.7333);
        \draw (4.60, 1.7333) rectangle (5.60, 1.8666);
        \draw (4.60, 1.8666) rectangle (5.60, 2);
        \draw (4.60, 2) rectangle (5.60, 2.1333);
        \draw[right] (10, 1) node {$L_k^2$};
        \draw [decorate,decoration={brace,amplitude=3pt}] (7, 1.6) -- (7, 1.0666) node [black, midway, right] {$\;rH_k$};
        \draw [decorate,decoration={brace,amplitude=.8pt}] (5.65, 1.8666) -- (5.65, 1.7333) node [black, midway, right] {$H_k$};
      \end{tikzpicture}
    \caption{The final bound in the proof of Lemma \ref{S:l:v_-=v_+}. The large boxes correspond to the displacements for $i\in J^y_k$ and the specific case $i=0$, and the small boxes for the rest.}
    \label{S:f:J_y}
  \end{figure}

{\color{black}

\section{Towards a complete LLN}\label{S:s:complete_LLN}
The next step of our work would be to prove a LLN for the random walk defined in Section \ref{S:ss:random_walk}, as is expressed in the following conjecture.

\begin{conjecture}[LLN]\label{S:c:LLN_temp}
There exists $\Delta\in\RR^2$ such that $$\PP\text{-almost surely},\;\;\; \frac{Z_n}{n}\xrightarrow[n\to\infty]{} \Delta.$$
\end{conjecture}

With Theorem \ref{S:t:direction}, we have a result that is weaker - albeit very interesting both in itself and in the methods used to prove it. Indeed, when assuming conjecture \ref{S:c:LLN_temp}, we immediately get Theorem \ref{S:t:direction} with $\chi=\frac{\Delta}{\Vert\Delta\Vert}$. However, going from Theorem \ref{S:t:direction} to an actual LLN is not trivial at all. It is actually sufficient to show that $\frac{\tau_n}{n}$ converges almost surely to derive the LLN from Theorem \ref{S:t:LLN_spatial}. In other words, what we are missing at this point is the understanding of the temporal behavior of $Z$.

This is \textit{a priori} a hard question, because the environment from the point of view of the random walk may not behave very nicely under our assumptions. We used renormalization methods to get around this issue, but it is unclear to which events describing the temporal behaviour of $Z$ we could apply a renormalization method.

\section{Applications}\label{S:s:applications}
In this section, we give examples of environments that satisfy the assumptions introduced in Section \ref{S:ss:environment}. These are taken from classical 1D dynamic or 2D static models for which we can control the vertical dependencies.


\subsection{One-dimensional dynamic environments}
In \cite{BHT} are presented several models of one-dimensional dynamic environments that have at most polynomial time correlations, with time evolving in $\RR$. More precisely, let $S$ be a topological space (usually, it is a subset of $\NN$). A continuous-time one-dimensional dynamic environment is a random variable on some probability space $(\Omega,\mathcal{T},\sP)$ taking values in $\mathcal{D}(\RR,S^\ZZ)$, the space of càdlàg functions from $\RR$ to $S^\ZZ$ endowed with the subsequent Borel $\sigma$-algebra. The state of environment $\eta$ at time $t$ and site $y$ is described by $\eta_t(y)$. We can define a two-dimensional (static) environment $\omega$ as a process on the same probability space by defining, for $x=(y,n)\in\ZZ^2$, $\omega(x)=\eta_n(y).$ If $\eta$ is translation-invariant and mixes polynomially in the vertical direction, then the same can be said of $\omega.$

Examples of environments satisfying such assumptions are given in \cite{BHT}: the contact process, Markov processes with a positive spectral gap, the East model and independent renewal chains (after adapting their definition so that it is valid for negative times too).

\subsection{Boolean percolation}
In \cite{ATT}, the authors show a mixing property for the Boolean percolation process in $\RR^2$, a model first introduced in \cite{Gil}. Here is a brief account of what this model consists of and how it can be used in the framework of this paper.

Heuristically, Boolean percolation can be defined using a Poisson point process of intensity $\lambda>0$ in $\RR^2$, and allocating independently to each point in this point process a ball of random radius, sampled from a common distribution $\sigma$ in $\RR_+$. One way to make this more rigorous is that chosen in \cite{ATT}.



For a subset $\eta\subseteq \RR^2\times \R_+$, let $$\mathcal{O}(\eta)=\bigcup_{(x,z)\in \eta} B(x,z),$$
where $B(x,z)$ is the Euclidean open ball of center $x$ and radius $z$. 

Let $\lambda>0$ and $\sigma$ be a probability measure on $(\RR_+,\mathcal{B}(\RR_+))$. We assume that $\sigma$ satisfies the following moment condition: there exists $\alpha>0$ such that
\ncS{S:c:Boolean2}
\begin{align}\label{S:e:moment_condition_Boolean}
    \int_0^\infty z^{2+\alpha}\,\dd \sigma(z)=\ucS{S:c:Boolean2}<\infty.
\end{align}
This assumption implies, using Markov's inequality, that the radii of our Boolean percolation have tails that decrease with a polynomial rate of exponent $\alpha+2$.

Let $\eta$ be a Poisson point process in $\RR^2\times \RR_+$ with intensity $\lambda\, \dd x \otimes \dd\sigma(z)$, where $\dd x$ is the Lebesgue measure on $\RR^2$, defined on a probability space $(\Omega,\mathcal{T},\sP).$

For every site $x\in\ZZ^2$, we let 
$$\omega(x)=\left\{\begin{array}{ll}1&\mbox{if $x\in\mathcal{O}(\eta)$;}\\ 0&\mbox{otherwise.}\end{array}\right.$$

Let us now state a mixing property for this environment. Proposition \ref{S:p:mixing_Boolean} gives a stronger property than the mixing property we want to get, using translation invariance, the right choice of $\kappa$ and provided that $\alpha$ is large enough. For $r>0$, let $B^\infty(r)=[-r,r]^2\cap\ZZ^2$.

\ncS{S:c:Boolean}
\begin{proposition}[\cite{ATT}, Proposition 2.2]\label{S:p:mixing_Boolean}
    Assume \eqref{S:e:moment_condition_Boolean}. For every $\kappa>0$, there exists $\ucS{S:c:Boolean}=\ucS{S:c:Boolean}(\lambda,\sigma,\kappa)>0$ such that for all $r\geqslant 1$ and for all pairs of functions $f_1,f_2:\{0,1\}^{\ZZ^2}\rightarrow [-1,1]$ such that $f_1(\omega)$ is $\sigma(\omega\vert_{B^\infty(r)})$-measurable and $f_2(\omega)$ is $\sigma(\omega\vert_{B^\infty(r(1+\kappa))^c})$-measurable, we have
    $$\E[f_1(\omega)f_2(\omega)]\leqslant \E[f_1(\omega)]\E[f_2(\omega)]+ \ucS{S:c:Boolean}\, r^{-\alpha}.$$
\end{proposition}

\subsection{Gaussian fields} \ncS{S:c:dec_Gaussian} \ncS{S:c:dec_Gaussian2}
In \cite{BHKT} (Section 6.1), the authors introduce an environment on $\ZZ^2$ using Gaussian fields. This environment satisfies a mixing assumption that is stronger than ours. Here is a brief account of what we need from \cite{BHKT} in our framework.

Let $q:\ZZ^2\rightarrow \RR_+$ a non-zero function such that
\begin{align}\label{S:e:hyp1Gaussian}
\forall (x_1,x_2)\in\ZZ^2,\;q(x_1,x_2)=q(-x_1,x_2).
\end{align}
We also assume that there exists $\lambda>2$ and $\ucS{S:c:dec_Gaussian}>0$ such that
\begin{align}\label{S:e:hyp2Gaussian}
\forall x\in\ZZ^2\setminus\{0\},\;q(x)\leqslant \ucS{S:c:dec_Gaussian} |x|^{-\lambda}.
\end{align}
We also consider a family $(W_x)_{x\in\ZZ^2}$ of i.i.d. standard normal random variables and we define the Gaussian field $(g_x)_{x\in\ZZ^2}$ by setting
$$g_x=\sum_{y\in\ZZ^2} q(x-y)\,W_y.$$
The environment we are interested in is given for $x\in\ZZ^2$, by $\omega(x)=\mathrm{sign}(g_x)\in\{\pm 1\}$. By construction $\omega$ is translation-invariant. It remains to check that it satisfies our mixing assumption. The authors of \cite{BHKT} show the stronger property that follows.
\begin{proposition}[\cite{BHKT}, Lemma 6.2]
    Recall \eqref{S:e:hyp1Gaussian} and \eqref{S:e:hyp2Gaussian}. There exists $\ucS{S:c:dec_Gaussian2}>0$ such that for every integer $h\geqslant 2$ and every box $C=[a,a+a']\times [b,b+b']\subseteq \ZZ^2$ with $a',b'\geqslant 1$, there exists a coupling between $\omega$ and a field $\omega^{C,h}$ such that
    $$\sP(\omega\neq \omega^{C,h})\leqslant \ucS{S:c:dec_Gaussian2} (a'b'+(a'+b')h+h^2)\,h^{-\lambda+3/2}$$
    and, if $A\subseteq C$ and $B\subseteq \ZZ^2$ satisfy $d(A,B)>h$, then $\omega^{C,h}\vert_A$ and $\omega^{C,h}\vert_B$ are independent.
\end{proposition}

This mixing property implies that if $B$ and $B'$ are two boxes as in Definition \ref{d:mixing}, and if $f_1$ and $f_2$ are two measurable functions on $\{\pm 1\}^{\ZZ^2}$ such that $f_1(\omega)$ is $\sigma(\omega\vert_{B\cap \ZZ^2})$-measurable and $f_2(\omega)$ is $\sigma(\omega\vert_{B'\cap\ZZ^2})$-measurable,
$$\E[f_1(\omega)f_2(\omega)]\leqslant \E[f_1(\omega)]\E[f_2(\omega)]+ch^{-\alpha}$$
for some constant $c=c(a,b)>0$, where $\alpha=-(2-\lambda+3/2).$ Therefore Assumption \ref{S:a:mixing_env} is satisfied provided that $\lambda>31/2$.

\subsection{Factors of i.i.d. with light-tail finite radii}
Let $Y=(Y_x)_{x\in\ZZ^2}$ be a family of i.i.d. random variables in $[0,1]$ on a probability space $(\Omega,\mathcal{T},\sP)$. Let $\omega:\ZZ^2\rightarrow \{0,1\}$ be a random variable on the same space. We say that $\omega$ is a factor of $Y$ with finite radius if there exist two measurable functions $\phi:[0,1]^{\ZZ^2}\rightarrow \{0,1\}$ and $\rho:[0,1]^{\ZZ^2}\rightarrow \RR_+$ such that:
\begin{itemize}[itemsep=0pt]
    \item For all $x\in\ZZ^2$, $\omega(x)=\phi(\theta^x Y)$, where $\theta^x \mathbf{y}=(y_{x+v})_{v\in\ZZ^2}$ for every $\mathbf{y}=(y_v)\in [0,1]^{\ZZ^2}$;
    \item For $\sP$-almost all $\mathbf{y},\mathbf{y'}\in [0,1]^{\ZZ^2}$ that coincide outside of $B(o,\rho(\mathbf{y}))$, $\phi(\mathbf{y})$ and $\phi(\mathbf{y}')$ are equal at $o$.
\end{itemize}

This implies that we only need to look at $Y$ in a ball of radius $\rho(Y)$ around a site $x\in\ZZ^2$ to determine $\omega(x)$. Random variable $R=\rho(Y)$ is called the radius of $\omega.$

\ncS{S:c:factors}\ncS{S:c:mixing_factors}
Note that $\omega$ is translation-invariant by construction. In order to show a mixing property for $\omega$, we need to make an additional assumption on the radius: we assume that there exist $\alpha>0$ and $\ucS{S:c:factors}>0$ such that for all $r>0$,
\begin{align}\label{S:e:assumption_radius}
    \sP(R>r)\leqslant \ucS{S:c:factors}\,r^{-\alpha}.
\end{align}
\begin{proposition}\label{S:p:mixing_factors}
There exists $\ucS{S:c:mixing_factors}>0$ such that for every $h>0,$ for every pair of boxes $B_1$ and $B_2$ such that $\mathrm{sep}(B_1,B_2)\geqslant h$, for every pair of $[0,1]$-valued functions $f_1$ and $f_2$ on $\{0,1\}^{\ZZ^2}$ such that $f_1(\omega)$ is $\sigma(\omega\vert_ {B_1\cap\ZZ^2})$-measurable and $f_2(\omega)$ is $\sigma(\omega\vert_ {B_2\cap\ZZ^2})$-measurable,
    \begin{align}\label{S:a:mixing_env_factors} \E[f_1(\omega)f_2(\omega)]\leqslant \E[f_1(\omega)]\E[f_2(\omega)]+\ucS{S:c:mixing_factors}\, h^{-\alpha}.
    \end{align}
\end{proposition}

\begin{proof}
Let us define, for $i\in\{1,2\}$, a box $\tilde{B}_i=(B_i+[-h/3,h/3]^2)\cap\ZZ^2$ and a function $\psi_{i}:[0,1]^{\tilde{B}_i}\rightarrow [0,1]^{\ZZ^2}$ by setting, for $\mathbf{y}=(y_v)_{v\in \tilde{B}_i}\in [0,1]^{\tilde{B}_i}$ and $x\in\ZZ^2$, $\psi_{\tilde{B}_i}(y)_x=y_x \mathbf{1}_{\tilde{B}_i}(x)$. Also, we define $g_i=f_i\circ \phi\circ \psi_{\tilde{B}_i}$, which takes arguments in $[0,1]^{\tilde{B}_i}$. The crucial idea is that if $R\leqslant h/3$, we have $f_i(\omega)=g_i(Y_{\tilde{B}_i})$, where $Y_{\tilde{B}_i}=(Y_v)_{v\in\tilde{B}_i}$. Now, remark that $Y_{\tilde{B}_1}$ and $Y_{\tilde{B}_2}$ are independent, since $\tilde{B}_1$ and $\tilde{B}_2$ are disjoint. Therefore,
\begin{align*}
    \E[f_1(\omega)\,f_2(\omega)]
    &\leqslant \E[g_1(Y_{\tilde{B}_1})\,g_2(Y_{\tilde{B}_2})\,\mathbf{1}_{R\leqslant h/3}]+ \sP(R>h/3)\\
    &\leqslant \E[g_1(Y_{\tilde{B}_1})]\,\E[g_2(Y_{\tilde{B}_2})]+\ucS{S:c:factors}\, h^{-\alpha}\\
    &\leqslant (\E[f_1(\omega)]+\ucS{S:c:factors}\, h^{-\alpha})\,(\E[f_2(\omega)]+\ucS{S:c:factors}\, h^{-\alpha})+\ucS{S:c:factors}\, h^{-\alpha}\\
    &\leqslant \E[f_1(\omega)]\,\E[f_2(\omega)]\,+ch^{-\alpha}.
\end{align*}
\end{proof}
Taking a closer look at this proof, we can see that in fact the mixing property holds not only for boxes that are vertically separated, but for any two sets of $\ZZ^2$ with vertical distance at least $h$.

\section*{Acknowledgments} This work could not have been possible without the extensive help of my PhD supervisors Oriane \textsc{Blondel} (ICJ, Villeurbanne, France) and Augusto \textsc{Teixeira} (IMPA, Rio de Janeiro, Brazil), and I would like to take this opportunity to thank them wholeheartedly for their involvement, kindness and patience. This work was supported by a doctoral contract provided by CNRS. Finally, my working in person with Augusto \textsc{Teixeira} in IST (Lisbon, Portugal) was enabled by grants from ICJ and Labex Milyon, and I would like to thank IST for welcoming me away from ICJ.

\bibliographystyle{alpha}
\bibliography{Biblio}

\end{document}